\documentclass[12pt,final]{article}   

\usepackage{amsmath}
\usepackage{amsthm}
\usepackage{amssymb}
\usepackage{amsfonts}
\usepackage{latexsym}               
\usepackage{amsgen}
\usepackage[mathscr]{eucal}
\usepackage{mathrsfs}
\usepackage{enumerate}
\usepackage{bm}
\usepackage{here}

\usepackage{mathptmx} 
\usepackage{color}
\usepackage{fancybox}
\usepackage{graphicx}
\usepackage{cite}

\usepackage{showkeys}

\newtheorem{lem}{Lemma}[section]
\newtheorem{thm}[lem]{Theorem}
\newtheorem{pro}[lem]{Proposition}
\newtheorem{cor}[lem]{Corollary}

\newcommand{\bqn}{\begin{equation}}
\newcommand{\eqn}{\end{equation}}
\newcommand{\beqx}{\begin{equation*}}
\newcommand{\eeqx}{\end{equation*}}
\newcommand{\barr}{\begin{array}}
\newcommand{\earr}{\end{array}}
\newcommand{\beqn}{\begin{eqnarray}}
\newcommand{\eeqn}{\end{eqnarray}}
\newcommand{\beqnx}{\begin{eqnarray*}}
\newcommand{\eeqnx}{\end{eqnarray*}}
\newcommand{\bmt}{\begin{multline}}
\newcommand{\emt}{\end{multline}}

\numberwithin{equation}{section}

\newcommand{\supp}{\operatorname{supp}}

\def\det{\mathop{\rm det}\nolimits}

\newcommand{\sD}{\mathscr{D}}

\newcommand{\sF}{\mathscr{F}}
\newcommand{\sG}{{\mathscr G}}
\newcommand{\sK}{{\mathscr K}}
\newcommand{\sL}{\mathscr{L}}

\newcommand{\sN}{\mathscr{N}}
\newcommand{\sO}{\mathscr{O}}

\newcommand{\sW}{\mathscr{W}}

\newcommand{\cA}{{\mathcal A}}

\newcommand{\cC}{{\mathcal C}}

\newcommand{\cF}{{\mathcal F}}
\newcommand{\cG}{{\mathcal G}}

\newcommand{\cK}{{\mathcal K}}
\newcommand{\cL}{{\mathcal L}}

\newcommand{\cN}{{\mathcal N}}

\newcommand{\cO}{{\mathcal O}}

\newcommand{\cR}{{\mathcal R}}

\newcommand{\bbn}{{\mathbb N}}

\newcommand{\real}{{\mathbb R}}

\newcommand{\al}{\alpha}

\newcommand{\ga}{{\gamma}}

\newcommand{\ve}{\varepsilon}
\newcommand{\de}{\delta}

\newcommand{\la}{\lambda}

\newcommand{\ro}{\rho}

\newcommand{\er}{\eqref}
\newcommand{\lb}{\label}
\newcommand{\qu}{\quad}

\title{Ionized Gas in an Annular Region}

\author{%
{\large\sc Walter A. Strauss${}^1$}
{\normalsize and}
{\large\sc Masahiro Suzuki${}^2$}
}

\date{%
\normalsize
${}^1$%
Department of Mathematics and Lefschetz Center for Dynamical Systems, 
Brown University,   
\\
Providence, RI 02912, USA
\\ [7pt]
${}^2$ 
Department of Computer Science and Engineering, 
Nagoya Institute of Technology,
\\
Gokiso-cho, Showa-ku, Nagoya, 466-8555, Japan
}

\begin{document}

\maketitle

\begin{abstract}
{   We consider a plasma that is created by a high voltage difference $\lambda$, which is known as a Townsend discharge.  We consider it to be confined to the region $\Omega$ between two concentric spheres, two concentric cylinders, or more generally between two star-shaped surfaces.  We first prove that if the plasma is initially relatively dilute, then either it may remain dilute for all time or it may not, depending on a certain parameter $\kappa(\lambda, \Omega)$.  Secondly, we prove that there is a connected one-parameter family of steady states.  This family connects the non-ionized gas to a plasma, either with a sparking voltage $\lambda^*$ or with very high ionization, at least in the cylindrical or spherical cases.  }

\end{abstract}

\begin{description}

\item[{\it Keywords:}]
plasma, 
hyperbolic-parabolic-elliptic system,
steady states, nonlinear stability, bifurcation

\item[{\it 2020 Mathematics Subject Classification:}]
35M33, 
35B32, 
35B35, 
76X05 
\end{description}

\newpage

\section{Introduction}

This paper concerns a model for the ionization of a gas, such as air,  
due to a strong applied electric field.  The high voltage thereby creates a plasma, 
which may induce very hot electrical arcs.  
A century ago Townsend experimented with a pair of parallel plates 
to which he applied a strong voltage that produces cascades of free electrons and ions.  
This phenomenon is called the Townsend discharge or avalanche. 
The minimum voltage for this to occur is called the sparking voltage.  
The collision of gas particles within the plasma is sometimes called the $\alpha$-mechanism. 
For more details of the physical phenomenon, we refer the reader to \cite{Rai}.

In the most classical experiments the gas fills the region between two parallel plates.  
In this paper the gas resides in a  smooth bounded domain $\Omega$ in $\mathbb R^d$, where $d$ is either $2$ or $3$ and where $\Omega:=\Omega_2 \backslash\overline\Omega_1$,  
both $\Omega_{1}$ and  $\Omega_{2}$ being simply connected bounded open sets 
with $\overline\Omega_1 \subset \Omega_{2}$.  
Their boundaries are smooth; one of them is the anode $\cA$ and the other is the cathode $\cC$. 
A more specific geometric assumption is made below.  
The case of $\Omega\subset \mathbb R^2$ is a mathematical way to consider a gas inside a cylinder of base $\Omega$ that is uniform in the direction normal to $\Omega$.
   
Many models have been proposed to describe this ionization phenomenon (see \cite{AB,DL1,DW,DT1,Ku1,Ku2,LRE,Mo1}).  
In 1985 Morrow \cite{Mo1} was probably the first to provide a realistic model of its detailed mechanism. 
His model  consists of continuity equations for the electrons and ions coupled to the Poisson equation 
for the electrostatic potential.  
For simplicity in this paper we consider only electrons and a single species of positive ions.  
We do not consider other mechanisms such as `secondary emission', `attachment' or `recombination'.  
Thus inside the plasma $\Omega$ the model is as follows.  
Let $\rho_i, u_i, \rho_e, u_e$ and$ -\Phi$ denote the the ion density, ion velocity, electron density, electron velocity and electrostatic potential, respectively.  
Of course, $\rho_i$ and $\rho_e$ are nonnegative.  
They satisfy the equations 
\begin{subequations}\lb{Mmodel}
\begin{gather}
\partial_t\rho_i + \nabla\cdot(\rho_i {\bm u}_i)  =  k_e h(|\nabla\Phi|) \rho_e ,
\lb{eqi} \\
\partial_t\rho_e + \nabla\cdot(\rho_e {\bm u}_e)  =  k_e h(|\nabla\Phi|) \rho_e ,
\lb{eqe} \\
\Delta \Phi =\ro_i-\ro_e, 
\lb{eqp} \\
{\bm u}_i := k_i\nabla\Phi, \quad
{\bm u}_e:=-k_e \nabla\Phi-k_e \frac{\nabla \rho_e}{\rho_e}, 
\lb{equ}
\end{gather}
where $h(s) := ase^{-{b}/{s}}$ and $k_e, k_i>0$ are physical constants.  

The first two equations express the transport of ions and electrons, while their 
right sides express the rate per unit volume of ion--electron pairs created 
by the impacts of the electrons.  
Specifically, the coefficient $a\exp(-b |\nabla\Phi |^{-1} )$ is 
the first Townsend ionization coefficient $\al$, 
which was determined empirically, as discussed in 
Section 4.1 of \cite{Rai}.  It can be found explicitly in equation (A1) of \cite{Mo1}, for instance.
The last equations \eqref{equ} express the impulse of the electric field ${\bf E}=\nabla \Phi $ 
on the particles in opposite directions.  
The last term expresses the motion of the electrons due to the gradient of their density, 
while the ions barely move when impacted by electrons.  
Substituting the constitutive velocity relations \er{equ} 
into the continuity equations \er{eqi} and \er{eqe}, 
we observe that the system is of hyperbolic-parabolic-elliptic type. 

Of course we have to specify the boundary conditions.  
The boundary has two disconnected parts: $\partial\Omega=\cC \cup \cA$.  
We prescribe the initial and boundary conditions 
\begin{gather}
  (\ro_i,\ro_e)(0,x)=(\ro_{i0},\ro_{e0})(x),  
    \quad x\in \Omega,
  \lb{i1} \\  
  \ro_i(t,x)=\ro_e(t,x)=\Phi(t,x)=0, \quad t>0, \ x \in \cA,
  \lb{ba} \\
  \ro_e(t,x)=0, \quad \Phi(t,x)=\lambda>0, \quad  t>0, \ x \in \cC,
  \lb{bc}
\end{gather}
\end{subequations}
where $\cA$ and $\cC$ satisfy $\cA \cup \cC = \partial \Omega$, and denote the anode and cathode, respectively.
The cathode is negatively charged, producing a voltage difference $\lambda$.  
The boundary conditions mean that at each instant  
electrons are repelled by the cathode and absorbed by the anode, 
while the ions are repelled from the anode. 
Of course the non-negativity of the mass densities $\ro_{i0}$ and $\ro_{e0}$ is a natural condition. 
For compatibility at the zeroth and first orders, it is required to assume that the initial data satisfy
\begin{subequations}\label{com}
\begin{gather}
\rho_{i0}(x)=\rho_{e0}(x)= \nabla \Phi(0,x) \cdot \nabla \rho_{i0}(x)=0 , \quad x \in \cA,
\lb{coma}\\
\rho_{e0}(x)=0,  \quad x \in \cC,
\lb{comc}   \\ 
\rho_{i0}\ge0, \ \rho_{e0}\ge0, \quad x\in \Omega.  
\lb{comd}
\end{gather}
\end{subequations}

Townsend defined the {\it sparking voltage} as the threshold of voltage at which
gas discharge occurs and continues.  
Mathematically, the sparking voltage $\lambda^*$ is defined (roughly) as the lowest voltage 
for which a certain natural operator $A$ has zero as its lowest eigenvalue.  
See \eqref{kap0} and \eqref{bp1} for the precise definition.    

Morrow's model is physically reliable.
Indeed, the article \cite{DL1} by Degond and Lucquin-Desreux derives the model 
from the general Euler-Maxwell system by scaling assumptions, 
in particular by assuming a very small mass ratio between the electrons and ions. 
In an appropriate limit the Morrow model is obtained 
at the end of their paper in equations (160) and (163), which we have specialized to 
assume constant temperature and no neutral particles. 
We are also ignoring the $\gamma$-mechanism, which refers to the secondary emission 
of electrons caused by the impacts of the ions with the cathode.  

Suzuki and Tani in \cite{ST1} gave the first mathematical analysis of the Morrow model 
with both the $\alpha$ and $\gamma$-mechanisms.
Typical shapes of the cathode and anode 
in physical and numerical experiments are a sphere or a plate.
So they proved the time-local solvability of an initial boundary value problem over 
domains with a pair of boundaries that are plates or spheres.
In another paper \cite{ST2} they did a deeper analysis between two parallel plates 
while ignoring the $\gamma$-mechanism.
They proved that there exists a sparking voltage
at which the trivial solution (with $\rho_i=\rho_e=0$) goes from stable to unstable. 
This fact means that gas discharge (electrical breakdown) can occur and continue 
for a voltage greater than the sparking voltage.
Of course, it is expected that non-trivial solutions bifurcate around the sparking voltage.

In \cite{SS1} we proved the global bifurcation of non-trivial solutions between two parallel plates 
and took advantage of the positivity of the densities $\rho_{i}$ and $\rho_{e}$.  
Both \cite{ST2} and \cite{SS1} only treated parallel plates, for which the stationary states satisfy an ordinary differential equation in a single variable.  In the present paper we consider a region between two star-shaped surfaces, in which all the spatial variables appear.  Here we also introduce the concepts of stability index and sparking voltage in a general context.  Even in the radial case (a spherical or cylindrical shell), 
our linearized operator now has variable coefficients and its nullspace is no longer explicit.   

Townsend introduced a $\ga$-mechanism in order to guarantee that gas discharge occurs beyond a sparking voltage.  
However we showed in \cite{ST2} and \cite{SS1} that 
such a mechanism  is not necessary for gas discharge. 
In our subsequent paper \cite{SS2} we did include a $\gamma$-mechanism and once again we  found sufficient conditions for the global bifurcation of non-trivial solutions between two parallel plates.
A notable point is that there are even some circumstances that do include a $\ga$-mechanism under which a sparking voltage does not exist, which is in contrast to Townsend's original theory.    

So far as we know, there is no rigorous mathematical study 
that analyzes gas discharge in domains other than plates.  
 Of course, there are many physical studies of gas discharges with different configurations of anode and cathode but without any analysis of partial differential equations.
The article \cite{LU1} investigated the gas discharge on cylinders or so-called {\it wire to wire} domains that are regions exterior to a pair of parallel cylinders.  
The analysis is similar in spirit to Townsend's theory.  On the other hand, Durbin and Turyn \cite{DT1} numerically  simulated non-trivial steady states of ionization in coaxial concentric and accentric cylinders 
by using a similar model to Morrow's.   
We do an analysis of this configuration in Sections 4 and 5 for $d=2$.
Morrow \cite{Mo2} also simulated the ionization in concentric spheres ($d=3$) using his model.
We also refer the reader to \cite{TK1}, which reviews the recent progress of numerical simulations for gas ionization.


\section{Main results}\lb{S2}  
\subsection{Geometric assumption and notation}\label{S2.1}
For mathematical convenience, we will decompose the electrostatic potential as
\[
 \Phi=V+\lambda H,
\]
where $V$ solves $\Delta V = \rho_i-\rho_e$ and vanishes on $\partial\Omega$, and 
$H$ is a solution to 
the Laplace equation $\Delta H=0$ with the boundary conditions $H=0$ on $\cA$ and $H=1$ on $\cC$.  
Clearly the harmonic function $H$  depends entirely on the domain $\Omega$ and the choice of $\cA$ and $\cC$.  
The maximum principle implies 
\begin{subequations}\label{H0}
\begin{align}
0<H<1 \quad &\text{in $\Omega$},
\\
\nabla H \cdot \bm{n} <0 \quad &\text{on $\cA$}, 
\\
\nabla H \cdot \bm{n} >0 \quad &\text{on $\cC$},
\end{align}
\end{subequations}
where $\bm{n}$ is the unit normal vector of $\partial \Omega$ pointing away from $\Omega$.  

We make the following fundamental assumption on the shape of $\Omega$, 
 expressed in terms of the harmonic function: 
\begin{align}\label{Hasp1}
|\nabla H| >0 \quad &\text{in $\Omega$}.
\end{align}
This is a condition depending only on $\Omega$. 
In fact, a sufficient condition for \eqref{Hasp1} to hold is that there exists a point $x_{0} \in \Omega_{1}$ 
such that both 
$\Omega_{1}$ and $\Omega_{2}$ are {\it starshaped} with respect to it. 
This fact is proven in Theorem 1 of Section 9.5 in \cite{Ev}. 
(See Figure \ref{fig1}.)

\begin{figure}[H]
\begin{center}
    \includegraphics[width=6.5cm, bb=0 0 785 708]{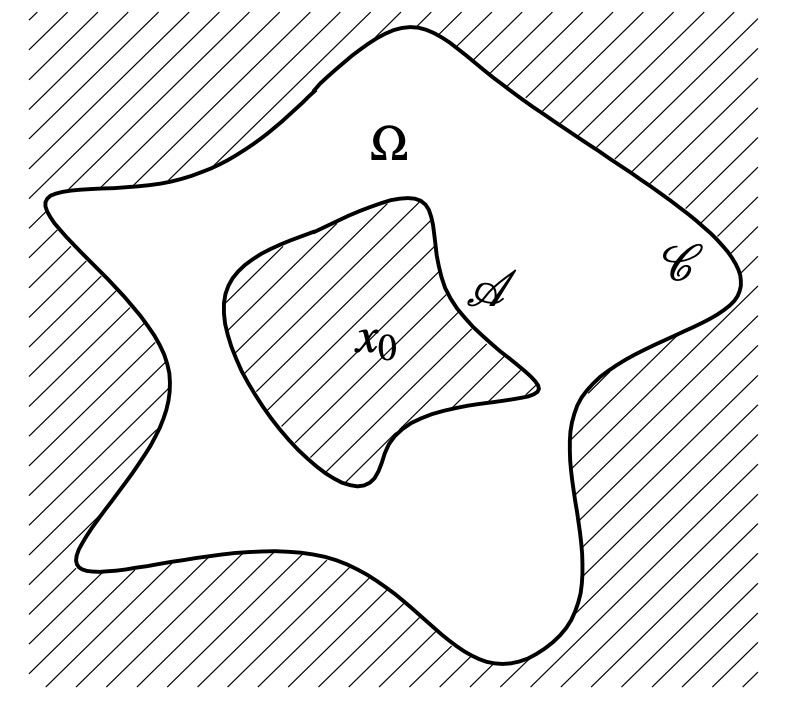}
  \caption{domain $\Omega$}
  \label{fig1} 
\end{center}
\end{figure}

From \eqref{H0}, \eqref{Hasp1}, and the boundary conditions $H=0$ on $\cA$ and $H=1$ on $\cC$, we also see that 
\begin{gather}\label{Hasp3}
\bm{n}=-\frac{\nabla H}{|\nabla H|} \quad \text{on $\cA$}, \quad
\bm{n}=\frac{\nabla H}{|\nabla H|} \quad \text{on $\cC$}.
\end{gather}
It is obvious that the associated stationary problem of \eqref{Mmodel} has the  
trivial stationary solution 
\[
(\rho_{i},\rho_{e},\Phi)=(0,0,\lambda H). 
\]

A simple example is the three-dimensional radial case, the region between two concentric spheres 
of radius $r_a$ for the anode and $r_c$ for the cathode.  In that case, 
$ H(r) = \frac{r_c(r-r_a)} {r(r_c-r_a)}$  where $r=|x|$.

\bigskip 
\noindent {\bf Notation.} 
For a non-negative integer $k$ and a number $\alpha \in (0,1)$, 
$C^{k,\alpha}(\overline{\Omega})$ is the H\"older space.
We sometimes abbreviate $C^{0,\alpha}(\overline{\Omega})$ by $C^{\alpha}(\overline{\Omega})$.
In addition, $C^{k,\alpha}_{0}(\overline{\Omega})$ is a subset of $C^{k,\alpha}(\overline{\Omega})$, whose elements (not their derivatives) vanish on the boundary $\partial \Omega= \cA \cup \cC$.
For $ 1 \leq p \leq \infty$, 
$L^p(\Omega)$ is the Lebesgue space. 
For a non-negative integer $k$,
$H^k(\Omega)$ is the $k$-th order Sobolev space in $L^2$ sense, 
equipped with the norm $\Vert\cdot\Vert_k$. 
Note that $H^0(\Omega) = L^2(\Omega)$ and we define $\Vert\cdot\Vert:=\Vert\cdot\Vert_0$.
The inner product of $L^2(\Omega)$ is denoted by
$\langle f,g \rangle$ for $f,g \in L^2(\Omega)$.
Moreover, $H_0^1(\Omega)$ is closures 
of $C_0^\infty(\Omega)$ with respect to the $H^1$-norm.
In the case that $\Omega$ is an annulus or a spherical shell, all the spaces with the subscript $r$, for instance $C^{k,\alpha}_{r}(\overline{\Omega})$, $L^p_{r}(\Omega)$, and $H^k_{r}(\Omega)$, are subsets whose elements are radial.
We denote by $C^m([0,T];X)$ the space of the $m$-times continuously
differentiable functions on the interval $[0,T]$ with values in a Banach space $X$,
and by $H^m(0,T;X)$ the space of $H^m$--functions 
on $(0,T)$ with values in a Banach space $X$.
Furthermore, we denote by $c$ and $C$ generic positive constants and
by $C_{\alpha}$ a generic positive constant 
depending on a special parameter $\al$.

\subsection{Dilute plasma} 
The stability of the trivial stationary solution $(\rho_{i},\rho_{e},\Phi)=(0,0,\lambda H)$ depends on the following function which involves both the domain $\Omega$ and the rate of pair creation.  
We define the {\it stability index}
\begin{gather}\label{kap0}
\kappa(\lambda) := \inf_{u \in H^{1}_{0}(\Omega), u\not\equiv 0}  
\frac{\displaystyle \int_{\Omega} |\nabla u|^{2} -g(\lambda |\nabla H|) u^{2} dx} 
{\displaystyle  \int_{\Omega} u^{2} dx}  
\end{gather} 
 for $\lambda>0$, where 
\begin{equation}   \label{g:def}
 g(s):= h(s)- \tfrac{s^2}{4} =    ase^{-\frac{b}{s}}-\tfrac{s^{2}}{4}\ .
\end{equation} 
The sign of $\kappa(\lambda)$ depends on the rate of ion-electron pair creation.  For instance, 
$\kappa(\lambda)<0$ if $a$ is large enough.
We also define the self-adjoint linear operator 
\begin{gather}  \label {operatorA}
A = -\Delta - g(\lambda |\nabla H|)  
\end{gather}
 on $L^2(\Omega)$ with domain $H^2(\Omega)\cap H^1_0(\Omega)$.  
Clearly 
$$
\kappa(\lambda) \text{ is the lowest eigenvalue of } A. $$  
It is well-known that 
it is a simple eigenvalue and its eigenfunction $\varphi_1$ does not vanish in $\Omega$.  

It will be convenient to consider {\it $\Phi$ as uniquely determined by $\rho_i$ and $\rho_e$} by solving the Poisson equation \eqref{eqp}.  
We also denote $\cR = (\rho_i,\rho_e)$ and  $\cR_0 = (\rho_{i0},\rho_{e0})$.  
We also define the Hilbert space 
\[  
X = H^2(\Omega)\times H^2(\Omega).  
\]  
Our first theorem asserts that the trivial solution $\rho_i=\rho_e=V=0$
is {\it asymptotically stable} if $\kappa(\lambda)>0$.  

\begin{thm}[Stability]\lb{thm1} 
  For given $\la$, suppose 
the stability index $\kappa(\lambda)$ is positive.
There exists $\ve>0$ such that
if the initial data $\cR_0 \in X$ with  $\|\cR_0\|_{X} < \ve$
satisfies the compatibility conditions \eqref{com}, 
then the problem \er{Mmodel} has a unique solution $\cR \in C([0,\infty);X)$ 
that also satisfies 
\begin{subequations}\lb{reg1}
\begin{align}
& \rho_i\in C^1([0,\infty);H^1(\Omega)), \quad \rho_i\geq0,
\\
& \rho_e\in  L^2(0,\infty;H^3(\Omega))\ \cap\  H^1(0,\infty;H^{1}(\Omega)), \quad \rho_e\geq0.
\end{align}
\end{subequations}
Moreover, 
there is a constant $C$ such that $\sup_{0\le t<\infty} \|\cR(t)\|_X \le C\|\cR_0\|_X$  .  
Furthermore,  $\|\cR(t)\|_X$ converges to zero exponentially fast as $t\to\infty$.  
\end{thm}

Our second theorem asserts that the trivial solution is {\it unstable} if $\kappa(\lambda)<0$.  
\begin{thm}[Instability] \lb{thm2} 
  For given $\la$, suppose  
the stability index $\kappa(\lambda)$ is negative.
There exists  $\varepsilon>0$ such that for all $\delta>0$, 
the problem \eqref{Mmodel} with the initial data $\|\mathscr R_0\|_X < \delta$ and $\rho_{e0} \neq 0$ has a unique solution $\mathscr R \in C([0,T];X)$ such that  $\|\mathscr R(T)\|_X \geq  \varepsilon$ for some $T>0$.
\end{thm}

If there are no electrons initially, then there will be no electrons in the future.  
It has been observed physically that if the initial electron density $\rho_{e0}$ vanishes,
then there is no gas discharge.
If in addition there are few ions initially, then the ions will also eventually disappear 
because they will be absorbed by the anode.  This is the content of the following proposition.
\begin{pro}[No electrons]\lb{pro1}
Let $\lambda>0$ but make no assumption on $\kappa(\lambda)$. 
There exists $\ve>0$ such that 
if the initial data $\cR_0 \in X$ satisfy
$\rho_{e0}=0$, $\|\rho_{i0}\|_{H^1} < \ve$, and the compatibility condition \eqref{com},
then the problem \er{Mmodel} has a unique solution $\cR \in C([0,\infty);X)$ that satisfies \er{reg1}.
Furthermore, there exists $T_0>0$ such that 
\begin{equation*}
\cR (t,x)=\bm{0} \qu \text{for} \ \ (t,x) \in [T_0,\infty) \times \Omega. 
\end{equation*}
\end{pro}
 
 \subsection{Sparking voltage}

{\bf Definition. }
 The {\it sparking voltage} is 
\begin{equation}\lb{bp1}
\lambda^{*}:=\inf\{ \lambda>0 \ | \ \kappa(\lambda)=0, \  \kappa'(\lambda)<0 \}.
\end{equation}

As asserted in the stability and instability theorems, 
any solution that is initially small converges to the trivial solution 
as $t$ tends to infinity in case $\lambda < \lambda^{*}$ with $\la$ near $\la^*$, while 
it leaves a neighborhood of the trivial solution  if $\rho_{e0} \neq 0$, $\lambda > \lambda^{*}$ and $\la$ near $\la^*$.  
This means that $\lambda^{*}$ is the threshold of voltage at which
gas discharge occurs and continues. 

\begin{lem}\label{SparkingV1}
If $a$ is sufficiently large and $b$ is sufficiently small, the equation $\kappa (\lambda)=0$ has 
exactly two solutions $\lambda^{*}$ and $\lambda^{\#}$ such that $\lambda^{*}<\lambda^{\#}$, $\kappa' (\lambda^{*})<0$, and $\kappa'(\lambda^{\#})>0$.
\end{lem}

\begin{proof} 
The function $g(s)$ defined in \eqref{g:def}  is either negative for all positive $s$ or it is positive in a finite interval $(s_1,s_2)$ 
in which it has a single local maximum $M$.  This can be seen clearly by looking for the function's zeros.  
If $s$ is a zero of $g(s)$, then $t=\frac{b}{4a}e^t$ where $t=\frac bs$.  The equation for $t$ clearly has at most two zeros.  In fact, it has two zeros that are very far apart 
if the constant $b$ is much smaller than $a$.  
Furthermore, the maximum $M\to\infty$ as $a\to\infty$.  So for large $a$ and small $b$, $g(s)$ 
has a wide and tall positive hump.  
In order to simplify the nonlinear term but retain the hump, it is convenient to define 
a modified function $\tilde g$ as follows.  
  
We set 
\begin{gather*}
\tilde{g}(s):=as-\frac14 s^{2}, \quad 
\tilde{\kappa}(\lambda) := \inf_{u \in H^{1}_{0}(\Omega), u\not\equiv 0}  
\frac{\displaystyle \int_{\Omega} |\nabla u|^{2} -\tilde{g}(\lambda |\nabla H|) u^{2} dx} 
{\displaystyle  \int_{\Omega} u^{2} dx}.
\end{gather*}
It is clear that $\tilde{\kappa} \in C^{1}(\mathbb R_{+})$ and 
\begin{gather*}
\tilde{\kappa}'(\lambda)=-\int_{\Omega} |\nabla H| \tilde{g}'(\lambda |\nabla H|)\tilde{\varphi}_{1}^{2} dx,
\end{gather*}
where $\tilde{\varphi}_{1}>0$ is the unique eigenfunction corresponding to the lowest eigenvalue of 
$-\Delta-\tilde{g}$ for which $\|\tilde{\varphi}_{1}\|_{L^2} =1$.
We note that $\tilde{g}$ is a good approximation of $g$ in the sense that $g$ converges to $\tilde{g}$ as $b \to 0$.

We claim that the lemma follows provided there exist numbers $\zeta_{1}$, $\zeta_{2}$, $\zeta_{3}$, and $\zeta_{4}$ such that
\begin{subequations}\label{zeta1}
\begin{gather}
0<\zeta_{1}<\zeta_{2}<\zeta_{3}<\zeta_{4}<\infty,
\\
\tilde{\kappa}>0 \text{ on } (0,\zeta_{1}], \quad \tilde{\kappa}'<0 \text{ on } (0,\zeta_{2}],
\\
\tilde{\kappa}<0 \text{ on } [\zeta_{2},\zeta_{3}],
\\
\tilde{\kappa}'>0 \text{ on } [\zeta_{3},\infty), \quad \tilde{\kappa}>0 \text{ on } [\zeta_{4},\infty).
\end{gather}
\end{subequations}
This means that $\tilde{\kappa}(\lambda)$ is first positive, then decreasing, then negative, then  increasing and finally positive.  

First we prove the claim.  
We regard $\kappa$ as a function of $(\lambda,b) \in (0,\infty) \times [0,1]$.
Then $\kappa$ and $\partial_{\lambda} \kappa$ are uniformly continuous on any compact subinterval.  
For $b=0$ we have $\kappa(\lambda,0)=\tilde{\kappa}(\lambda)$ and $\partial_{\lambda}  \kappa(\lambda,0)=\tilde{\kappa}'(\lambda)$. 
Since $0<e^{-b/s} \leq 1$, we have $\tilde{\kappa}(\lambda) \leq \kappa(\lambda,b)$.
Therefore 
\begin{gather*}
{\kappa}>0 \text{ on } ((0,\zeta_{1}] \times [0,1]) \cup ([\zeta_{4},\infty) \times [0,1]).
\end{gather*}
Because of \eqref{zeta1} and the uniform continuity of $\kappa$ and $\partial_{\lambda} \kappa$,
there exists $\delta_{1}>0$ such that 
\begin{gather*}
{\kappa}'<0 \text{ on } [\zeta_{1},\zeta_{2}]\times[0,\delta_{1}],
\quad
{\kappa}<0 \text{ on } [\zeta_{2},\zeta_{3}]\times[0,\delta_{1}],
\quad
{\kappa}'>0 \text{ on } [\zeta_{3},\zeta_{4}]\times[0,\delta_{1}].
\end{gather*}
So we infer that the equation $\kappa (\lambda)=0$ has exactly two solutions $\lambda^{*}$ and $\lambda^{\#}$ such that $\lambda^{*}<\lambda^{\#}$, $\kappa' (\lambda^{*})<0$, and $\kappa'(\lambda^{\#})>0$ for $b \ll 1$.  This proves the claim.  

It remains to verify  \eqref{zeta1} in case $a \gg 1$, which we accomplish as follows.    
The domain $\Omega$ is fixed.  By \eqref{Hasp1} and the boundary conditions,  
we have $0<\omega_1:=\inf_{x\in\Omega}|\nabla H(x)|$ and $\omega_2:=\sup_{x\in\Omega}|\nabla H(x)|<\infty$.
The operator $-\Delta$ in $\Omega$ with Dirichlet boundary conditions has a smallest 
eigenvalue $\sigma_1>0$. 
We define the open subset $D:=\{x \in \Omega \, | \, \omega_{1} < |\nabla H| < \min\{3\omega_{1}/2,(\omega_{1}+\omega_{2})/2 \} \}$.  We choose $\phi_{D}$ such that ${\rm \supp} \phi_{D} \subset D$ and $\int_{\Omega} \phi_{D}^{2} dx=1$.   Now we choose $\zeta_{1}>0$ small enough that 
\begin{gather*}
\tilde{\kappa}(\lambda) \geq \sigma_{1} - \tilde{g}(\zeta_{1} \omega_{2}) >0  
\end{gather*}
for $\lambda \in (0,\zeta_{1}]$.  
We set $\zeta_{2}:=3\omega_{2}^{-1}a/2$.  Then 
\begin{gather*}
\tilde{\kappa}'(\lambda)= - \int_{\Omega} |\nabla H| \tilde{g}'(\lambda |\nabla H|)\tilde{\varphi}_{1}^{2} dx<0  
\end{gather*} 
for $\lambda \in (0,\zeta_{2}]$, due to the positivity of  $\tilde{g}'(s)$ in $(0,2a)$.  

Next we set $\zeta_{3}:=5\omega_{1}^{-1}a/2$ and observe that  for $\lambda \in [\zeta_{2},\zeta_{3}]$ we have 
\begin{gather*}
\frac{3}{2}\frac{\omega_{1}}{\omega_{2}}a=\omega_{1} \zeta_{2} \leq \lambda |\nabla H| \leq \frac{3}{2} \omega_{1} \zeta_{3} = \frac{15}{4} a \quad \text{on } {\rm supp} \phi_{D}.  
\end{gather*}  
Since $\frac{3\omega_1}{2\omega_2}a < 2a$  and $\tilde g$ is maximized at $2a$, we have  
\begin{gather*}
 \tilde{g}(\lambda |\nabla H| ) \geq  \min\left\{\tilde{g} \left(\frac{3}{2}\frac{\omega_{1}}{\omega_{2}}a\right), \tilde{g}\left(\frac{15}{4}a\right) \right\} 
\geq c a^{2} \quad \text{on } {\rm supp} \phi_{D},
\end{gather*}
where $c$ is a positive constant independent of $a$.
Choosing $u=\phi_D$ in the definition of $\tilde{\kappa}$ and letting $a \gg 1$, we deduce that for any $\lambda \in [\zeta_{2},\zeta_{3}]$,
\begin{gather*}
\tilde{\kappa}(\lambda)
\leq \int_\Omega \{|\nabla\phi_D|^2 - \tilde g(\lambda |\nabla H|)\ \phi_D^2\} \, dx  
\leq  \int_\Omega |\nabla\phi_D|^2 dx - ca^{2}<0.
\end{gather*}
It also follows from $\tilde{g}'(s)<0$ for $s\in (2a,\infty)$ that for $\lambda \in [\zeta_{3},\infty)$,
\begin{gather*}
\tilde{\kappa}'(\lambda)=-\int_{\Omega} |\nabla H| \tilde{g}'(\lambda |\nabla H|)\tilde{\varphi}_{1}^{2} dx>0.
\end{gather*}
We set $\zeta_{4}:=5\omega_{1}^{-1}a$ and then see from $\tilde{g}(s)<0$ for $s \in (4a,\infty)$ that for $\lambda \in [\zeta_{4},\infty)$,
\begin{gather*}
\tilde{\kappa}(\lambda) \geq \sigma_{1} >0.
\end{gather*}
Thus we can find the desired numbers $\zeta_{j}$ in the case $a \gg 1$.
The proof is complete.  
\end{proof}

\subsection{Stationary solutions}
Assume now that the sparking voltage $\lambda^*>0$ exists.  
Regarding the voltage $\lambda$ on the cathode $\cC$ as a bifurcation parameter, 
we can expect 
that a one-parameter family of non-trivial stationary solutions may arise from 
the point $(\lambda,\rho_{i},\rho_{e},\Phi)=(\lambda^{*},0,0,\lambda^{*}H)$.
The result on  local bifurcation is summarized in Theorem \ref{thm3}. 
The main difficulty is that, although bifurcation theory is usually applicable only to elliptic systems, our system is of mixed hyperbolic-elliptic type.

\begin{thm}[Local Bifurcation]\lb{thm3}
Let the positive number $\lambda^*$ be defined by \er{bp1}. 
Then there exists a small number $\eta>0$  
such that for each $|s| < \eta$,
there exists $(\Lambda(s),S_{i }(s), S_{e}(s), W(s) ) \in Z$, 
where $Z:=\real\times C^{1,\alpha}(\overline{\Omega}) \times C^{2,\alpha}_{0}(\overline{\Omega}) \times C^{3,\alpha}_{0}(\overline{\Omega})$.    It is uniformly bounded in $s$ with values in $Z$.  
For each $s$ it provides a stationary solution   
\[  
(\rho_i,\rho_e,\Phi)(s)=(s\varphi_i^{*} +s^{2}S_{i}(s) ,s\varphi_e^{*} +s^{2}S_{e}(s), 
s\varphi_{v}^{*}+s^{2}W(s)+(\lambda^{*}+s\Lambda(s)) H),
\]
where $\varphi_i^{*} \in C^{1,\alpha}(\overline{\Omega})$ 
is positive in $\Omega$,   $\varphi_i^{*} \big|_\cA =0 $,   
$\varphi_e^{*} \in C^{2,\alpha}_{0}(\overline{\Omega})$ is positive in $\Omega$, and $\varphi_v^{*} \in C^{3,\alpha}_{0}(\overline{\Omega})$.  
\end{thm}

On the other hand, in the radial case we can deduce {\it global bifurcation}.  
By the radial case we mean that $\Omega$ is either an annulus $\{(x,y) \in \mathbb R^{2}  ;\   r_{1}^{2}<r^{2}:=x^{2} + y^{2}<r_{2}^{2} \}$ or a spherical shell $\{(x,y,z) \in \mathbb R^{3}  ;\   r_{1}^{2}<r^{2}:=x^{2} + y^{2} + z^{2}<r_{2}^{2} \}$.  
The advantage of the radial case is that the system has solutions that depend only on the radial coordinate $r$.  When dealing with such solutions, the transport operator, which involves $\nabla \rho_i$, depends only on the single derivative $\partial_r\rho_i$ and not on $\partial_\theta\rho_i$, so that it is effectively an elliptic operator.  
Ellipticity is required for compactness in the global bifurcation proof. 
It remains an open problem to prove global bifurcation if $\Omega$ is not radial.

\begin{thm}[Global Bifurcation]\lb{thm4}
Let $\lambda^*>0$ satisfy \er{bp1} and let $\Omega$ be either an annulus ($d=2$) or a spherical shell ($d=3$).  Then there exists a unique continuous one-parameter family $\sK$ (that is, a curve) 
of radial stationary solutions $(\rho_{i},\rho_{e},\Phi)$ of the problem \eqref{Mmodel}.
Both densities are positive, $\rho_i\in C^1_{r}, \ \rho_e\in C^2_{r}, \ \Phi\in C^3_{r}$, 
the curve begins at the trivial stationary solution $(\rho_{i},\rho_{e},\Phi)=(0,0,\lambda^{*}H)$ with voltage $\lambda^{*}$ 
and it  ``ends" with one of the following two alternatives: 

Either (i) the density $\rho_i + \rho_e$ becomes unbounded along $\sK$, 

Or (ii) the curve ends at a different trivial stationary solution with some voltage $\lambda^{\#} \neq \lambda^{*}$.
\end{thm}

\noindent
Typical bifurcation diagrams are drawn in Figures \ref{fig2} and \ref{fig3}. 
More details are given in Theorem \ref{mainthm}.

\begin{minipage}{7.5cm}
\begin{figure}[H]
\begin{center}
    \includegraphics[width=7cm, bb=0 0 880 543]{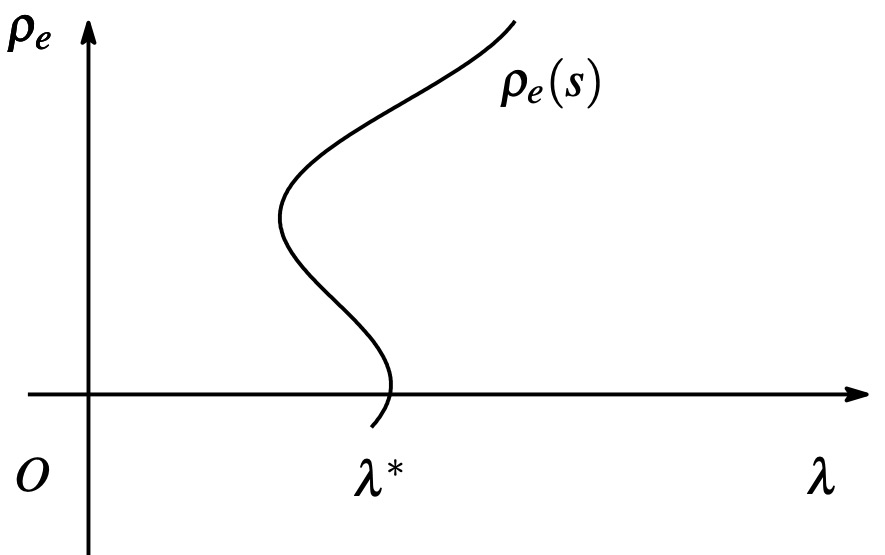}
  \caption{alternative (i)}
  \label{fig2} 
\end{center}
\end{figure}
\end{minipage}
\begin{minipage}{7.5cm}
\vspace{5mm}
\begin{figure}[H]
\begin{center}
    \includegraphics[width=7cm, bb=0 0 863 466]{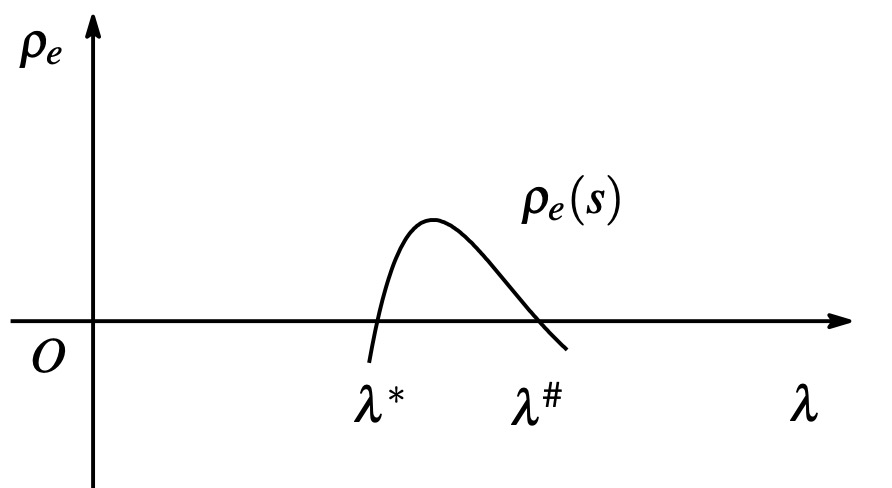}
  \caption{alternative (ii)}
  \label{fig3} 
\end{center}
\end{figure}
\end{minipage}

This paper is organized as follows.  
In Section 3 we first make a simple change of variables and some preliminary estimates.  
The assumption \eqref{Hasp1} on $\Omega$ plays a key role.  
Then we use the energy method to prove Theorem \ref{thm1} (stability).  
Next we prove Theorem \ref{thm2} with the stability index being negative, 
by making use of the lowest eigenfunction of $A$, which leads to a growing mode.  
For the case of no electrons, we prove by means of the energy method that eventually there are also no ions either (Proposition \ref{pro1}).   

In Section 4 we prove the existence of a local bifurcation curve $\sK$ of stationary solutions.  
The standard proof fails because solutions of the hyperbolic equation for the ion density do not gain derivatives.  
In Section 5 we prove that $\sK$ extends to a global curve.   
It is required to prove Fredholm and compactness conditions.  
Because the equation for the ion density is not elliptic, the proof of these conditions is restricted to the radial case, that is, a circular cylinder or a spherical shell. 
Finally we make explicit use of the positivity of the densities in order to show that either $\sK$ 
goes from the trivial solution at the sparking voltage $\lambda^*$ to that at another voltage $\lambda^\#$ or else the densities become unbounded along $\sK$.


\section{Stability analysis of the trivial solution}\lb{S3}
This section provides the stability analysis of the trivial solution $(\rho_{i},\rho_{e},\Phi)=(0,0,\lambda H)$.
It is convenient to rewrite the initial--boundary value problem \er{Mmodel} 
in terms of the 
modified functions  
\begin{gather}\label{newf0}
R_i:=\ro_i e^{- pH}, \qu 
R_e:=\ro_e e^{\frac{\lambda}{2}H}, \qu 
V=\Phi-\lambda H,
\end{gather}
where $p > 0$ is a fixed constant to be determined in the proof of Lemma \ref{apriori1}.
We recall that $h(s):=ase^{-{b}/{s}}$ and $g(s):=h(s)-{s^{2}}/{4}$.
As a result, we can rewrite the original problem as
\begin{subequations}\lb{r0}
\begin{gather}
\partial_tR_i
+k_i \nabla (V+\lambda H ) \cdot \nabla R_i 
+ p k_i \lambda |\nabla H|^{2} R_i
={k_e}h(|\lambda \nabla H| )
e^{{-pH}-\frac{\lambda}{2}H}R_e
+k_if_i[R_i,R_e,V], 
\lb{re1} \\
\partial_tR_e-k_e\Delta R_e -k_eg(|\lambda \nabla H|)R_e=k_ef_e[R_e,V],
\lb{re2} \\
\Delta V =e^{pH}R_i-e^{-\frac{\lambda}{2}H}R_e
\lb{re3}
\end{gather}
with the initial and boundary conditions 
\begin{gather}
(R_i,R_e)(0,x)=(R_{i0},R_{e0})(x),  \qu 
 R_{i0}(x) \geq 0, \ R_{e0}(x) \geq 0, 
\lb{ri1} \\
  R_{i}(t,x)=R_e(t,x)=V(t,x)=0, \quad  t>0, \ x \in \cA,
\lb{rba} \\
  R_e(t,x)=V(t,x)=0, \quad  t>0, \ x \in \cC,
\lb{rbc}
\end{gather}
\end{subequations}
where the nonlinear terms $f_i$ and $f_e$ are defined as
\begin{align*}
f_i[R_i,R_e,V]:=&- (\Delta V+p \nabla V \cdot \nabla H) R_i
-\frac{k_e}{k_i}\left\{ h(|\lambda \nabla H|)- h(|\nabla V + \lambda \nabla H|) \right\} e^{-pH-\frac{\lambda}{2}H}R_e,
\\
f_e[R_e,V]:=&\nabla V \cdot \nabla R_e
-\frac{\lambda}{2} R_e\nabla V \cdot \nabla H+R_e\Delta V
-\left\{ h(|\lambda \nabla H|)- h(|\nabla V + \lambda \nabla H|) \right\} R_e.
\end{align*}
The compatibility condition \eqref{com} can be rewritten as follows:
\begin{subequations}\label{com2}
\begin{gather}
R_{i0}(x)=R_{e0}(x)= (\nabla V(0,x)+\lambda \nabla H (x)) \cdot \nabla R_{i0}(x)=0 , \quad x \in \cA,
\\
R_{e0}(x)=0,  \quad x \in \cC,
 \\ 
R_{i0}\ge0, \ R_{e0}\ge0, \quad x\in \Omega.  
\end{gather}
\end{subequations}
The {\it trivial stationary solution} is  
\[
(R_i,R_e,V)=(0,0,0).
\]

The advantage of using the new unknown functions $R_i$ and $R_e$
lies in the following two facts.
The first one is that the rewritten hyperbolic equation has 
the dissipative term $pk_i \lambda |\nabla H|^{2} R_i$ thanks to the assumption \eqref{Hasp1},
even though the original hyperbolic equation does not have 
any dissipative structure.
Secondly, the linear part of the rewritten parabolic equation
is self-adjoint in $H^{1}_{0}(\Omega)$.
These two facts play important roles in the proofs 
of both the nonlinear stability and instability of the trivial stationary solution.  
Note that the equations are linear in $(R_i,R_e)$ for each $V$.

Now we make a series of straightforward observations. 
For notational convenience, we define 
\[
 N(T):=\sup_{0 \leq t \leq T} \|\cR(t)\|_X = \sup_{0 \leq t \leq T}(\|R_i(t)\|_2+\|R_e(t)\|_2).
\] 
By elliptic theory from \er{re3}, $V$ is easily estimated as 
\begin{gather}
 \|V(t)\|_{2+k} \lesssim_{p} \|R_i(t)\|_k+\|R_e(t)\|_k
\quad \text{for} \ \ t\geq 0, \ \  k=0,1,2,
\lb{ell1}\\
 \|\partial_{t} V(t)\|_{2} \lesssim_{p} \|\partial_{t} R_i(t)\|+\|\partial_{t} R_e(t)\|
\quad \text{for} \ \ t\geq 0.
\lb{ell2}
\end{gather}
The notation $\lesssim_p$ means that the constant after $\le$ may depend on $p$, 
but not on $T, V, R_e, R_i,$ etc.  The constant $p$ was introduced in \eqref{newf0}. 
Therefore, assuming that $N(T)$ is sufficiently small, and 
 due to \eqref{H0}, \eqref{Hasp1}, and \eqref{ell1} and the boundedness of $\Omega$, 
there is a constant $c_{1}>0$ depending on $\lambda$ such that
\begin{align}
(\nabla V + \lambda \nabla H) \cdot \bm{n} \leq -c_{1} \quad &\text{on $\cA$}, 
\label{chara2}\\
(\nabla V + \lambda \nabla H) \cdot \bm{n} \geq c_{1} \quad &\text{on $\cC$},
\label{chara1}\\
|\nabla H| \geq c_{1}  \quad &\text{in $\Omega$}.  
\label{Hasp2}
\end{align}
Assuming $N(T)<1$, 
we can pointwise estimate both nonlinear terms $f_i$ and $f_e$ in \er{r0} as 
\begin{gather}
|f_i[R_i,R_e,V]|  \lesssim_{p} (|\nabla V| + |\Delta V| ) (| R_i|+|R_e|) \lesssim_{p} N(T) (| R_i|+|R_e|),
\lb{NL1} \\
|f_e[R_e,V]| \lesssim (|\nabla V| + |\Delta V| )(|R_e|+|\nabla R_e|) \lesssim_{p} N(T)(|R_e|+|\nabla R_e|).
\lb{NL2}  
\end{gather} 
The derivatives of $f_i$ and $f_e$ are also estimated as
\begin{align}
 |\partial_{t}f_i[R_{i},R_e,V]| 
&\lesssim_{p}  N(T)(|\partial_{t} R_i| + |\partial_{t} R_e| + |\partial_{t}\nabla V| + |\partial_{t}\nabla^{2} V|),
\lb{NL4} \\
|\nabla f_i[R_{i},R_e,V]| 
&\lesssim_{p}  N(T)(|\nabla R_i| + |\nabla R_e| + |\nabla^{2} V| + |\nabla^{3} V|),
\lb{NL6} \\
 |\nabla^{2} f_i[R_{i},R_e,V]| 
& \lesssim_{p} |\nabla^{3} V| |\nabla R_{i}|+N(T) \sum_{j=1}^{2} (|\nabla^{j} R_i| + |\nabla^{j} R_e| + |\nabla^{j} V| + |\nabla^{j+1} V|),
\lb{NL5} \\
 |\partial_{t}f_e[R_e,V]| 
&\lesssim_{p}  |\partial_{t}\nabla V||\nabla R_e| 
+ N(T)(|\partial_{t} R_e|+| \partial_{t} \nabla R_e| + |\partial_{t}\nabla V| + |\partial_{t}\nabla^{2} V|).
\lb{NL3} 
\end{align} 
Using the estimates \eqref{NL1} and \eqref{NL2} together with \eqref{re1}, \eqref{re2}, \eqref{ell1}, and \eqref{ell2}, 
we also see for $N(T)<1$ that
\begin{gather}\label{NT1}
\|\partial_{t} R_i(t)\|+\|\partial_{t} R_e(t)\| +  \|\partial_{t} V(t)\|_{2}  \lesssim_{p} N(T).
\end{gather}

\subsection{Nonlinear stability for $\kappa>0$}\lb{S3.1}

Our goal in this subsection is to prove Theorem \ref{thm1}.  
In particular, the smallness assumption on the initial data is required 
to determine the sign of the characteristics of the hyperbolic equation \er{re1} as \er{chara1} and \er{chara2}.  
This easily leads to the local-in-time solvability of problem \er{r0}.  

\begin{lem}\lb{local-time}
For any $\lambda>0$, there exist an amplitude $\ve>0$ and a time $T>0$ such that
if the initial data $(R_{i0}, R_{e0}) \in H^2(\Omega) \times H^2(\Omega)$ satisfy
$\|R_{i0}\|_{2}+\|R_{e0}\|_{2} < \ve$ and the compatibility condition \eqref{com2}, 
then the problem \er{r0} has a unique solution $(R_i,R_e,V)$ that satisfies
\begin{subequations}\lb{reg2}
\begin{align}
& R_i\in C([0,T];H^2(\Omega))\cap C^1([0,T];H^1(\Omega)), \quad R_i\geq0,
\\
& R_e\in C([0,T];H^{2}(\Omega) )\cap L^2(0,T;H^3(\Omega)) \cap H^1(0,T;H^{1}(\Omega)), \quad R_e\geq0,
\\
& V \in C([0,T];H^{4}(\Omega) ).
\end{align}
\end{subequations}
\end{lem}
\begin{proof}
We omit the easy proof because it is similar to that of Theorem 2.2  and Remark 2.3 in \cite{ST1}. 
\end{proof}

\begin{lem}\lb{apriori1}
Let $\kappa(\lambda)>0$.
Suppose that $(R_i,R_e,V)$ satisfies \er{reg2}  and is a solution to the problem \er{r0}.
There exist $\de>0$ and $\gamma>0$ such that if $N(T) \leq \de$, then the following 
estimate holds for $t\in[0,T]$:
\begin{equation}\lb{APE1}
e^{\ga t} (\|R_{i}(t)\|_2^2+\|R_{e}(t)\|_2^2)
+\int_0^t e^{\ga \tau} (\|R_{i}(\tau)\|_2^2+\|R_{e}(\tau)\|_2^2) \,d\tau    
\lesssim \|R_{i0}\|_2^{2}+\|R_{e0}\|_2^{2}. 
\end{equation}
Recall that $\|\cdot\|_k$ denotes the norm of $H^k(\Omega)$.
We emphasize that the constant in $\lesssim$ is independent of $T$.
\end{lem}

\begin{proof}[Proof of Theorem \ref{thm1}, assuming Lemma \ref{apriori1}]
The global-in-time solution $(R_i,R_e,\!V)$ with \er{reg2} 
is constructed by a standard continuation argument 
using the time local solvability established in Lemma \ref{local-time} 
and the a priori estimate in Lemma \ref{apriori1}. 
Once the global solution is constructed, 
it is obvious that the resulting global solution
satisfies the estimate \eqref{APE1} for all $t\in[0,\infty)$. 
We conclude that the global solution decays
exponentially fast in $X$ as $t$ goes to infinity.  
Of course, it follows from \eqref{ell1} that $\|V(t)\|_4$ also decays exponentially. 
These facts together with \eqref{newf0} immediately ensure Theorem 2.1.
\end{proof}

Now we begin the rather long proof of Lemma \ref{apriori1}.

\begin{proof}[Proof of Lemma \ref{apriori1}.]
{\it First we estimate $R_e$ in $L^2$.} 
We multiply \er{re2} by $2e^{\gamma t} R_e$ for some $\gamma \in (0,1)$, 
integrate it by parts over $[0,t]\times \Omega$, 
and use the boundary conditions \er{rba} and \er{rbc} 
as well as the estimate \er{NL2} to obtain
\begin{align}
{}&
e^{\gamma t} \int_\Omega {R}_e^2 \, dx
+2k_e\int_0^t \!\! \int_\Omega e^{\gamma \tau} \left(|\nabla R_e|^2-g(|\lambda \nabla H|){R}_e^2\right)\, dxd\tau
\notag \\
&= \int_{\Omega} {R}_{e0}^2 \, dx
+\gamma\int_0^t \!\! \int_{\Omega} e^{\gamma \tau} R_e^2 \, dxd\tau
+2k_e\int_0^t \!\! \int_{\Omega} e^{\gamma \tau} f_eR_e \, dxd\tau
\notag \\
&\leq \|R_{e0}\|^2
+(\gamma +{C_{p}N(T)})\int_0^t e^{\gamma \tau} \|R_e(\tau)\|^2_1 \, d\tau,
\lb{apes1}
\end{align}
where the constant $C_{p}>0$ depends on $p$.
Because $\kappa(\lambda)>0$,  
there exists a constant $c_{2}>0$ such that
\begin{gather}
\int_{\Omega} |\nabla u|^{2} -g(|\lambda \nabla H|) u^{2} dx \geq c_{2} \| u\|_{1}^{2}  \quad \text{for} \ \ u \in H^1_0(\Omega).
\lb{poin} 
\end{gather}
Applying this inequality \er{poin} to the second term on the left side of \er{apes1} 
and then taking $\gamma>0$ and $N(T)$ sufficiently small, we arrive at
\begin{equation}\lb{apes2}
e^{\gamma t}\|R_e(t)\|^2
+\int_0^t e^{\gamma \tau} \|R_e(\tau)\|_1^2 \,d\tau
\lesssim \|R_{0e}\|^2.
\end{equation}
Here we have taken $\gamma$ independent of $p$.
Next we multiply \er{re2} by $2e^{\gamma t} \partial_{t} R_e$, 
integrate it by parts over $[0,t]\times \Omega$, 
use the boundary conditions \er{rba} and \er{rbc} 
as well as the estimate \er{NL2} and the Schwarz inequality,
and take $N(T)$ small to obtain
\begin{align}
{}&
e^{\gamma t} k_{e} \int_\Omega  |\nabla {R}_e|^2 \, dx
+2 \int_0^t \!\! \int_\Omega e^{\gamma \tau} |\partial_{t} R_e|^2 \, dxd\tau
\notag \\
&= k_{e}\int_{\Omega} |\nabla {R}_{e0}|^2 \, dx
+\gamma k_{e} \int_0^t \!\! \int_{\Omega} e^{\gamma \tau} |\nabla R_e|^2 \, dxd\tau
+2k_e\int_0^t \!\! \int_{\Omega} e^{\gamma \tau}  (g(|\lambda \nabla H|)R_{e} +f_e) \partial_{t} R_e \, dxd\tau
\notag \\
&\leq \int_0^t \!\! \int_\Omega e^{\gamma \tau} |\partial_{t} R_e|^2 \, dxd\tau
+ C \| \nabla R_{e0}\|^2 
+ C\int_0^t e^{\gamma \tau} \|R_e(\tau)\|^2_1 \, d\tau.
\lb{apes7}
\end{align}
From this estimate and \eqref{apes2}, it follows that
\begin{equation}\lb{apes8}
e^{\gamma t}\|\nabla R_e(t)\|^2
+\int_0^t e^{\gamma \tau} \| \partial_{t} R_e(\tau)\|^2 \,d\tau
\lesssim \|R_{0e}\|_{1}^2.
\end{equation}
Furthermore, we apply $\partial_{t}$ to \er{re2}, multiply the resulting equation by $2e^{\gamma t} \partial_{t} R_e$,
and integrate it by parts over $[0,t]\times \Omega$ to obtain
\begin{align}
{}&
e^{\gamma t} \int_\Omega |\partial_{t} {R}_e|^2 \, dx
+2k_e\int_0^t \!\! \int_\Omega e^{\gamma \tau} \left(|\partial_{t} \nabla R_e|^2-g(|\lambda \nabla H|)|\partial_{t} {R}_e|^2\right)\, dxd\tau
\notag \\
&= \int_{\Omega} |\partial_{t} {R}_{e}|^2(0,x) \, dx
+\gamma\int_0^t \!\! \int_{\Omega} e^{\gamma \tau} |\partial_{t} R_e|^2 \, dxd\tau
+2k_e\int_0^t \!\! \int_{\Omega} e^{\gamma \tau} (\partial_{t} f_e) (\partial_{t} R_e) \, dxd\tau
\notag \\
&\leq \|R_{e0}\|_{2}^2
+(\gamma +{C_{p}N(T)})\int_0^t e^{\gamma \tau} (\| \partial_{t} R_e(\tau)\|_{1}^2 + \| \partial_{t} R_i(\tau)\|^2) \, d\tau,
\lb{apes9}
\end{align}
where we have also used \eqref{re2}, \eqref{ell2}, and \eqref{NL3} as well as the Schwarz and Sobolev inequalities.
Then using \eqref{poin} and taking $\gamma$ and $N(T)$ small, we arrive at 
\begin{equation}\lb{apes2t}
e^{\gamma t}\| \partial_{t} R_e(t)\|^2
+\int_0^t e^{\gamma \tau} \| \partial_{t} R_e(\tau)\|_1^2 \,d\tau
\lesssim \|R_{0e}\|_{2}^2 + \int_0^t e^{\gamma \tau} \| \partial_{t} R_i(\tau)\|^2 \, d\tau.
\end{equation}
Hereafter we fix $\gamma>0$ so that \eqref{apes2}, \eqref{apes8}, and \eqref{apes2t} hold.
We note that $\gamma$ is independent of $p$.

{\it Now let us estimate $R_i$.}  Choose $p>0$ so large that $2p\lambda k_i c_{1}^{2} >\gamma$, 
where  $c_{1}$ is the constant in \eqref{Hasp2}. 
We multiply \er{re1} by $2e^{\gamma t} R_i$, 
integrate by parts over $[0,t]\times \Omega$, 
and use the  boundary condition \er{rba} to obtain
\begin{align}
{}&
e^{\gamma t} \int_\Omega R_i^2 \, dx
+\int_0^t \!\! \int_\Omega e^{\gamma \tau} (2p\lambda k_i|\nabla H|^{2} -\gamma)R_i^2 \, dxd\tau
+k_i\int_0^t \!\! \int_{\cC} e^{\gamma \tau} R_{i}^{2}  \nabla (V+\lambda H)\cdot \bm{n} \, dSd\tau 
\notag \\
&= \int_\Omega R_{i0}^2 \, dx
+ k_{i}\int_0^t \!\! \int_\Omega   e^{\gamma \tau} R_{i}^{2} \Delta V  \, dxd\tau
+2\int_0^t \!\! \int_\Omega  
e^{\gamma \tau}\left(
{k_e}h(|\lambda \nabla H| ) e^{-pH-\frac{\lambda}{2}H}R_e
+ k_if_i \right)R_i \, dxd\tau
\notag \\
&\lesssim_{p} \|R_{i0}\|^2
+( \mu +N(T))\int_0^t e^{\gamma \tau} \|R_i(\tau)\|^2 \, d\tau
+\mu^{-1} \int_0^t e^{\gamma \tau} \|R_e(\tau)\|^2 \, d\tau,
\lb{apes5}
\end{align}
where $\mu$ is a positive constant, and we have used \eqref{ell1} and \er{NL1} as well as the Schwarz and Sobolev inequalities.  
From \er{Hasp2} and $2p\lambda k_i c_{1}^{2} >\gamma$, it is seen that
the coefficient in the second term on the left hand side of \er{apes5} is positive, i.e.
\begin{gather}\label{coeff1}
2p\lambda k_i|\nabla H|^{2} -\gamma>0.
\end{gather} 
Owing to \eqref{chara1}, the third term on the left side of \er{apes5} (the integral over $\cC$) is non-negative and thus can be dropped.
Then letting $\mu$ and $N(T)$ be sufficiently small and using \er{apes2}, we arrive at
\begin{equation}\lb{apes6}
e^{\gamma t}\|R_i(t)\|^2 +\int_0^t e^{\gamma \tau} \|R_i(\tau)\|^2 \,d\tau 
\lesssim_{p} \|R_{i0}\|^2+\|R_{e0}\|^2.
\end{equation}
Next we apply the temporal derivative $\partial_{t}$ to \er{re1},
multiply the resulting	 equation by $2e^{\gamma t} \partial_{t} R_i$, 
and integrate by parts over $[0,t]\times \Omega$ to obtain
\begin{align}
{}&
e^{\gamma t} \int_\Omega |\partial_{t}R_i|^2 \, dx
+\int_0^t \!\! \int_\Omega e^{\gamma \tau} (2p\lambda k_i|\nabla H|^{2}-\gamma) |\partial_{t} R_i|^2 \, dxd\tau
+k_i\int_0^t \!\! \int_{\cC} e^{\gamma \tau} |\partial_{t}R_{i}|^{2}  \nabla (V+\lambda H)\cdot \bm{n} \, dSd\tau 
\notag \\
&= \int_\Omega |\partial_{t} R_{i}|^2|(0,x) \, dx
+k_{i} \int_0^t \!\! \int_\Omega   e^{\gamma \tau} |\partial_{t} R_{i}|^{2} \Delta V  \, dxd\tau
- k_{i} \int_0^t \!\! \int_\Omega  e^{\gamma \tau}  (\partial_{t} \nabla V \cdot \nabla R_{i})  \partial_{t} R_i\, dxd\tau
\notag \\
& \quad +2\int_0^t \!\! \int_\Omega  
e^{\gamma \tau}\left(
{k_e}h(|\lambda \nabla H| ) e^{-pH-\frac{\lambda}{2}H} \partial_{t} R_e
+ k_i\partial_{t} f_i \right) \partial_{t} R_i \, dxd\tau
\notag \\
&\lesssim_{p} \|R_{i0}\|_{1}^2 + \|R_{e0}\|^2
+(\mu +N(T))\int_0^t e^{\gamma\tau} \|\partial_{t} R_i(\tau)\|^2 \, d\tau
+\mu^{-1} \int_0^t e^{\gamma \tau} \|\partial_{t} R_e(\tau)\|^2 \, d\tau,
\lb{apes10}
\end{align}
where we have used \eqref{ell1}, \eqref{ell2}, and \er{NL4} as well as the Schwarz and Sobolev inequalities.  
We note that the second term on the left hand side of \er{apes10} is non-negative due to \eqref{coeff1}.
Owing to \er{chara1},
the third term on the left hand side of \er{apes5} 
is non-negative and thus can be dropped.
Then letting $\mu$ and $N(T)$ be sufficiently small and using \er{apes8}, we see that
\begin{equation}\lb{apes6t}
e^{\gamma t}\|\partial_{t}R_i(t)\|^2 +\int_0^t e^{\gamma \tau} \|\partial_{t}R_i(\tau)\|^2 \,d\tau 
\lesssim_{p} \|R_{i0}\|_{1}^2+\|R_{e0}\|_{1}^2.
\end{equation}

It remains to take the spatial derivatives.  That is, we will prove the estimate
\begin{gather}
e^{\gamma t}\|\nabla^{k}R_i(t)\|^2 +\int_0^t e^{\gamma \tau} \|\nabla^{k}R_i(\tau)\|^2 \,d\tau 
\lesssim_{p} \|R_{i0}\|_{2}^2+\|R_{e0}\|_{2}^2 + \int_0^t e^{\gamma \tau} \|R_e(\tau)\|^2_2 \, d\tau 
\lb{apes11}
\end{gather}          
 for $ k=1,2$.   
{\it  Lemma \ref{apriori1} follows from \eqref {apes11}.
Indeed, in order to complete the proof of  Lemma \ref{apriori1},}
we regard the parabolic equation \eqref{re2} as the Poisson equation to obtain the elliptic estimate
\begin{gather}\label{paraell1}
\|R_{e}\|_{2} \lesssim \|\partial_{t} R_{e}\| + \|R_{e}\|_{1}.
\end{gather}
Combining \eqref{apes2}, \eqref{apes8}, \eqref{apes2t}, \eqref{apes6}, and \eqref{apes6t},
and then applying \eqref{paraell1}, we have
\begin{align*}
e^{\gamma t} \left( \|R_i(t)\|^2 + \|R_e(t)\|_{2}^2 \right)
+\int_0^t e^{\gamma \tau} \left( \|R_i(\tau)\|^2 + \|R_e(\tau)\|_{2}^2 \right) \,d\tau 
\lesssim_{p} \|R_{i0}\|_{2}^2+\|R_{e0}\|_{2}^2.
\end{align*}
This estimate and \eqref{apes11} immediately yield \eqref{APE1}.
\end{proof}

\begin{proof}[Proof of \eqref{apes11}]
In order to take spatial derivatives we must be careful about the boundary.   
So we use local coordinates near the anode $\cA$.  
We will cover $\cA$ by a fixed number of small balls $V_j, j=1,...,N$ 
and cover the rest of $\Omega$ by one open set $V_0$.  Then we will choose a partial of unity 
$\{\zeta_j\}$ for $V_0,V_1,...,V_N$.  
Furthermore we will choose maps $\bm{\Phi_j}$ that locally flatten the boundary.

Now we make this procedure explicit.  
For any $\ve_{0}>0$, dimension $d=2,3$, and $0\le j\le N$, we choose the domain $V_{j}$, 
the functions $\zeta_{j}$, ${\zeta}_{j}'$, the bijection map 
$\bm{\Phi}_{j}:\mathbb R^{d} \to \mathbb R^{d}$ and its inverse map $\bm{\Psi}_{j}:\mathbb R^{d} \to \mathbb R^{d}$, to satisfy the 
following properties.

					\begin{enumerate}[(i)]
\item For $0\le j\le N$, both $\bm{\Psi}_{j}$ and $\bm{\Phi}_{j}$ belong to $C^{3}(\mathbb R^{d})$ 
and $|\det \nabla \bm{\Phi}_{j}|=1$, $|\det \nabla \bm{\Psi}_{j}|=1$.  
Moreover,  $\nabla  \bm{\Phi}_{j}=T_{j}+S_{j}$, and $\nabla  \bm{\Psi}_{j}={}^{t}T_{j}+S_{j}^{\star}$, where $T_{j}$ is a constant orthonormal matrix, ${}^{t}T_{j}$ is the transpose of $T_{j}$, and $S_{j},S_{j}^{\star} \in C^{2}(\mathbb R^{d})$ are $n\times n$ matrix functions with $\|S_{j}\|_{L^{\infty}} + \|S_{j}^{\star}\|_{L^{\infty}} \leq \ve_{0}$. Furthermore, $\nabla^{2} \bm{\Phi}_{j},\nabla^{2} \bm{\Psi}_{j} \in W^{1,\infty}(\real^d)$.
\item 
$V_{0} \subset \Omega$, while for $j\ne0$, we assume $\Omega\cap V_{j} = \bm{\Phi}_{j}(\mathbb R^{d}_{+}) \cap V_{j}$, and $\cA \cap V_{j} =  \bm{\Phi}_{j}(\partial \mathbb R^{d}_{+}) \cap V_{j}$.
\item 
${\rm supp} \zeta_{j} \subset {\rm supp} \zeta'_{j} \subset V_{j}$ and $ \zeta'_{j}=1$ on ${\rm supp} \zeta_{j}$.
There exist constants $\sigma>0$ and $L>0$ such that $\sum_{j=0}^{N} \zeta_{j}(x)=1$ for $x \in \Omega_{\sigma}$, 
$\sum_{j=1}^{N} \zeta_{j}(x)=1$ for $x \in \cA_{\sigma}$, and ${\rm supp} \nabla \zeta_{j} \subset \{ x \, | \, |x| \leq L\}$ for $j=0,\ldots,N$, where $\Omega_{\sigma}=\{x \in \mathbb R^{d} ; {\rm dist} (x,\Omega)\leq \sigma \}$ and $\cA_{\sigma}=\{x \in \mathbb R^{d} ; {\rm dist} (x,\cA)\leq \sigma \}$.
\item 
For $j\ne0$, 
$ \tilde{\bm{n}} (x) \cdot \nabla_{x} f  =  -(\partial_{y_{d}} f +  \bm{r}_{j}(y) \cdot \nabla_{y} f)$ for $x \in \Omega\cap V_{j}$ and $y=\bm{\Psi}_{j}(x)$,  where $\tilde{\bm{n}}:=-\nabla H/ |\nabla H|$, the restriction of $\tilde{\bm{n}}$ on $\cA$
is a normal unit vector of $\cA$ (see \eqref{Hasp3}), and the vector $\bm{r}_{j}=(r_{j1},\cdots,r_{jd}) \in C^{2}(\bm{\Psi}_{j}(\Omega\cap V_{j}))\cap W^{2,\infty}(\bm{\Psi}_{j}(\Omega\cap V_{j}))$ satisfies $ \|\bm{r}_{j}\|_{L^{\infty}} \leq \ve_{0}$.
\end{enumerate}
Hereafter we set $\ve_{0}=\frac{1}{2}$.

First let us consider {\it the estimate of $R_{i}$  away from the anode} $\cA$ by using the cut-off function $\zeta_{0}$.
We apply a first or second derivative
$\partial^{\alpha}=\partial^{\alpha}_{x}$ to the problem \er{re1} to obtain an equation for $\partial^\alpha R_i$, namely: 
\begin{align*}
&  \partial_t (\partial^{\alpha} {R}_i)
    + k_i \nabla (V+\lambda H ) \cdot \nabla(\partial^{\alpha}{R}_i)
+ p k_i \lambda  |{\nabla} {H}|^{2} (\partial^{\alpha} {R}_i)
\\
& = \partial^{\alpha} \left\{{k_e}h(|\lambda {\nabla} {H}| )e^{-p{H}-\frac{\lambda}{2}{H}}{R}_e \right\}
+ k_i\partial^{\alpha}{f}_i + {\cal R}_{\alpha},
\end{align*}
where the commutators ${\cal R}_{\alpha}$ are defined by
\begin{align*}
{\cal R}_{\alpha}&:= 
k_{i} [\partial^{\alpha}, \nabla (V+\lambda H ) \cdot {\nabla}]  {R}_i
+ p k_i \lambda  [\partial^{\alpha}, |{\nabla} {H}|^{2}] {R}_i.
\end{align*}
				We multiply 
this equation by  $2e^{\gamma t} \zeta_{0}^{2} \partial^{\alpha} {R}_i$
and integrate by parts over $[0,t]\times \Omega$ to obtain 
he following estimate (which is analogous to \eqref{apes5}):
\begin{align}
&e^{\gamma t} \int_\Omega \zeta_{0}^{2} |\partial^{\alpha}R_i|^2 \, dx
+\int_0^t \!\! \int_\Omega e^{\gamma \tau} \zeta_{0}^{2} (2p\lambda k_i|\nabla H|^{2}-\gamma) |\partial^{\alpha} R_i|^2 \, dxd\tau
\notag \\
& \quad +k_i\int_0^t \!\! \int_{\cC} e^{\gamma \tau}\zeta_{0}^{2}  |\partial^{\alpha}R_{i}|^{2}  \nabla (V+\lambda H)\cdot \bm{n} \, dSd\tau 
\notag \\
&= \int_\Omega \zeta_{0}^{2} |\partial^{\alpha} R_{i0}|^2 \, dx
+k_{i} \int_0^t \!\! \int_\Omega   e^{\gamma \tau} \zeta_{0}^{2} |\partial^{\alpha} R_{i}|^{2} \nabla \cdot \left\{ \zeta_{0}^{2} \nabla (V+\lambda H ) \right\}  \, dxd\tau
\notag \\
& \quad +2\int_0^t \!\! \int_\Omega  
e^{\gamma \tau} \zeta_{0}^{2} \left[
\partial^{\alpha} \left\{{k_e}h(|\lambda {\nabla} {H}| )e^{-p{H}-\frac{\lambda}{2}{H}}{R}_e \right\}
+ k_i\partial^{\alpha} f_i +{\cal R}_{\alpha} \right] 
\partial^{\alpha} R_i \, dxd\tau
\notag \\
&\leq \tilde{C}\|R_{i0}\|_{2}^2  + \tilde{C} \int_0^t e^{\gamma \tau} \| \nabla^{|\alpha|} R_i(\tau)\|^2 \, d\tau + C_{p}\int_0^t e^{\gamma \tau} \| R_i(\tau)\|_{|\alpha|-1}^2 \, d\tau
\notag \\
&\quad +C_{p} N(T)\int_0^t e^{\gamma \tau} \| R_i(\tau)\|^2_{|\alpha|} \, d\tau
+ C_{p} \int_0^t e^{\gamma \tau} \| R_e(\tau)\|_{2}^2 \, d\tau,
\lb{apes14}
\end{align}
where we have used \eqref{ell1}, \er{NL6}, and \eqref{NL5} as well as the Schwarz and Sobolev inequalities.  
Here $\tilde{C}>0$ is a constant independent of $p$.
We note that the second term on the left hand side of \er{apes14} is non-negative due to \eqref{coeff1}.
Furthermore, owing to \eqref{chara1}, the third term on the left hand side of \er{apes14} is non-negative and thus can be dropped.

Next we deal with the {\it estimate of $R_{i}$  near the anode} $\cA$.  
In the region $V_j$ for a fixed $j$ we use the local coordinate $y=\bm{\Psi}_{j}(x)$. 
We recall $\ve_{0}=\frac{1}{2}$. 
We define  
\begin{gather*}
(\hat{R}_{i}, \hat{R}_{e}, \hat{V})(t,y):=(R_{e},R_{i}, V)(t,\bm{\Phi}_{j}(y)), \quad
(\hat{H},\hat{\zeta}_{j},\hat{\zeta}_{j}')(y):=(H,\zeta,\zeta')(\bm{\Phi}_{j}(y)), \\
\hat{\nabla} := \nabla\bm{\Psi}_{j} \nabla_{y}.
\end{gather*}
Then  from \eqref{r0} and property (iv) of the local coordinate we have the transformed equation 
\begin{subequations}\label{yr0}
\begin{gather}
\begin{aligned}
&  \partial_t\hat{R}_i
+ k_{i} \lambda |\hat{\nabla} \hat{H}| (\partial_{y_{d}} \hat{R}_i +  \bm{r}_{j} \cdot \nabla_{y} \hat{R}_i)
+ k_i \hat{\nabla} \hat{V} \cdot \hat{\nabla} \hat{R}_i
+ p k_i \lambda  |\hat{\nabla} \hat{H}|^{2} \hat{R}_i
\\
& = 
{k_e}h(|\lambda \hat{\nabla} \hat{H}| )
e^{-p\hat{H}-\frac{\lambda}{2}\hat{H}}\hat{R}_e
+k_i\hat{f}_i, 
\end{aligned}
\label{yre1}\\
\hat{R}_{i}(0,y)= \hat{R}_{i0}(y) :=R_{i0}(\bm{\Phi}_{j}(y)),
\label{yri1}\\
\hat{R}_{i }(t,y_{1},y_{2},0)=0,
\label{yrba}
\end{gather}
\end{subequations}
where $t>0$ and $y \in \bm{\Psi}_{j} (\Omega\cap V_{j})$.
We note that for $c_{1}$ defined in \eqref{Hasp2},
\begin{gather}\label{coeff2}
|\hat{\nabla} \hat{H}(y)|=|\nabla_{x}H(\bm{\Phi}_{j}(y))| \geq c_{1} >0, \quad y \in \bm{\Psi}_{j} (\Omega\cap V_{j}).
\end{gather}
For notational convenience, we may express $\hat{\nabla} \hat{V} \cdot \hat{\nabla}$ as
\begin{gather*}
\hat{\nabla} \hat{V} \cdot \hat{\nabla} = \sum_{l=1}^{d} a_{l}(\nabla_{y} \bm{\Psi}_{j},\hat{\nabla} \hat{V}) \partial_{y_{l}}, \quad
\text{where} \quad |a_{l}(\nabla_{y} \bm{\Psi}_{j},\hat{\nabla} \hat{V})| \lesssim |\hat{\nabla} \hat{V}|.
\end{gather*}

We finally apply the spatial derivatives $\partial^{\alpha}=\partial^{\alpha}_{y}$   
(say for dimension 3) to the problem \er{yr0} to obtain
\begin{subequations}\label{hr0}
\begin{gather}
\begin{aligned}
&  \partial_t (\partial^{\alpha} \hat{R}_i)
+ k_{i} \lambda |\hat{\nabla} \hat{H}| \{\partial_{y_{3}} (\partial^{\alpha}\hat{R}_i) +  \bm{r}_{j} \cdot \nabla_{y} (\partial^{\alpha}\hat{R}_i)\}
+ k_i \hat{\nabla} \hat{V} \cdot \hat{\nabla} (\partial^{\alpha} \hat{R}_i)
+ p k_i \lambda  |\hat{\nabla} \hat{H}|^{2} (\partial^{\alpha} \hat{R}_i)
\\
& = \partial^{\alpha} \left\{{k_e}h(|\lambda \hat{\nabla} \hat{H}| )e^{-p\hat{H}-\frac{\lambda}{2}\hat{H}}\hat{R}_e \right\}
+ k_i\partial^{\alpha}\hat{f}_i + {\cal R}_{\alpha}',
\end{aligned}
\label{hre1}\\
\partial^{\alpha} \hat{R}_{i}(0,y)= \partial^{\alpha} \hat{R}_{i0}(y),
\label{hri1}\\
\partial^{\alpha}\hat{R}_{i }(t,y_{1},y_{2},0)=\left\{
\begin{array}{ll}
{\cal B}={\cal B}(t,y_{1},y_{2},0) & \text{if $\alpha=(0,0,2)$},
\\
0 & \text{otherwise},
\end{array}
\right.
\label{hrba}
\end{gather}
\end{subequations}
where the commutators ${\cal R}_{\alpha}'$ and the boundary data ${\cal B}$ are defined by
\begin{align*}
{\cal R}_{\alpha}'&:= 
k_{i} \lambda \left[\partial^{\alpha}, |\hat{\nabla} \hat{H}| \partial_{y_{3}}\right]  \hat{R}_i
+ k_{i} \lambda \left[\partial^{\alpha}, |\hat{\nabla} \hat{H}| \bm{r}_{j} \cdot \nabla_{y} \right] \hat{R}_i
+ k_{i} \left[\partial^{\alpha}, \hat{\nabla} \hat{V}\cdot \hat{\nabla} \right] \hat{R}_i
\\
&\quad + p k_i \lambda  \left[\partial^{\alpha}, |\hat{\nabla} \hat{H}|^{2}\right] \hat{R}_i,
\\
{\cal B} & := \frac{k_{e}\partial_{y_{3}}\left\{ h(|\hat{\nabla} \hat{V} + \lambda \hat{\nabla} \hat{H}| )e^{-p\hat{H}-\frac{\lambda}{2}\hat{H}}\hat{R}_e \right\}}{k_{i}\left\{ \lambda |\hat{\nabla} \hat{H}|(1+r_{j3}) +a_{3} \right\}}(t,y_{1},y_{2},0).
\end{align*}
			Here 
we have used \eqref{yre1}, \eqref{yrba}, $\hat{R}_{e}(t,y_{1},y_{2},0)=0$, $\|\bm{r}_{j}\|_{L^{\infty}} \leq \frac{1}{2}$, and $|a_{3}| \lesssim |\hat{\nabla}\hat{V}| \lesssim N(T) \ll 1$ in deriving $\partial_{y_{3}} \hat{R}_{i}(t,y_{1},y_{2},0)=0$.  
We have used \eqref{hre1} with $\alpha=(0,0,1)$ and \eqref{hrba} with $|\alpha| =1$ in deriving $\partial_{y_{3}}^{2} \hat{R}_{i}(t,y_{1},y_{2},0)=\cal B$. 
The standard trace theorem gives
\begin{gather}\label{bb1}
\|  \hat{\zeta}_{j}' {\cal B} (t,y_{1},y_{2},0)\|_{L^{2}(\mathbb R^{2}_{+})} \lesssim \|R_{e}\|_{2},
\end{gather}
where we have also used $\|\bm{r}_{j}\|_{L^{\infty}} \leq \frac{1}{2}$ and $|a_{3}|\lesssim|\hat{\nabla}\hat{V}| \lesssim N(T) \ll 1$ to ensure the positivity of the denominator of $\cal B$.

Now we multiply \eqref{hre1} by $2e^{\gamma t} \hat{\zeta}_{j}^{2} \partial^{\alpha}\hat{R}_i$ and
integrate by parts over $[0,t]\times \mathbb R^{3}_{+}$ to obtain 
\begin{align}
{}&
e^{\gamma t} \int_{\mathbb R^{3}_{+}}  \hat{\zeta}_{j}^{2} |\partial^{\alpha}\hat{R}_i|^2 \, dy
+\int_0^t \!\! \int_{\mathbb R^{3}_{+}} e^{\gamma \tau} \hat{\zeta}_{j}^{2} (2p\lambda k_i|\hat{\nabla}\hat{H}|^{2} -\gamma)|\partial^{\alpha} \hat{R}_i|^2 \, dyd\tau
\notag \\
&= \int_\Omega \hat{\zeta}_{j}^{2} |\partial_{\alpha} \hat{R}_{i0}|^2 \, dy
+ k_{i} \int_0^t \!\! \int_{\mathbb R^{2}} e^{\gamma \tau} \hat{\zeta}_{j}^{2} \left\{ \lambda |\hat{\nabla} \hat{H}|(1+r_{j3}) +a_{3} \right\} |\partial^{\alpha} \hat{R}_i|^2(t,y_{1},y_{2},0) dy_{1}dy_{2}d\tau
\notag \\
& \quad + k_{i} \int_0^t \!\! \int_{\mathbb R^{3}_{+}}  e^{\gamma \tau} |\partial^{\alpha} \hat{R}_i|^2 \sum_{l=1}^{3} \partial_{y_{l}}\left[ \hat{\zeta}_{j}^{2} \left\{ \lambda |\hat{\nabla} \hat{H}|(\delta_{3l} + r_{jl}) +a_{l} \right\} \right] \, dyd\tau
\notag \\
& \quad +2\int_0^t \!\! \int_{\mathbb R^{3}_{+}}
e^{\gamma \tau} \hat{\zeta}_{j}^{2}\left(
\partial^{\alpha} \left\{{k_e}h(|\lambda \hat{\nabla} \hat{H}| )e^{-p\hat{H}-\frac{\lambda}{2}\hat{H}}\hat{R}_e \right\}
+ k_i\partial^{\alpha}\hat{f}_i + {\cal R}_{\alpha}'
\right) \partial^{\alpha} \hat{R}_i \, dyd\tau
\notag \\
&\leq \tilde{C}\|R_{i0}\|_{2}^2  + \tilde{C} \int_0^t e^{\gamma \tau} \| \nabla^{|\alpha|} R_i(\tau)\|^2 \, d\tau + C_{p}\int_0^t e^{\gamma \tau} \| R_i(\tau)\|_{|\alpha|-1}^2 \, d\tau
\notag \\
&\quad +C_{p} N(T)\int_0^t e^{\gamma \tau} \| R_i(\tau)\|^2_{|\alpha|} \, d\tau
+ C_{p} \int_0^t e^{\gamma \tau} \| R_e(\tau)\|_{2}^2 \, d\tau,
\lb{apes13}
\end{align}
where we have used \eqref{ell1}, \er{NL6}, \er{NL5}, \eqref{hrba}, \eqref{bb1}, and ${\rm supp} \zeta'_{j}=1$ on ${\rm supp} \zeta_{j}$ as well as the Schwarz and Sobolev inequalities. 
Here $\tilde{C}>0$ is a constant independent of $p$.
We note that the second term on the left hand side of \er{apes13} is non-negative due to \eqref{coeff1} and \eqref{coeff2}.

We are finally able to prove \eqref{apes11}.
We recall that the second terms on the left sides of \er{apes14} and \er{apes13} provide  good contributions, while the third term on the left side of \er{apes14}, 
which is an integral over $\cC$, can be dropped.
We sum up $\eqref{apes14}$ and $\eqref{apes13}$ for $j=1,\ldots,N$ and $\alpha$ with $|\alpha|=1$.
Then we take $p$ sufficiently large to absorb the term $\tilde{C} \int_0^t e^{\gamma \tau} \| \nabla^{|\alpha|} R_i(\tau)\|^2\, d\tau$ into the good contribution on the left hand side of the resulting inequality, and also let $N(T)$ be small relative to $p$.
Using \eqref{apes6}, we finally arrive at \eqref{apes11} with $k=1$.
In a similar fashion we  obtain \eqref{apes11} for $k=2$. 
\end{proof}

\subsection{Nonlinear instability for $\kappa<0$}    \lb{S3.2}

This subsection provides the proof of Theorem \ref {thm2}.  
The key assumption is that $\kappa(\lambda)<0$.  
Recall that the self-adjoint operator $A = -\Delta - g(\lambda \nabla H)$ has its lowest eigenvalue 
$-\mu_1 = \kappa(\lambda)$ with an eigenfunction denoted by $\varphi_1$.  
Then $\varphi_1$ does not vanish so we may assume it is positive in $\Omega$ and that $\|\varphi_1\|=1$.

For convenience, we choose initial data $\cR_0  =(\delta\psi_i,\delta \psi_e) \in X$
that satisfies the compatibility condition \eqref{com2}, as well as
\begin{equation}\lb{ass1}
\psi_i,\psi_e \geq 0,  \quad
\psi_e \not\equiv0, \quad
\|\psi_i\|_{1}^2+\|\psi_e\|_{1}^2=1.
\end{equation}

\begin{lem}\lb{apriori3}
Let $\kappa(\lambda)<0$.
There exists $\ve_0>0$ such that 
if the solution $(R_i,R_e)$ of problem \er{r0} with $(R_{i0},R_{e0})=(\delta \psi_i, \delta \psi_e)$
for some $\delta \in (0,\ve_0)$ satisfies $N(T) \leq \ve_0$ and if 
\begin{equation*} 
\left\|R_e(t)  -  \delta e^{k_etA}\psi_e \right\|
\leq \delta e^{- \kappa(\lambda)k_e t}
\qu \text{for} \ \ t \in[0,T],
\end{equation*}
then there is a constant $m>0$ such that 
\begin{gather}
\|R_i(t)\|_2+\|R_e(t)\|_2 \lesssim \delta e^{- \kappa(\lambda)k_e t},
\lb{APE4} \\
    \left\|R_e(t)  -  \delta e^{k_etA}\psi_e \right\|
< m\delta^2 e^{- 2\kappa(\lambda)k_e t}
\quad \text{for} \ \ t\in[0,T].  
\lb{APE5} 
\end{gather}
\end{lem}

\begin{proof}
Note that $\|e^{k_etA}\psi_e\| \le e^{- \kappa(\lambda)k_e t}$, so that
\begin{equation}\lb{apes20} 
\|R_e(t)\| \lesssim \delta e^{- \kappa(\lambda)k_e t}.
\end{equation} 
By a similar energy method as in the proof of Lemma \ref{apriori1},
the inequality \er{APE4} can be shown as follows.
Let $\gamma =0$ but $p$ be the same constant as in the proof of Lemma \ref{apriori1}.
Since $p$ is fixed, we do not need to be so precise for its dependence, so we refrain to write it in the following inequalities.  
Using \eqref{apes20} instead of \eqref{poin} to estimate the term $g(|\lambda \nabla H|){R}_e^2$ in \eqref{apes1}, 
we obtain
\[
\|R_e(t)\|^2
+\int_0^t \|R_e(\tau)\|_1^2 \,d\tau
\lesssim  \delta^{2} e^{- 2\kappa(\lambda)k_e t},
\]
where we have also used $\|R_{e0}\| \leq \delta$.
It also follows from the same method of the derivation of \eqref{apes8} that 
\[
\|\nabla R_e(t)\|^2
+\int_0^t  \| \partial_{t} R_e(\tau)\|^2 \,d\tau
\lesssim \delta^{2} e^{- 2\kappa(\lambda)k_e t}.
\]
Following the derivation of \eqref{apes2t} but using the above inequality instead of \eqref{poin},
we arrive at
\[
\| \partial_{t} R_e(t)\|^2
+\int_0^t \| \partial_{t} R_e(\tau)\|_1^2 \,d\tau
\lesssim \delta^{2} e^{- 2\kappa(\lambda)k_e t} + \int_0^t \| \partial_{t} R_i(\tau)\|^2 \, d\tau.
\]
Furthermore, we just follow the derivations of \eqref{apes6}, \eqref{apes6t}, and \eqref{apes11} to obtain 
\begin{gather*}
\|R_i(t)\|^2 +\int_0^t \|R_i(\tau)\|^2 \,d\tau 
\lesssim \delta^{2} e^{- 2\kappa(\lambda)k_e t},
\\
\|\partial_{t}R_i(t)\|^2 +\int_0^t \|\partial_{t}R_i(\tau)\|^2 \,d\tau 
\lesssim \delta^{2} e^{- 2\kappa(\lambda)k_e t},
\\
\|\nabla^{k}R_i(t)\|^2 +\int_0^t \|\nabla^{k}R_i(\tau)\|^2 \,d\tau 
\lesssim \delta^{2} e^{- 2\kappa(\lambda)k_e t} + \int_0^t \|R_e(\tau)\|^2_2 \, d\tau \quad \text{for $k=1,2$.}
\end{gather*}
Combining these inequalities in the same way as in the last paragraph of the proof of Lemma \ref{apriori1},
we conclude \eqref{APE4}.

To prove \er{APE5}, 
we use  Duhamel's principle to write  
\begin{gather*}
R_e(t)  =  e^{k_etA} \delta\psi_e  +  k_e\int_0^t e^{k_e(t-\tau)A} f_e(\tau) d\tau.  
\end{gather*}
It follows that 
\begin{align*}
\left\|R_e(t)-\delta e^{k_etA} \psi_e \right\|  
&=  k_e \int_0^t  \|e^{k_e(t-\tau)A} f_e(\tau)\| d\tau  
\\
& \lesssim \int_0^t  e^{- \kappa(\lambda)k_e (t-\tau)} (\|R_{i}(\tau)\|_2^{2}+\|R_{e}(\tau)\|_2^{2}) d\tau  
< m\delta^2 e^{- 2\kappa(\lambda)k_e t}
\end{align*}
by \er{NL2} and \er{APE4},
where $m$ is a positive constant independent of $t$.
\end{proof}

\bigskip   
\begin{proof}[Proof of Theorem \ref{thm2}]
The assertion of Theorem \ref{thm2} is equivalent to the 
{\it 
existence of $\ve>0$ such that for all $\delta>0$, 
the rewritten problem \eqref{r0} with the initial data $\|R_{i0}\|_{2} + \|R_{e0}\|_{2}  < \delta$ and $R_{e0} \neq 0$ has a unique solution $(R_{i},R_{e}) \in C([0,T];X)$ such that  $\|R_{i}(T)\|_{2} + \|R_{e}(T)\|_{2}  \ge \ve$ for some $T>0$.
}
In what follows, we prove this equivalent assertion.

In terms of $\ve_0$ and $m$ given in Lemma \ref{apriori3},
we also define $\nu$ and $T$ and take positive constants $\ve$ and  $\delta$ to satisfy 
\begin{gather*}
\nu=\langle\psi_e,\varphi_1\rangle,
\quad
\ve<\min \left\{1,\ve_0,\frac{\nu}{3m},\frac{\nu^2}{6m}\right\},
\quad
\delta<\min\left\{\frac{\ve_0}{2},\ve\right\},
\\  
T = \sup\left\{   t; \ 
\left\|R_e(\tau)-\delta  e^{k_e\tau A}\psi_e  \right\|
\leq    
\delta e^{- \kappa(\lambda)k_e \tau},  \  
\|R_{i}(\tau)\|_{2} + \|R_{e}(\tau)\|_{2} < \ve \  \text{for} \ \tau \in (0,t) 
\!\right\},
\end{gather*}
where   $\nu \in (0,1]$ by \er{ass1}.
Then we choose any $T_*>0$ so that 
\[
2\ve < \delta \nu e^{-\kappa(\lambda) k_eT_*} < 3\ve.
\]

Let us show that either $T_*<T$ or $N(T)=\ve$ holds. Suppose, contrary to our claim, that
$T_*\geq T$ and $N(T)<\ve$ hold.
From the definition of $T$, the following equality holds:
\begin{align*}
\delta e^{-\kappa(\lambda) k_e T}
=\left\|   R_e(T) - \delta e^{k_eAT}\psi_e  \right\|
< m\delta^2 e^{- 2\kappa(\lambda)k_eT }, 
\end{align*}
where the inequality holds by the virtue of Lemma \ref{apriori3}.
This gives $1/m<\delta e^{-\kappa(\lambda) k_e T}$. 
On the other hand, the definitions of $T_*$ and $\ve$ yield
$\delta e^{-\kappa(\lambda) k_e T_*}<3\ve/\nu\leq 1/m$.
These two inequalities lead to $T_*<T$ 
which contradicts our assumption $T_*\geq T$.

If $N(T)=\ve$ holds, the proof is complete.
Thus, it remains to deal with the case $T_*<T$ 
for which \er{APE5} holds from Lemma \ref{apriori3}.
By \er{APE5}, the triangle inequality, and Parseval's equality,
we obtain
\begin{align*}
\|R_e(T_*)\| 
 \geq \delta e^{- \kappa(\lambda)k_eT_* }\langle \psi_e, \varphi_1 \rangle
 -m\delta^2e^{- 2\kappa(\lambda)k_eT_* }.  
\end{align*}
By the conditions on $\nu$, $T_*$, and $\ve$, this expression is estimated from below by
\begin{align*}
\delta e^{- \kappa(\lambda)k_eT_* }\langle \psi_e, \varphi_1 \rangle
-m\delta^2e^{- 2\kappa(\lambda)k_eT_* }
& = 
\delta \nu e^{- \kappa(\lambda)k_eT_* } \left(1 -\frac{m}{\nu^2} \delta \nu e^{- \kappa(\lambda)k_eT_* } \right)
\\
& \geq 2\ve  \left(1 -3\frac{m}{\nu^2}\ve \right) 
>\ve.
\end{align*}
Consequently, we conclude that $\|R_e(T_*)\| \geq \ve$.
\end{proof}

\subsection{Electron-free case}\lb{S3.3}
This subsection is devoted to the proof of Proposition \ref{pro1}, which asserts that the set 
$\{(\rho_{i0},\rho_{e0})\in H^2(\Omega) \times H^2(\Omega) \, ; \, \text{$\rho_{e0}\equiv0$ and \eqref{com} holds} \}$
is a local stable manifold of the system \er{re1}--\er{re3} for any $\lambda > 0$.
Similarly to the proof of Theorem \ref{thm1} in subsection \ref{S3.1}, 
we can establish the unique existence of time-global solutions to problem \er{r0}
by combining the time-local solvability stated in Lemma \ref{local-time} 
and the following a priori estimate.

\begin{lem}\lb{apriori2}
Let $\lambda>0$.
Suppose that $(R_i,R_e,V)$ satisfying \er{reg2} is a solution 
to problem \er{r0} in a time interval $[0,T]$ with the initial data $R_{e0}\equiv0$.
There exist constants $\de>0$ and $\gamma>0$ such that if $N(T) \leq \de$, then there hold that
\begin{gather}
R_e(t,x)=0 \quad \text {for} \ \ (t,x) \in [0,T]\times \Omega,
\lb{APE2} \\
e^{\ga t}\|R_i(t)\|_2^{2}
+\int_0^t e^{\ga \tau} \|R_i(\tau)\|_2^{2} \,d\tau
\lesssim \|R_{i0}\|_2^{2} \quad \text {for} \ \ t \in [0,T].
\lb{APE3}
\end{gather}
\end{lem} 
\begin{proof}
We multiply \er{re2} by $2R_e$,
integrate it by parts over $[0,t]\times \Omega$, 
and use $R_{e0}\equiv0$ and the boundary conditions \er{rba} and \er{rbc} to obtain 
\begin{align*}
\int_\Omega {R}_e^2 \, dx
+2k_e\int_0^t \!\! \int_\Omega |\nabla R_e|^2 \, dxd\tau
&=2k_e\int_0^t \!\! \int_\Omega g(|\lambda \nabla H|){R}_e^2 \, dxd\tau
+2k_e\int_0^t \!\! \int_\Omega f_e R_e \, dxd\tau
\notag \\
&\le \mu \int_0^t \|\nabla R_{e}(\tau)\|^2 \, d\tau
+ C_{\mu} \int_0^t \|R_{e}(\tau)\|^2 \, d\tau,
\end{align*}
where $\mu>0$ and $C_{\mu}$ is constant depending on $\mu$.  
We have used \er{NL2} and Schwarz's inequality in deriving the inequality.
Then letting $\mu<2k_e$ and applying Gr\"onwall's inequality, we obtain \eqref{APE2}.

It remains to prove \er{APE3}, knowing that $R_e\equiv 0$.  
Even for the case $\lambda>0$, the inequalities \er{apes6} and \er{apes11} in the proof of Lemma \ref{apriori1} are true.  
Thus we obtain \eqref{APE3}.
\end{proof}

\begin{proof}     [Proof of Proposition \ref{pro1}] 
It suffices to prove that the ions disappear; that is, $R_i=\ro_i e^{- pH}$ eventually vanishes.
We showed above that $R_e\equiv0$. 
Hence, it suffices to prove that there exists some time $T_0>0$ such that
$R_i(T_0,x)=0$ for any $x \in \Omega$, 
because the problem \er{r0} with the initial time $t=T_0$ and 
the initial data $R_i(T_0,x)=R_e(T_0,x)=0$
has a unique solution $(R_i,R_e,V)=(0,0,0)$.
We note that $N(\infty)$ can be made arbitrarily small by taking $\|R_{i0}\|_{2}$ suitably small.

We claim that 
\begin{gather}\label{ulbound1}
0<\Phi(t,x):=(V + \lambda H)(t,x) < \lambda, \quad t \geq 1, \ x \in \Omega.
\end{gather}
Indeed, we apply the maximum principle to the linear elliptic equation \eqref{re3} 
together with \eqref{rba}, \eqref{rbc}, $R_e\equiv0$ and $R_{i} \geq 0$ 
so as to obtain $V \leq 0$. 
Since $H<1$ in $\Omega$, we have  $(V + \lambda H)(t,x) < \lambda$.  

It remains to show that $(V + \lambda H)(t,x) >0$ in $\Omega$. 
First we note that $(V + \lambda H)(t,x)=0$ if  $x \in \cA$. 
We will make use of the characteristics of the hyperbolic equation \eqref{re1}.
For a suitably small $\ve>0$, we let $\cA_{\ve}=\{x \in \Omega ; {\rm dist} (x,\cA)<\ve \}$.  
Given $x\in\cA_\ve$ and $t\ge 1$, let $y(\cdot)$ be   
 the characteristic (parameterized by $s$) defined by  
 $$\frac{dy}{ds} = \nabla (V + \lambda H)(s,y(s)) \text{ with } y(t)=x.$$  
 It  satisfies $y(t_{0}) \in \cA$ for some $t_{0} \in (0,t)$.
Furthermore, it is seen from \eqref{Hasp2} and \eqref{NT1} that 
\begin{gather*}
\frac{d}{ds}(V + \lambda H)(s,y(s)) = (\partial_{t} V)(s,y(s)) + |\nabla V + \lambda \nabla H|^{2}(s,y(s))\geq \frac{1}{2} c_{1}^2>0.
\end{gather*}
These facts mean that $(V + \lambda H)(t,x)=(V + \lambda H)(t,y(t))>(V + \lambda H)(t_{0},y(t_{0}))=0$ for $t\geq 1$ and $x \in  \cA_{\ve}$. 
On the other hand, the maximum principle ensures that $H(x^*) \geq c_{2}>0$ 
for any $x^* \in \Omega  \backslash \cA_{\ve}$, where $c_2$ depends on $\ve$.
Using this and $|V|  \lesssim N(\infty) \ll 1$, we conclude that $(V + \lambda H)(t,x^*) >0$ for $t \geq 1$ and $x^* \in \Omega  \backslash \cA_{\ve}$.
Thus the claim \eqref{ulbound1} has been proven.

Now we fix $T=c^{-1}\lambda+1$ and define 
\begin{gather*}
Q:=\left\{(t,x) \, \left| \,  
1 \leq t \leq T, \  0 \leq (V + \lambda H)(t,x) \leq c (t-1) \right\}\right., 
\end{gather*}
where $c$ is a small positive constant to be chosen later. 
From \eqref{ulbound1},  we see that $(V + \lambda H)(t,x)=0$ holds if and only if $x \in \cA$.
The boundary $\partial Q$ has three parts 
(where $t=T$, where $V+\lambda H=0$ and where $V+\lambda H=c(t-1)$.  That is, 
\begin{gather*}
\partial Q=S_{1}\cup S_{2}\cup S_{3},
\end{gather*}
where
\begin{align*}
S_{1}&:=\left\{(t,x) \, \left| \,  1 < t < T, \  x \in \cA \right. \right\} , 
\quad S_{3}:=\{T\}\times \Omega , 
\\
S_{2}&:=\left\{(t,x) \, \left| \,  1 < t < T, \  (V + \lambda H)(t,x) = c (t-1)\right.\right\}.
\end{align*}
From \eqref{Hasp2} and the implicit function theorem, we see that $S_{2}$ is a $C^{1}$-surface.
The surface $S_2$ has the normal vector $[\partial_t V-c,\ \ \nabla(V+\lambda H)]$ in time-space coordinates.
Now 
we multiply the hyperbolic equation \eqref{re1} by $2{R}_i$ and integrate over the domain $Q$ to obtain 
\begin{align*}   
\int \!\! \int_{Q} \partial_t ({R}_i^2)
+ \nabla \cdot \left\{k_i \nabla ({V}+\lambda {H} ) {R}_i^2\right\} \, dxd\tau
&= -k_{i} \int \!\! \int_Q (\Delta {V}+ 2 p\nabla {V} \cdot \nabla {H} -2p\lambda |\nabla {H}|^{2}){R}_i\, dxd\tau
\notag \\
& \lesssim  \int \!\! \int_{Q} {R}_i^2 \, dxd\tau,  
\end{align*}
where we have used the bounds \eqref{ell1} and $H \in W^{1,\infty}(\Omega)$ 
in deriving the inequality.
The left side equals an integral over $\partial Q$.  Because the integral over $S_1$ vanishes,  
we obtain 
\begin{equation}\lb{apes19}
\int_{S_{2}} k_{i} \left\{ |\nabla ({V}+\lambda {H})|^{2} + \partial_{t} V -c\right\}R_i^2  \, ds
+\int_\Omega R_i^2 \left(T,x\right)  \, dx   \lesssim  \int \!\! \int_{Q} {R}_i^2 \, dxd\tau.
\end{equation}
However we see from \eqref{NT1} that $|\partial_{t} V| \ll 1$ if $N(\infty) \ll 1$.
This fact with \er{Hasp2} ensures that
the first term in \er{apes19} is non-negative for $c\le c_0$ for suitably small $c_0$.  
The inequality is valid for all $c_0^{-1}\lambda +1 \le T<\infty$.
Applying Gr\"onwall's inequality, we deduce that $\|R_i(T)\|^2=0$ 
which means $\rho_i(T,x)=0$ for all $x \in \Omega$ and all such $T$.
The proof is complete.        
\end{proof}


\section{Local bifurcation}\lb{S4}
This section is devoted to proving Theorem \ref{thm3} 
which asserts the local bifurcation of non-trivial stationary solutions.
We write the stationary system as 
\begin{equation*}  
\sG_j(\lambda,\rho_i,R_e,V)=0 \qu \text{for $j=1,2,3,$}  
\end{equation*}
where $R_{e}:=\ro_e e^{\frac{\lambda}{2}H}$ and 
\begin{align*}
&\sG_1(\lambda,\rho_i,R_e,V)
:=k_i \nabla \cdot \left\{ \rho_i \nabla(V+\lambda H)\right\}
-k_eh\left(|\nabla (V+\lambda H)|\right)e^{-\frac{\lambda}{2}H}R_e =0,
\\
&\sG_2(\lambda,R_e,V)
:=  -\Delta R_e 
-\nabla V \cdot \nabla R_{e}
+\left\{\frac{\lambda}{2} \nabla V \cdot \nabla H
-\Delta V
+\frac{\lambda^2}{4}|\nabla H|^{2}
-h\left(|\nabla (V+\lambda H)|\right)\right\}R_e =0,
\\
&\sG_3(\lambda,\rho_i,R_e,V)
:=\Delta V-\rho_i+e^{-\frac{\lambda}{2}H}R_e=0.
\end{align*}
This follows directly from the original system \eqref{re1}--\eqref{re3} 
and from $g(s)=ase^{-\frac{b}{s}}-\tfrac{s^{2}}{4}$.
Because the operator $\cG_1$ is not elliptic, we cannot directly apply the classical Crandall-Rabinowitz  Theorem.  
Our alternative goal is to rewrite this system in order to be able to employ a Lyapunov-Schmidt reduction, following the ideas in \cite{KN}.

\subsection{Transport equation}  
In order to handle the hyperbolic equation $\sG_1=0$, we begin with the boundary value problem for the simple linear transport equation
\begin{subequations}\label{TranEq1}
\begin{gather}
\nabla\Phi\cdot\nabla\rho_i + b\rho_i = f \hbox{ in } \Omega, 
\\
\rho_{i} =0 \hbox{ on } \cA.
\end{gather}
\end{subequations}

\begin{lem}     \label{rho_iLemma}
Let $\Phi \in C^{3+\ell,\alpha}(\overline\Omega)$, $b \in C^{1+\ell,\alpha}(\overline\Omega)$, and $\Phi \in C^{2+\ell,\alpha}(\overline\Omega)$ for integers $\ell \geq 0$.
Assume that 
\bqn    \label{ineqsPhi}
\nabla\Phi\ne0 \text{ in } \Omega, \quad 
\bm{n}\cdot \nabla\Phi >0 \text{ on }\cC, \quad
\Phi=0, \ \bm{n}\cdot \nabla\Phi <0 \text{ on }\cA,  
\eqn
where $\bm{n}$ denotes the unit normal vector outward from $\Omega$.
Then there exists a unique solution $\rho_i \in C^{1+\ell,\al}(\overline\Omega)$ of  the boundary value problem \eqref{TranEq1}.
Furthermore,    
\begin{gather}\label{hypes1}
\|\rho_{i}\|_{C^{1+\ell,\al}} \leq  C
   \| f \|_{C^{1+\ell,\al}},
\end{gather}
where 
$C$ is a positive constant that depends on $\inf_{x\in \Omega} |\nabla \Phi(x)|$, $\|\Phi\|_{C^{3+\ell,\al}}$, and $\| b \|_{C^{1+\ell,\al}}$ but is independent of $f$.  

\end{lem}
\begin{proof}
The conditions \eqref{ineqsPhi} mean that $\nabla\Phi$ 
has no critical points in $\Omega$, while it points outward on $\cC$ and inward on $\cA$.  
The characteristics of the transport equation in \eqref{TranEq1} are defined by 
$\frac{dy}{ds}=\nabla\Phi(y)$ with $y(0)=x$ for $x\in\overline\Omega$.    
We denote them by $y(s;x)$.  
Of course, the three expressions in \eqref{ineqsPhi} are not only non-zero 
but are bounded away from zero.  
These assumptions on $\Phi$ imply that each characteristic runs from $\cA$ to $\cC$ 
and the set of characteristics covers all of $\overline\Omega$.  
This is a consequence of the assumption that $\Phi\in C^{3+\ell,\al}(\overline\Omega)$ 
together with standard facts about ODEs.  
On each characteristic $y(s;x)$, it takes the form 
$\frac{d\rho_i}{ds} + b\rho_i = f$, which is easily solved.    
Because $\Phi\in C^{3+\ell,\al}$, $b\in C^{1+\ell,\al}$, and $f\in C^{2+\ell,\al}$,
the unique solution $\rho_i$ belongs to $C^{1+\ell,\al}$.

It remains to show the bound \eqref{hypes1}. In terms of  the characteristic $y(s;x)$ within $\overline\Omega$, the solution $\rho_{i}$ is written as
\begin{gather}\label{form-rhoi}
\rho_i(x)=\int_{s_{a}(x)}^{0} e^{-\int_{s}^{0} b(y(\sigma;x)) d\sigma} f(y(s;x)) ds, \quad x\in \overline{\Omega},
\end{gather}
where the parameter value $s_{a}(x)\leq 0$ is defined so that $y(s_{a}(x);x) \in \cA$. 
To obtain a bound of $|s_{a}(x)|$, we observe that
\begin{gather*}
\Phi(x) - 0 = \int_{s_{a}(x)}^{0} \frac{d}{ds} \Phi(y(s;x)) ds = \int_{s_{a}(x)}^{0}  |\nabla \Phi(y(s;x))|^{2} ds \geq \inf_{\Omega} |\nabla \Phi(x)|^{2} |s_{a}(x)|.
\end{gather*}
This gives $|s_{a}(x)| \leq \|\Phi\|_{C^{3+\ell,\al}} / \inf_{\Omega} |\nabla \Phi|^{2}$ 
and also $\Phi>0$ in $\Omega$.
We note that ${\Phi}(x)=0$ holds if and only if $x \in \cA$ and so represents $\cA$. 
Standard ODE theory with $\Phi\in C^{3+\ell,\al}$ ensures that $y(s;\cdot) \in C^{2+\ell}$ as well as its derivatives are bounded by $\|\Phi\|_{C^{3+\ell,\al}}$.
Let us check the regularity of $s_{a}$.
We consider the equation $\widetilde{\Phi}(\widetilde{y}(s;x))=0$, where $\widetilde{\Phi}$ is a smooth extension of $\Phi$ to the exterior of $\Omega$, and $\widetilde{y}$ denotes the characteristic associated to the extension $\widetilde{\Phi}$.
As above, $\widetilde{\Phi}(x)=0$ represents $\cA$.
Applying the implicit function theorem to $\widetilde{\Phi}(\widetilde{y}(s;x))=0$ at the point $s=s_{a}(x)$,
we see that $s_{a} \in C^{2+\ell}$ as well as its derivatives are bounded by $\|\Phi\|_{C^{3+\ell,\al}}$ and $\inf_{\Omega} |\nabla \Phi|$. 
Making use of these properties in \eqref{form-rhoi}, we arrive at \eqref{hypes1}.
\end{proof}

\subsection{Linearized operator}\label{LO1}
We define the spaces 
\begin{align*}
 X^{\ell}&:=  \{f\in C^{1+\ell,\alpha}(\overline{\Omega});f|_{\cA}=0\}  \times C^{2+\ell,\alpha}_0(\overline\Omega) \times C^{3+\ell,\al}_0(\overline\Omega), 
\\
Y^{\ell}&:= C^{1+\ell,\alpha}(\overline\Omega) \times C^{\ell,\alpha}(\overline\Omega) \times C^{1+\ell,\al}(\overline\Omega), \quad \ell=0,1,
\end{align*}
where the subscript $0$ refers to functions (not their derivatives) that vanish on 
$\partial\Omega = \cA \cup \cC$ 
and $\al\in(0,1)$ is a fixed number indicating H\"older continuity.  
We define the linearized operator    
\begin{align*}
\cL&:= \partial_{(\rho_i,R_e,V)} \sG(\lambda^*,0,0,0)  =  
\begin{bmatrix}
\, k_{i} {\rm div} ( \, \cdot \, \lambda^{*} \nabla  H ) & -k_{e}h(|\lambda^{*} \nabla H|)e^{-\frac{\lambda^{*}}{2}H}   & 0 \,  
\\
\, 0 & -\Delta -g(|\lambda^{*} \nabla H|) & 0 \,
\\
\, -1 & e^{-\frac{\lambda^{*}}{2}H} & \Delta \,
\end{bmatrix}.
\end{align*}
The three rows of $\sL$ are denoted as $\sL_1,\sL_2,\sL_3$.  
Notice that the nullspace of $\sL$ is one-dimensional. 
Indeed, we recall that $\kappa(\lambda^{*})=0$, where $\kappa$ was defined in \eqref{kap0}.
As mentioned earlier, zero is the lowest eigenvalue of the self-adjoint 
operator $A$ on $L^2(\Omega)$ with domain $H^2(\Omega)\cap H^1_0(\Omega)$. 
We denote by $\varphi_e^{*} \in H^2(\Omega)\cap H^1_0(\Omega)$ 
the corresponding positive eigenfunction.  
On fact, $\varphi_e^{*}$ belongs to $C^{3,\al}(\overline{\Omega})$.  
Given $\varphi_e^*$, the other two components $\sL_1=0$ and $\sL_3=0$ have unique solutions 
in $C^{3,\al}(\overline{\Omega})$ and $C^{5,\al}(\overline{\Omega})$, respectively, due to Lemma \ref{rho_iLemma}.  
So the nullspace of $\sL$ is one-dimensional.  
We normalize this eigenfunction by requiring $\| \varphi_i^* \|=\| \varphi_e^* \|=\| \varphi_v^* \|=1$.    
For brevity we denote $\varPhi^{*}:={}^{t}(\varphi_i^*,\varphi_e^*,\varphi_v^*)$.  
We denote by $Q$ the orthogonal projection to the range $R(\cL)$.  
Then $P:=I-Q$ is a one-dimensional projection. 

We claim that the range of $\sL$ is 
$R(\cL)=\{(f_{i},f_e,f_v) \in Y^{0} \, ; \, \langle f_e,\varphi_e^{*} \rangle=0 \}.$  
This will follow directly from 
$R(\sL_{2})=\{f_e \in C^\al(\overline{\Omega}); \langle f_e,\varphi_e^{*} \rangle=0  \}.  $ 
To prove this fact about $\sL_2$ and to make use of duality, we temporarily convert to Hilbert spaces.   
Let $X^\bullet$ be the same as $X^0$ and $Y^\bullet$ the same as $Y^0$, 
except that $C^{k,\al}$ is replaced by $H^k$.  
We define $\sL^{\bullet} :X^{\bullet} \to Y^{\bullet}$ as the unique linear extension 
of $\sL^{\bullet}$ of $\sL$ to $X^{\bullet}$.
Note that $\sL_{2}^{\bullet}$ is a self-adjoint operator and its domain is 
independent of $S_{i}$ and $W$.  
By standard operator theory from the fact that 
$N(\sL_{2}^{\bullet})=\langle \varphi_e^{*} \rangle$, we have 
$R(\sL_{2}^\bullet)=\{ f_e \in L^2(\Omega); \langle f_e,\varphi_e^{*} \rangle=0\}.$
Because $Y\subset Y^\bullet$, the condition $\langle f_e,\varphi_e^{*} \rangle=0$ 
is necessary for the solvability of the problem 
$\sL_{2}^{\bullet}(S_e)=f_{e} \in C^\al(\overline{\Omega})$.
On the other hand, if 
$f_e \in \{f_{e} \in C^\al(\overline{\Omega}) ; \langle f_e,\varphi_e^{*} \rangle=0 \}$,
we have a unique solution $S_{e} \in H^{2}(\Omega) \cap H^{1}_{0}(\Omega)$ 
of $\sL_{2}^\bullet(S_e) =f_{e}$.
But  $S_e \in C^{2,\al}(\overline{\Omega})$ by standard elliptic estimates.
Thus we have determined the range of $\sL_2$.

\subsection{The system, rewritten}
We will write the desired solution $(\lambda(s),\rho_{i}(s),R_{e}(s),V(s))$ of the full nonlinear system 
$\sG=0$ in the form 
\begin{gather*}
\lambda=\lambda^{*}+s \Lambda, \quad \rho_{i}=s (\varphi_{i}^{*}+s S_{i}), \quad R_{e}=s(\varphi_{e}^{*}+s S_{e}), \quad V= { s(\varphi_{v}^{*}+s W) }, \quad  (S_{i },S_{e}, { W}) \in R(Q),
\end{gather*}
and then look for $(\Lambda,S_{i } ,S_{e}, W)$, depending on a small parameter $s$.  

We begin by writing the three components of $\sG=0$ as follows.
\begin{gather*} 
\begin{aligned}
&k_i\nabla\cdot(s(\varphi_i^*+sS_i) \lambda^* \nabla H) 
+ k_i\nabla\cdot \big(  s(\varphi_i^{*}+sS_i) \nabla(s(\varphi_v^{*}+sW)+s\Lambda H)\big)  \\
& \quad - k_eh(|\lambda^*\nabla H|) e^{-\frac{\lambda^*}{2} H} s(\varphi_e^{*}+sS_e)  
- s\Lambda k_e|\nabla H| h'(|\lambda^*\nabla H)|) e^{-\frac{\lambda^*}{2} H} s(\varphi_e^{*}+sS_e) \\
& \quad + s\Lambda k_eh(|\lambda^*\nabla H|)\tfrac H2 e^{-\frac{\lambda^*}{2} H} s(\varphi_e^{*}+sS_e) = s^2 \cR_1,
\end{aligned}
\\
-\Delta(s(\varphi_e^*+sS_e)) - g(|\lambda^*\nabla H|)s(\varphi_e^*+sS_e)  
- s\Lambda|\nabla H| g'(|\lambda^*\nabla H|) s(\varphi_e^* + sS_e)   
=  s^2 \cR_2,
\end{gather*}
and 
\begin{align*}
\Delta(s(\varphi^*_v + sW)  - s(\varphi^*_i+sS_i)  +  e^{-\frac{\lambda^*}{2} H} s(\varphi_e^{*}+sS_e)  
-s\Lambda\tfrac H2 e^{-\frac{\lambda^*}{2} H} s(\varphi_e^{*}+sS_e)   
=   s^2 \cR_3 ,
\end{align*}
where 
					\begin{align*}
\cR_{1} 
& :=  k_eh\left(|\nabla \{ s (\varphi_{v}^{*} +s W)+(\lambda^{*}+s \Lambda ) H \}|\right)e^{-\frac{\lambda^{*}+s \Lambda}{2}H}\frac1s  (\varphi_{e}^{*}+ sS_{e})
\\
& \quad  - k_{e}h(|\lambda^{*} \nabla H|)e^{-\frac{\lambda^{*}}{2}H} \frac1s  (\varphi_{e}^{*}+s S_{e})
\\
&\quad - k_{e} s \Lambda |\nabla H| h'(|\lambda^{*} \nabla H|)e^{-\frac{\lambda^{*}}{2}H} \frac1s  (\varphi_{e}^{*}+s S_{e})
+k_{e} s\Lambda h\left(|\lambda^{*} \nabla  H|\right) \frac{H}{2} e^{-\frac{\lambda^{*}}{2}H}\frac1s  (\varphi_{e}^{*}+s S_{e}),
\\
\cR_{2} 
&:= \nabla (\varphi_{v}^{*} +s W) \cdot \nabla (\varphi_{e}^{*}+s S_{e}) 
-  \left\{ \frac{\lambda^{*}+s \Lambda}{2}  \nabla (\varphi_{v}^{*} +s W) \cdot \nabla H -\Delta(\varphi_{v}^{*} +s W) \right\} (\varphi_{e}^{*}+s S_{e})
\\
& \quad  - \left\{\frac14|(\lambda^{*}+s \Lambda)\nabla H|^{2} -h\left(|\nabla \{ s(\varphi_{v}^{*} +s W)+(\lambda^{*}+s \Lambda) H\} |\right)\right\}\frac1s(\varphi_{e}^{*}+s S_{e})
\\
& \quad - g(|\lambda^{*} \nabla H|) \frac1s (\varphi_{e}^{*}+s S_{e}) - s \Lambda |\nabla H|g'(|\lambda^{*} \nabla H|) \frac1s (\varphi_{e}^{*}+s S_{e}),
\\
\cR_{3}  
&:=-e^{-\frac{\lambda^{*}+s \Lambda}{2}H}\frac1s (\varphi_{e}^{*}+s S_{e}) + e^{-\frac{\lambda^{*}}{2}H}\frac1s (\varphi_{e}^{*}+s S_{e})  - s\Lambda \frac{H}{2} e^{-\frac{\lambda^{*}}{2}H} \frac1s (\varphi_{e}^{*}+s S_{e}) . 
\end{align*}

Dividing by $s^2$, this system can be written more conveniently as 
			\begin{align}      \label{bifurcationEq1}
\cL
\begin{bmatrix}
S_{i } \\ S_{e} \\ W
\end{bmatrix}
+\Lambda \cK 
\begin{bmatrix}
\varphi_{i}^{*}+s S_{i} \\ \varphi_{e}^{*}+s S_{e} \\ \varphi_{v}^{*} +s W
\end{bmatrix}
+ \cN[\nabla(\varphi_{v}^{*} +s W { + \Lambda H})] 
\begin{bmatrix}
\varphi_{i}^{*}+s S_{i} \\ \varphi_{e}^{*}+s S_{e} \\ \varphi_{v}^{*} +s W
\end{bmatrix}
= \cR[s,\Lambda, S_{e},  W],
\end{align}
where 
$\cR:=(\cR_{1},\cR_{2},\cR_{3})$, 
						\begin{align*}  
\cK&:= 
\begin{bmatrix}
\,  0 & -k_{e}|\nabla H| h'(|\lambda^{*} \nabla H|)e^{-\frac{\lambda^{*}}{2}H} +k_{e}  h\left(|\lambda^{*} \nabla  H|\right) \frac{H}{2} e^{-\frac{\lambda^{*}}{2}H} & 0 \,
\\
\, 0 & - |\nabla H|g'(|\lambda^{*} \nabla H|) & 0 \,
\\
\, 0 & -\frac{H}{2} e^{-\frac{\lambda^{*}}{2}H}  & 0 \,
\end{bmatrix}  ,    
\end{align*}
and 
					\begin{align*}
\cN[\nabla u] & :=
\begin{bmatrix}
k_i {\rm div} \left\{ \, \cdot \, \nabla u \right\} & 0 & 0
\\
0 & 0 & 0
\\
0 & 0 & 0
\end{bmatrix}.
\end{align*}

As we shall see later on, it is important to treat the bilinear $\cN$ term separately  because it comes from the hyperbolic equation so that we cannot treat it as a reminder as we did with the nonlinear terms in $\cR$,  due to the loss of regularity of the mapping $v\to\cN[\nabla u]v$.
Now we decompose this equation \eqref{bifurcationEq1} by separately applying the projections $P$ and $Q$, in order to implement the Liapunov--Schmidt procedure.  
As shown in subsection \ref{LO1}, we know that 
\begin{equation*}
R(\cL)=\{(f_{i},f_e,f_v) \in Y^{0} \, ; \, \langle f_e,\varphi_e^{*} \rangle=0 \}.
\end{equation*}
The projection $P$ is the $L^{2}$ projection of the second row (zero on the first and third rows) of  \eqref{bifurcationEq1} onto $\varphi_{e}^{*}$.   
Applying $P$ to  \eqref{bifurcationEq1} we have
\begin{gather*}
\kappa'(\lambda^{*}) \Lambda 
= \left(-\int_{\Omega}  |\nabla H|g'(|\lambda^{*} \nabla H|) (\varphi_{e}^{*})^{2}  dx \right) \Lambda 
=  \big\langle \cR_{2}, \varphi_{e}^{*}  \big\rangle.
\end{gather*}
				Since 
$\kappa'(\lambda^*)\ne 0$, $\Lambda$ is determined by $\cR_2$.  
Furthermore we apply $Q$ to \eqref{bifurcationEq1} and then rearrange to obtain
\begin{align*}
&\cL
\begin{bmatrix}
S_{i } \\ S_{e} \\ W
\end{bmatrix}
+\Lambda Q (\cK { + \cN[\nabla H] }) 
\begin{bmatrix}
\varphi_{i}^{*} \\ \varphi_{e}^{*} \\ \varphi_{v}^{*} 
\end{bmatrix}
+ s Q \cN[\nabla(\varphi_{v}^{*} +s W { + \Lambda H})] 
\begin{bmatrix}
 S_{i} \\  S_{e} \\  W
\end{bmatrix}
\\
&= Q\cR
-s \Lambda Q \cK  
\begin{bmatrix}
 S_{i} \\  S_{e} \\  W
\end{bmatrix}
- Q { \cN[\nabla(\varphi_{v}^{*} +s W)] }
\begin{bmatrix}
\varphi_{i}^{*} \\ \varphi_{e}^{*} \\ \varphi_{v}^{*}
\end{bmatrix},
\end{align*}
where we have used { $QR(\cL)=R(\cL)$}.
Thereby we arrive at the infinite-dimensional equation  
\begin{gather}\label{bifurcationEq2}
M(s,{ \Lambda},U)
\begin{bmatrix}
\Lambda \\ U
\end{bmatrix}
= F(s,\Lambda, U),
\end{gather}
				where we define 
\begin{gather}  
U:=
\begin{bmatrix}
 S_{i} \\  S_{e} \\  W
\end{bmatrix},
\quad
M(s,{ \Lambda},U):=
\begin{bmatrix}
\kappa'(\lambda^{*}) & \bm{0}
\\
Q (\cK { + \cN[\nabla H] })  \varPhi^{*} & \cL + sQ\cN[\nabla(\varphi_{v}^{*} +s W { + \Lambda H})] 
\end{bmatrix},
\\
F(s,\Lambda, U):=
\begin{bmatrix}
\langle \cR_{2}[s, \Lambda, S_{e}, W], \varphi_{e}^{*} \rangle \\ Q\cR[s,\Lambda, S_{e}, W] - s \Lambda Q \cK  U - Q { \cN[\nabla(\varphi_{v}^{*} +s W)] } \varPhi^{*}
\end{bmatrix}.   \label{F(s)}
\end{gather}  

\subsection{Proof of local bifurcation} 
In order to prove the local bifurcation, it suffices to find  $(\Lambda,U)$ which satisfies \eqref{bifurcationEq2} for each $s$ with $|s| \ll 1$.
Recall the spaces $X^{\ell}$ and $Y^{\ell}$.  
We linearize the left side of  \eqref{bifurcationEq2} by writing  
\begin{gather}\label{bifurcationEq3}
M(s,\hat{\Lambda},\hat{U})
\begin{bmatrix}
\Lambda \\ U
\end{bmatrix}
= F(s,\Lambda, U),
\end{gather}
where for a given constant $\hat{\Lambda}$ and a given function $\hat{U} \in QX^{1}$, the linear map $M(s,\hat{\Lambda},\hat{U})$ is defined so that 
\begin{gather*}
M(s,\hat{\Lambda},\hat{U}): 
\sD \to \mathbb R \times QY^{1}.
\end{gather*}
The domain 
of $M(s,\hat\Lambda,\hat{U})$ is given as $\sD = M(s,\hat{\Lambda},\hat{U})^{-1}( \mathbb R \times QY^{1})$, so that $M(s,\hat{\Lambda},\hat{U})$ maps onto $ \mathbb R \times QY^{1}$.   The domain of $M$ for the hyperbolic part (the first of the three components)  is not $C_{0}^{1+\ell,\alpha}$ but only $C_0^{\ell,\alpha}$. 
Note that the domain of $M(s,\hat{\Lambda},\hat{U})$ depends on $(\hat{\Lambda},\hat{U})$.  
The technique will avoid the difficulty that the hyperbolic equation does not regularize.  
In \cite{KN} it was employed to study  the compressible Navier-Stokes equation.

\begin{lem}\label{inverseM1}
For a given constant $K>0$, then there exists $\eta_{1}$ such that if $|s| \leq \eta_{1}$, $\hat{U} \in QX^{1}$, and $|\hat{\Lambda}|+\|\hat{U}\|_{X^{1}} \leq K$, then the following (i) and (ii) hold:
\begin{enumerate}[(i)]
\item $\cL + sQ\cN[\nabla(\varphi_{v}^{*} +s \hat{W}+ \hat{\Lambda} H)]$ has a bounded inverse from  $QY^{\ell}$ to $QX^{\ell}$, and there exists a positive $C_{1}$ depending on $\eta_{1}$ and $K$ such that 
$$ 
\| (\cL + sQ\cN[\nabla(\varphi_{v}^{*} +s \hat{W} + \hat{\Lambda} H)])^{-1} U \|_{X^{\ell}} \leq C_{1} \| U \|_{Y^{\ell}}
$$
for all $U \in QY^\ell$ and $\ell=0,1.$
\item $M(s,\hat{\Lambda},\hat{U})$ has a bounded inverse from 
$\real\times QY^{\ell}$ to $\mathbb R \times QX^{\ell}$
, and there exists a positive constant $C_{2}$ depending on $\eta_{1}$ and $K$  such that
\begin{gather*}
\left\| M(s,\hat{\Lambda},\hat{U})^{-1} 
\begin{bmatrix}
\Lambda \\ U
\end{bmatrix} 
\right\|_{\mathbb R \times X^{\ell}} 
\leq C_{2}
\left \|\begin{bmatrix}
\Lambda \\ U
\end{bmatrix} 
\right\|_{\mathbb R \times Y^{\ell}}, \quad \ell=0,1.
\end{gather*}
\end{enumerate}
\end{lem}
\begin{proof}
In order to prove (i), it suffices to show that 
\begin{gather}\label{cLeqes1}
\|\widetilde{U}\|_{X^{\ell}} \leq C_{1} \|\widetilde{F}\|_{Y^{\ell}}
\end{gather}
for any $\widetilde{U}={}^{t}(\widetilde{S}_{e},\widetilde{S}_{i},\widetilde{W}) \in QX^{\ell}$ which solves 
\begin{gather}\label{cLeq1}
\big(\cL + sQ\cN[\nabla(\varphi_{v}^{*} +s \hat{W}+ \hat{\Lambda} H)]\big) \widetilde{U} = \widetilde{F}:= 
\begin{bmatrix}
\widetilde{f}_{i} \\ \widetilde{f}_{e} \\ \widetilde{f}_{v}
\end{bmatrix}
\in QY^{\ell}.
\end{gather}
First we solve the second row of \eqref{cLeq1}, i.e.
\begin{gather}  \label{tildeSXXX}
-\Delta \widetilde{S}_{e} -g(|\lambda^{*} \nabla H|) \widetilde{S}_{e} =\widetilde{f}_{e}, 
\end{gather}
where $\widetilde{f}_{e}$ satisfies $\langle \widetilde{f}_e,\varphi_e^{*} \rangle=0$.
Recall that zero is the lowest eigenvalue of the self-adjoint 
operator $A$ on $L^2(\Omega)$ with domain $H^2(\Omega)\cap H^1_0(\Omega)$ as well as
$\varphi_e^{*}$ is the corresponding positive eigenfunction.  
Therefore standard elliptic theory ensures the solvability of \eqref{tildeSXXX} and the estimate
$\|\widetilde{S}_{e}\|_{C^{2+\ell,\alpha}} \leq C  \|\widetilde{F}\|_{Y^{\ell}}$,
where $C>0$ is a constant independent of $\widetilde{F}$.
Next, by taking $|s|$ small enough  subject to $K$ and applying Lemma \ref{rho_iLemma} with \eqref{Hasp1}, we solve the first row of \eqref{cLeq1}, i.e.
\begin{gather*}
k_{i} {\rm div} ( \widetilde{S}_{i} \lambda^{*} \nabla  H ) 
+ k_{i} {\rm div} ( \widetilde{S}_{i} \nabla (s \varphi_{v}^{*} +s^{2} \hat{W}+ s \hat{\Lambda} H) ) 
 = k_{e}h(|\lambda^{*} \nabla H|)e^{-\frac{\lambda^{*}}{2}H}\widetilde{S}_{e}+\widetilde{f}_{i}.
\end{gather*}
Furthermore, there holds that
$\|\widetilde{S}_{i}\|_{C^{1+\ell,\alpha}} \leq C  \|\widetilde{F}\|_{Y^{\ell}}$.
By standard elliptic theory it is easy to solve the third row of \eqref{cLeq1}, i.e.
\begin{gather*}
\Delta \widetilde{W} = - \widetilde{S}_{i} + e^{-\frac{\lambda^{*}}{2}H}  \widetilde{S}_{e}+\widetilde{f}_{v}
\end{gather*}
and obtain
$\|\widetilde{W}\|_{C^{3+\ell,\alpha}} \leq C  \|\widetilde{F}\|_{Y^{\ell}}$.
Thus \eqref{cLeqes1} holds. 

Assertion (ii) follows from (i) since $\kappa'(\lambda^{*})>0$ holds due to the definition of the sparking voltage $\lambda^*$ and also due to the inverse of $M$ being explicitly written as 
\begin{align*}
M(s,\hat{\Lambda},\hat{U})^{-1} =\begin{bmatrix}
1/\kappa'(\lambda^{*}) & \bm{0}
\\  \widetilde{M}_{21} & \big(\cL + sQ\cN[\nabla(\varphi_{v}^{*} +s \hat{W}+\Lambda H)]\big)^{-1}
\end{bmatrix},
\end{align*}
where
\begin{align*}
\widetilde{M}_{21}:=-\big(\cL + sQ\cN\big[\nabla(\varphi_{v}^{*} +s \hat{W}+\Lambda H)\big]\big)^{-1}
 \big(Q (\cK + \cN[\nabla H] )  \varPhi^{*} \big)/\kappa'(\lambda^{*}).
\end{align*}The proof is complete.
\end{proof}

\begin{lem}\label{inverseM2}
Recall the definition \eqref{F(s)} of $F$.
For a given constant $K>0$, there exists $\eta_{2}$ such that if $|s| \leq \eta_{2}$, $U_{j} \in QX^{1}$, and $|\Lambda_{j}|+\|U_{j}\|_{X^{1}} \leq K$ for $j=1,2$, there holds the following estimates:
\begin{gather}
\| F(s,\Lambda_{1}, U_{1}) - F(0,0, 0)  \|_{\mathbb R \times Y^{1}} \leq C_{K}|s|,
\label{Fes1}\\
\| F(s,\Lambda_{1}, U_{1}) - F(s,\Lambda_{2}, U_{2})  \|_{\mathbb R \times Y^{0}} \leq C_{K}|s|(|\Lambda_{1}-\Lambda_{2}|+\|U_{1}-U_{2}\|_{X^{0}}),
\label{Fes2}
\end{gather}
where $C_{K}>0$ is a constant depending on $K$ but independent of $s$.
\end{lem}
\begin{proof}
We show only the estimate \eqref{Fes2}, since \eqref{Fes1} is very similar.
It is clear that 
\begin{gather*}
\| s \Lambda_{1} Q \cK  U_{1} - s \Lambda_{2} Q \cK  U_{2} \|_{Y^{0}} \le C K |s| (|\Lambda_{1} - \Lambda_{2}|+\|U_{1}-U_{2}\|_{Y^{0}}),
\\
\| Q { \cN[\nabla(\varphi_{v}^{*} +s W_{1})] } \varPhi^{*} - 
{ Q\cN[\nabla(\varphi_{v}^{*} +s W_{2})] } \varPhi^{*}\|_{Y^{0}} 
 \le C|s| \|W_1-W_2\|_{C^{3,\alpha}}    
\le C |s| \|U_{1}-U_{2}\|_{X^{0}}.   
\end{gather*}
What is left  is to estimate the difference $\cR[s, \Lambda_{1}, S_{e1}, W_{1}] - \cR[s, \Lambda_{2}, S_{e2}, W_{2}] $.
We treat only the difference of the first row, since the other two rows can be handled similarly. 
We write $\cR_{1}[s, \Lambda, S_{e}, W] $ as
\begin{align*}
\cR_{1}&=
k_e\left\{ h\left(|\nabla \{ s (\varphi_{v}^{*} +s W)+(\lambda^{*}+s \Lambda ) H \}|\right) -  h\left(|\nabla \{ (\lambda^{*}+s \Lambda ) H \}|\right) \right\}
e^{-\frac{\lambda^{*}+s \Lambda}{2}H}\tfrac1s  (\varphi_{e}^{*}+ sS_{e})
\\
& \quad + \left\{ \widetilde{h}(x,s\Lambda) - \widetilde{h}(x,0) - s \Lambda (\partial_{\sigma }\widetilde{h})(x,0)  \right\} \tfrac1s  (\varphi_{e}^{*}+ sS_{e}),
\end{align*}
where $\widetilde{h}(x,\sigma):= h\left(|\nabla \{ (\lambda^{*}+\sigma ) H \}|\right) e^{-\frac{\lambda^{*}+\sigma}{2}H}$.
By Taylor's theorem for small $|s|$ we expand 
\begin{gather*}
h\left(|\nabla \{ s (\varphi_{v}^{*} +s W)+(\lambda^{*}+s \Lambda ) H \}|\right) -  h\left(|\nabla \{ (\lambda^{*}+s \Lambda ) H \}|\right) = O(s) (\varphi_{v}^{*} +s W),
\\
\widetilde{h}(x,s\Lambda) - \widetilde{h}(x,0) - s \Lambda (\partial_{\sigma }\widetilde{h})(x,0) =O(s^{2})\Lambda^{2}.
\end{gather*}
Furthermore,  in $\cR_{1}$ all $\Lambda_{j}$ and $U_{j}$ appear as products with $s$. 
So a direct computation yields
\begin{gather*}
\| \cR_{1}[s, \Lambda_{1}, S_{e1}, W_{1}] - \cR_{1}[s, \Lambda_{2}, S_{e2}, W_{2}] \|_{Y^{0}} \leq C_{K} |s| (|\Lambda_{1} - \Lambda_{2}|+\|U_{1}-U_{2}\|_{Y^{0}}).
\end{gather*}
The proof is complete.
\end{proof}

\begin{proof}[Proof of Theorem \ref{thm3}]
We define a sequence $(\Lambda^{n},U^{n})$ as follows.   For $n=1$, it is a solution of 
\begin{gather*}
M(0,0,0)
\begin{bmatrix}
\Lambda^{1} \\ U^{1}
\end{bmatrix}
= F(0,0, 0).
\end{gather*}
For $n \geq 2$, it is a solution of 
\begin{gather}\label{iteration1}
M(s,\Lambda^{n-1},U^{n-1})
\begin{bmatrix}
\Lambda^{n} \\ U^{n}
\end{bmatrix}
= F(s,\Lambda^{n-1}, U^{n-1}).
\end{gather}

First we show that the sequence is well-defined and uniformly bounded in $n$.  
From Lemma \ref{inverseM1}, we see that $(\Lambda^{1},U^{1})$ is well-defined and satisfies
\begin{gather*}
|\Lambda^{1}|+\|U^{1}\|_{X^{1}} \leq C_{2}\|F(0,0, 0)\|_{\mathbb R \times Y^{1}}.
\end{gather*}
We set $K:=2C_{2}\|F(0,0, 0)\|_{\mathbb R \times Y^{1}}$ and apply Lemmas \ref{inverseM1} and \ref{inverseM2} inductively as follows.
Suppose that $|\Lambda^{n-1}|+\|U^{n-1}\|_{X^{1}} \leq  K$ holds.
Then $(\Lambda^{n},U^{n})$ is well-defined due to Lemma \ref{inverseM1}.
				We write 
\begin{gather*}
M(s,\Lambda^{n-1},U^{n-1})
\begin{bmatrix}
\Lambda^{n} \\ U^{n}
\end{bmatrix}
=  F(0,0, 0) + \{F(s,\Lambda^{n-1}, U^{n-1})-F(0,0, 0)\}.
\end{gather*}
Applying $M^{-1}$ to this, treating the first term and the bracketed term separately using Lemmas \ref{inverseM1} and \ref{inverseM2}, and letting $s$ be sufficiently small, we have
\begin{gather*}
|\Lambda^{n}|+\|U^{n}\|_{X^{1}} \leq \frac{K}{2} + C_2KC_{K} |s|  \leq  K.
\end{gather*}
Thus 
$(\Lambda^{n},U^{n})$ is well-defined for all $n\ge 1$ and satisfies 
\begin{gather}\label{bound1}
|\Lambda^{n}|+\|U^{n}\|_{X^{1}} \leq  K.
\end{gather}

Our main task is to show that $\{(\Lambda^{n},U^{n})\}$ is a Cauchy sequence in $\mathbb R \times X^{0}$.  
Note the decrease in regularity.
It is straightforward to see from  \eqref{iteration1}  and the definition of $M$ that
\begin{align*}
&M(s,\Lambda^{n},U^{n})
\begin{bmatrix}
\Lambda^{n+1} - \Lambda^{n} \\ U^{n+1} - U^{n}
\end{bmatrix}
+
\begin{bmatrix}
0 & \bm{0}
\\
\bm{0} & s Q\cN\big[\nabla ( s(W^{n} - W^{n-1}) + (\Lambda^{n}-\Lambda^{n-1})H) \big] 
\end{bmatrix}
\begin{bmatrix}
\Lambda^{n} \\ U^{n}
\end{bmatrix}
\\
&=
F(s,\Lambda^{n}, U^{n})-F(s,\Lambda^{n-1}, U^{n-1}).
\end{align*}
Applying $(M(s,\Lambda^n,U^n))^{-1}$ to this, using Lemmas \ref{inverseM1} with $\ell=0$ and \ref{inverseM2}, and letting $s$ be sufficiently small, we have 
\begin{align*}
&|\Lambda^{n+1}-\Lambda^{n}|+\|U^{n+1}-U^{n}\|_{X^{0}} 
\\
&\leq C_2C_{K} |s| (|\Lambda^{n}-\Lambda^{n-1}|+\|U^{n}-U^{n-1}\|_{X^{0}} )
\\
&\qquad  +  C_2\|s\cN \left[ \nabla\left\{s(W^n-W^{n-1}) + (\Lambda^n-\Lambda^{n-1}) H\right\}\right] U^n \|_{Y^0}  
\\
&\leq C_2C_{K} |s| (|\Lambda^{n}-\Lambda^{n-1}|+\|U^{n}-U^{n-1}\|_{X^{0}} ) 
\\
&\qquad  +  C_2|s|\left\{ |s|k_i\| {\rm div} \{ S_i^n \nabla(W^n-W^{n-1})\} \|_{ C^{1,\alpha}}  
+k_i |\Lambda^n-\Lambda^{n-1}|\| {\rm div} ( S_i^n \nabla H) \|_{C^{1,\alpha}}   \right\}
\\
&\leq \tilde C_{K} |s| (|\Lambda^{n}-\Lambda^{n-1}|+\|U^{n}-U^{n-1}\|_{X^{0}} )
\\
&\leq \tfrac12 (|\Lambda^{n}-\Lambda^{n-1}|+\|U^{n}-U^{n-1}\|_{X^{0}} )
\end{align*}
because of
\begin{align*}
 \|{\rm div} \{S_i^n \nabla(W^n-W^{n-1}) \}\|_{C^{1,\alpha}}  
 \lesssim \|S_i^n\|_{C^{2,\alpha}}  \|W^n-W^{n-1}\|_{C^{3,\alpha}} 
 \lesssim  \|U^n\|_{X^1} \|U^n-U^{n-1}\|_{X^0} 
\end{align*}
and the bound \eqref{bound1}.   
Hence $\{(\Lambda^{n},U^{n})\}$ is a Cauchy sequence in $\mathbb R \times X^{0}$.
For each small $s$, there exists $(\Lambda,U) \in \mathbb R \times X^{0}$ such that 
$(\Lambda^{n},U^{n})$ converges to $(\Lambda,U)$ in $\mathbb R \times X^{0}$ as $n \to \infty$.
It is easy to see that this limit $(\Lambda,U)$ satisfies the infinite-dimensional equation \eqref{bifurcationEq2}.
The proof is complete.
\end{proof}


\section{Global bifurcation for the radial cases}     \lb{S5}
In this section, our purpose is to prove Theorem \ref{thm4}.  The proof has two parts.  
First we will apply a general global bifurcation theorem and then we will make use of the positivity of the densities.  
Here we are forced to assume that the domain $\Omega$ is either a circular annulus $\{(x,y) \in \mathbb R^{2}  ;   r_{1}^{2}<r^{2}:=x^{2} + y^{2}<r_{2}^{2} \}$ or a spherical shell $\{(x,y,z) \in \mathbb R^{3}  ;   r_{1}^{2}<r^{2}:=x^{2} + y^{2} + z^{2}<r_{2}^{2} \}$.
The theory of global bifurcation goes back to Rabinowitz \cite{Ki1,Rab}
 using topological degree.  
 A different version using analytic continuation 
 goes back to Dancer \cite{BST1,Da1}. 
The specific version that is most convenient to use here is Theorem 6 in \cite{CS1}, 
which is the following.  

\begin{thm}[\! \cite{CS1}]\lb{Global0}
Let $X$ and $Y$ be Banach spaces,
$\sO$ be an open subset of ${\mathbb R}\times X$ and $\sF:\sO \to Y$ be a real-analytic function. 
Suppose that
\begin{enumerate}[{(H}1{)}]
\item $(\lambda,0)\in \sO$ and $\sF(\lambda,0)=0$ for all $\lambda \in \mathbb R$;
\item for some $\lambda^* \in \mathbb R$, $N(\partial_u\sF(\lambda^*,0))$ and 
$Y\backslash R(\partial_u\sF(\lambda^*,0))$ are one-dimensional, 
with the null space generated by $u^*$, which satisfies  the transversality condition
\[
\partial_{\la,u}^2\sF(\lambda^*,0)(1,u^*)\notin R(\partial_u\sF(\lambda^*,0)),  
\]
where $\partial_u$ and $\partial_{\la,u}^2$ mean Fr\'echet derivatives for $(\la,u) \in \sO$,
and $N(\sL)$ and $R(\sL)$ denote the null space and range of a linear operator $\sL$ between
two Banach spaces;
\item $\partial_u\sF(\lambda,u)$ is a Fredholm operator of index zero, for any $(\la,u) \in \sO$
that satisfies the equation $\sF(\lambda,u)=0$;
\item for some sequence $\{\sO_j\}_{j\in \mathbb N}$ of bounded closed subsets of $\sO$ with
$\sO=\cup_{j\in \mathbb N} \sO_j$, 
the set $\{(\la,u) \in \sO ;\sF(\lambda,u)=0\}\cap \sO_j$ is compact for each $j\in\mathbb N$.
\end{enumerate}

Then there exists in $\sO$ a continuous curve 
${\sK}=\{(\la(s),u(s));s \in \mathbb R\}$
of $\sF(\lambda,u)=0$ such that:
\begin{enumerate}[{(C}1{)}]
\item $(\la(0),u(0))=(\la^*,0)$;
\item $u(s)=su^*+o(s)$ in $X$  as $s \to 0$;
\item there exists a neighborhood $\sW$ of $(\la^*,0)$
and $\ve > 0$ sufficiently small such that
\begin{equation*}
\{(\la,u)\in \sW ; u\neq 0 \text{ and } \sF(\la,u)=0 \}
=\{(\la(s),u(s)) ; 0<|s|<\ve\};
\end{equation*}
\item $\sK$ has a real-analytic reparametrization locally around each of its points;
\item one of the following alternatives occurs:
\begin{enumerate}[(I)]
\item for every $j\in\mathbb N$, there exists $s_j>0$ such that $(\la(s),u(s))\notin \sO_j$
for all $s \in \mathbb R$ with $|s|>s_j$;
\item there exists $T>0$ such that
$(\la(s),u(s))=(\la(s+T),u(s+T))$ for all $s \in \mathbb R$.
\end{enumerate}
\end{enumerate}
Moreover, such a curve of solutions of $\sF(\lambda,u)=0$
having the properties (C1)-(C5) is unique (up to reparametrization).
\end{thm}

To apply the theorem to our situation, we write the stationary system as
\begin{equation}\lb{sp0}
\sF_j(\lambda,\rho_i,R_e,V)=0 \qu \text{for $j=1,2,3,$}  
\end{equation}
where $R_{e}:=\ro_e e^{\frac{\lambda}{2}H}$ and we denote 
$\cF = (\cF_1,\cF_2, \cF_3) = (\sG_1,\sG_2, \sG_3)$, where (as above)
\begin{align*}
\sF_1(\lambda,\rho_i,R_e,V)
:=&k_i \nabla \cdot \left\{ \rho_i \nabla(V+\lambda H)\right\}
-k_eh\left(|\nabla (V+\lambda H)|\right)e^{-\frac{\lambda}{2}H}R_e,
\\
\sF_2(\lambda,\rho_i,R_e,V)
:=&-\Delta R_e 
-\nabla V \cdot \nabla R_{e}
+\left\{\frac{\lambda}{2} \nabla V \cdot \nabla H
-\Delta V
+\frac{\lambda^2}{4}|\nabla H|^{2}
-h\left(|\nabla (V+\lambda H)|\right)\right\}R_e,
\\
\sF_3(\lambda,\rho_i,R_e,V)
:=&\Delta V-\rho_i+e^{-\frac{\lambda}{2}H}R_e.
\end{align*}
Furthermore, we define the two spaces of radial functions: 
\begin{align*}
X: \ &\rho_i \in \{f\in C^1_{r}(\overline{\Omega});f|_{\cA}=0\}, \ \ R_e \in \{f\in C^2_{r}(\overline{\Omega});f|_{\cC}=f|_{\cA}=0\},
\\
& V \in \{f\in C^3_{r}(\overline{\Omega});f|_{\cC}=f|_{\cA}=0 \};
\\
Y: \ &\sF_1 \in C^0_{r}(\overline{\Omega}), \ \ \sF_2 \in C^0_{r}(\overline{\Omega}),
\ \ \sF_3 \in C^1_{r}(\overline{\Omega}),
\end{align*}  
where $C^k_{r}(\overline{\Omega}):=\{f \in C^{k}(\overline{\Omega}) ; \text{$f$ is radial} \}$ for $k=0,1,2,3$.
Note that if $f$ and $\bm{u}=(u_{1},\ldots,u_{d})$ are radial for $d=2,3$, then the following holds:
\begin{gather}\label{dop1}
\Delta f= \frac{1}{r^{d-1}} \partial_{r} (r^{d-1} \partial_{r} f ),
\quad
\nabla f= \frac{x}{r}\partial_{r} f, 
\quad
\nabla \cdot \bm{u}= \frac{1}{r^{d-1}} \partial_{r} (r^{d-1} u_{1}),
\quad
\text{where $r^{2}:= \sum_{i=1}^{d} |x_{i}|^{2}$.}
\end{gather}
We also use the sets 
\begin{align*}  
\sO:=&\{(\lambda,\rho_i,R_e,V)\in (0,\infty)\times X;\ |\nabla (V+\lambda H)|>0\} \ 
=\ \bigcup_{j\in \bbn} \sO_j ,\text{ where } 
\\
\sO_j:=&\{(\lambda,\rho_i,R_e,V)\in (0,\infty)\times X; \ 
\lambda+\|(\rho_i,R_e,V)\|_{X}\leq j ,  \  
\lambda\geq \tfrac1j,  \ |\nabla (V+\lambda H)| \geq \tfrac1j\}.
\end{align*}
Note that $\sO$ is an open set and each $\sO_j$ is a closed bounded subset of $\sO$.   

\subsection{Application of Theorem \ref{Global0}}
In this subsection, we apply Theorem \ref{Global0} to our stationary problem. 
Hypothesis (H1) is obvious.  
The conditions (H2)--(H4) are validated in Lemmas \ref{null1}--\ref{cpt1} below, respectively.

\begin{lem}\lb{null1} 
Let $\lambda^*>0$ be defined by \er{bp1}, 
and $\sL= \partial_{(\rho_i,R_e,V)}\sF(\lambda^{*},0,0,0)$ be the linearized operator around a trivial solution. 
Then 

(i) The nullspace $N(\sL)$ is one-dimensional and generated by $u^*=(\varphi_i^*,\varphi_e^*,\varphi_v^*)$,
where $\varphi_i^*$ and $\varphi_e^*$ are positive functions. 
 
(ii) The quotient space $Y\backslash R(\sL)$ is one-dimensional. 

(iii) The following transversality condition is valid: 
\begin{gather*}
\partial_\la \partial_{(\rho_i,R_e,V)}  \sF(\lambda^{*},0,0,0)[1,\varphi_i^*,\varphi_e^*,\varphi_v^*]
\notin R(\sL).
\end{gather*}
\end{lem}

\begin{proof}
It is clear that $\sL=(\sL_{1},\sL_{2},\sL_{3})$, where
\begin{align*}
\sL_{1} (S_i,S_e)         
&:=  k_i \nabla (\lambda^{*} H) \cdot \nabla S_i  
-k_eh\left(|\nabla (\lambda^{*} H)|\right)e^{-\frac{\lambda_{*}}{2}H}S_e,
\notag \\
\sL_{2} (S_e)         
&:=- \Delta S_e -g(|\lambda^{*} \nabla H|)S_e,
\notag \\
\sL_{3} (S_i,S_e,W)       
&:=\Delta W - S_i + e^{-\frac{\lambda^{*}}{2}H}S_e.
\end{align*}

(i). First we will show that ${\rm dim} \, N(\sL)=1$.
The domain of the second operator $\sL_{2}$ is independent of $R_{i}$ and $W$.
We recall that $\kappa(\lambda^{*})=0$, where the stability index $\kappa(\cdot)$ was defined in \eqref{kap0}.
As mentioned earlier, zero is the lowest eigenvalue value of the self-adjoint 
operator $A$ on $L^2(\Omega)$ with domain $H^2(\Omega)\cap H^1_0(\Omega)$ and it is simple. 
We denote by $\varphi_e^{*} \in H^2(\Omega)\cap H^1_0(\Omega)$ 
the corresponding normalized positive eigenfunction.
Then all the solutions of the equation $\sL_{2}=0$
are $A_e:=\{c\varphi_e^{*};\, c \in \mathbb R\}  \subset  H^2(\Omega)\cap H^1_0(\Omega)$.
But  $\varphi_e^{*}$ belongs to $C^2_{r}(\overline{\Omega})$ 
since $\Omega$ is either a circular annulus or a spherical shell.
Therefore, $N(\sL_{2})=   A_e$.     

Next we observe that the first linear equation $\sL_{1}=0$
with $S_e=c\varphi_e^{*}$ has a unique solution $S_i=c\varphi_i^{*} \in \{f\in C^1_{r}(\overline{\Omega});f|_{\cC}=0\}$,
where $\varphi_i^{*}$ can be written explicitly in terms of $\varphi_e^{*}$.
For instance, in  case the anode $\cA=\{r=r_{1}\}$ and the cathode $\cC=\{r=r_{2}\}$, 
we have the solution 
\begin{equation}  \label{varphi_i}
 \varphi_i^{*} (r)=\frac{k_e}{k_i}r^{1-d}\{ \partial_r (\lambda^{*} H)(r)\}^{-1}
\int_{r_{1}}^r s^{d-1} h(|\partial_r (\lambda^{*} H)(s)|) 
e^{-\frac{\lambda^{*}}{2}H(s)}\varphi_e^{*}(s)\,ds>0.
\end{equation}

The third linear equation $\cL_{3}=0$ with $(S_{i},S_e)=c(\varphi_i^{*},\varphi_e^{*})$ has a unique solution $W=c\varphi_v^{*} \in \{f\in C^3_{r}(\overline{\Omega});f|_{\cC}=f|_{\cA}=0 \}$.
Thus we conclude that $N(\sL)$ is generated by $u^*=(\varphi_i^*,\varphi_e^*,\varphi_v^*)$, where $\varphi_i^*$ and $\varphi_e^*$ are positive functions.  Hence ${\rm dim} \, N(\sL)=1$.

(ii). Next we prove ${\rm dim}(Y\backslash R(\sL))=1$.
In fact we will show that
\begin{equation}\lb{range1}
R(\cL^{*})
=\{(f_i,f_e,f_v) \in Y \, ; \, \langle f_e,\varphi_e^{*} \rangle=0 \}.
\end{equation}
We begin by representing the range of the second operator $\sL_{2}$ as 
\begin{gather}
R(\sL_{2})=\{f_e \in C^0_{r}(\overline{\Omega}); \langle f_e,\varphi_e^{*} \rangle=0  \}.
\lb{reperesent1}
\end{gather}
In order to make use of duality as in Section \ref{LO1}, we temporarily convert to Hilbert spaces.   
Let $X^\bullet$ be the same as $X$ except that $C^k_{r}$ is replaced by $H^k_{r}:=\{f \in H^{k}(\Omega) ; \text{$f$ is radial} \}$ for $k=1,2,3$.  
Let $Y^\bullet$ be the same as $Y$ except that $C^k_{r}$ is replaced by $H^k_{r}$ for $k=0,1$.  
We denote $L^{2}_{r}=H^0_{r}$.
We define $\sL^{\bullet} :X^{\bullet} \to Y^{\bullet}$ as the unique linear extension 
of $\sL^{\bullet}$ of $\sL$ to $X^{\bullet}$.
Note that $\sL_{2}^{\bullet}$ is a self-adjoint operator and its domain is 
independent of $S_{i}$ and $W$. 
By standard operator theory from the fact that 
$N(\sL_{2}^{\bullet})=\langle \varphi_e^{*} \rangle$, we have 
$R(\sL_{2}^\bullet)=\{ f_e \in L^2_{r}(\Omega); \langle f_e,\varphi_e^{*} \rangle=0\}$.
Because $Y\subset Y^\bullet$, the condition $\langle f_e,\varphi_e^{*} \rangle=0$ is necessary
for the solvability of the problem $\sL_{2}^{\bullet}=f_{e} \in C^0_{r}(\overline{\Omega})$.
On the other hand, if $f_e \in \{f_{e} \in C^0_{r}(\overline{\Omega}) ; \langle f_e,\varphi_e^{*} \rangle=0 \}$,
we have a unique solution $S_{e} \in \{f \in H^{2}(\Omega) \cap H^{1}_{0}(\Omega) ; \text{$f$ is radial} \}$ to the problem $\sL_{2}^\bullet =f_{e}$.
But  $S_e \in C^2_{r}(\overline{\Omega})$ by standard elliptic estimates.
These facts 
lead to the representation \er{reperesent1}.
For any $f_{i} \in C^0_{r}(\overline{\Omega})$, 
the equation $\cL_{1}(S_i,S_e)=f_{i}$  has a unique solution $S_{i} \in \{f\in C^1_{r}(\overline{\Omega});f|_{\cC}=0\}$.
For any $f_{v} \in C^1_{r}(\overline{\Omega})$, the third equation 
$\cL_{3}(S_i,S_e,W)=f_{v}$ has a unique solution 
$S_{v} \in \{f\in C^3_{r}(\overline{\Omega});f|_{\cC}=f|_{\cA}=0 \}$.
Thus we have proven \eqref{range1}.

(iii). The transversality condition is verified as follows.
We observe by using \eqref{kap0} and \eqref{bp1} that 
\begin{align*}   
\langle \partial_\la \partial_{(\rho_i,R_e,V)}  \sF_{2}(\lambda^{*},0,0,0)[1,\varphi_i^*,\varphi_e^*,\varphi_v^*],\varphi_e^{*} \rangle
&=k_{e} \{\partial_\la \langle -\Delta \varphi_e^{*}  - g(|\lambda \nabla H|)\varphi_e^{*} ,  \varphi_e^{*}  \rangle \} |_{\lambda=\lambda^{*}}
\notag\\
&=k_{e} \kappa'(\lambda^{*}) < 0.
\end{align*}
Making use of \eqref{range1}, we arrive at the transversality condition.
\end{proof}

\begin{lem}\lb{index1} 
For arbitrary $(\lambda,\rho_i^0,R_e^0,V^0)\in {\sO}$, 
the Fr\'echet derivative 
$\sL^0 = \partial_{(\rho_i,R_e,V)} \sF(\lambda,\rho_i^0,R_e^0,V^0)$
is a linear Fredholm operator of index zero.
\end{lem}

\begin{proof}
Recalling \eqref{dop1}  we  see that
$\inf_{r} | \partial_{r} (V+\lambda H)| =\inf_{x} | \nabla (V+\lambda H)|> 0$
for any fixed choice of $(\lambda,\rho_i^0,R_e^0,V^0)$.
The linearized operator $\sL^0=(\sL_1^0,\sL_2^0,\sL_3^0)$ has the form 
\begin{align}
\sL_1^0 =& \sL_1^0(S_i,S_e,W) = 
k_{i} \partial_r (V^0 + \lambda H)\partial_{r} S_{i} + b_1S_i + b_2\partial_r^2 W + b_3S_e + b_4\partial_r W , 
\lb{sL1}\\
\sL_2^0 =& \sL_2^0(S_i,S_e,W) = 
- \partial_r^{2} S_e  + a_1\partial_r S_e + b_5 S_e + b_6\partial_r^2W + b_7\partial_rW  ,   
\lb{sL2}\\
\sL_3^0 =& \sL_3^0(S_i,S_e,W) = 
\partial_r^{2} W  + a_2\partial_r W + a_3S_i + a_4S_e ,
\lb{sL3}
\end{align}
where the coefficients $a_1,...,a_4$ belong to $C^1_{r}(\overline{\Omega})$
and the coefficients $b_1,...,b_7$ belong to $C^0_{r}(\overline{\Omega})$.

First we will prove that the linear operator $\sL^0$ satisfies the estimate 
\bqn  \label{ellest} 
C \|(S_i,S_e,W)\|_X  \le  \|\sL^0(S_i,S_e,W)\|_Y  +  \|(S_i,S_e,W)\|_Y 
\eqn
for all $(S_i,S_e,W)\in X$ and 
for some constant $C$ depending only on $(\lambda,\rho_i^0,R_e^0,V^0)$.  
Indeed, keeping in mind that  $|\partial_r V^0 + \lambda| \ge 1/j$, we see from \er{sL1} and \er{sL3} that 
\begin{align}
\|\partial_r S_i\|_{C^0_{r}} 
&=\|\{k_{i}\partial_r (V^0 + \lambda H)\}^{-1}
\notag\\
&\qquad \times
(b_1S_i + b_2 (-\sL_3^0  +a_{2}\partial_r W+a_3S_i+a_4S_e) + b_3S_e + b_4\partial_r W-\sL_1^0)\|_{C^0_{r}}
\notag\\
&\leq C(\|S_i\|_{C^0_{r}}+\|S_e\|_{C^0_{r}}+\|W\|_{C^1_{r}}+\|\sL_3^0\|_{C^0_{r}}+\|\sL_1^0\|_{C^0_{r}})
\notag\\
&\leq  C\|\sL^0(S_i,S_e,W)\|_Y  +  C\|(S_i,S_e,W)\|_Y.
\lb{estsL1}
\end{align}
By writing \er{sL2} as 
\[ \partial_r^2 S_e  - a_1\partial_r S_e=  -\sL_2^0  + b_5S_e +b_6(-\sL_3 +a_{2}\partial_r W+a_3S_i+a_4S_e) 
+a_7\partial_rW,  \]
 and  $a_1\partial_rS_e = \partial_r (a_1S_e) -(\partial_r a_1)S_e$, 
we see that 
\begin{gather}
 \|\partial_r^2S_e\|_{C^0_{r}}  \le  \|\sL^0(S_i,S_e,W)\|_Y  +  \|(S_i,S_e,W)\|_Y .  
\end{gather}
Finally, \er{sL3} leads to
\begin{align}
\|\partial_r^2 W\|_{C^1_{r}} 
=\|a_{2}\partial_r W + a_3S_i + a_4S_e-\sL_3^0\|_{C^1_{r}}
\leq  C\|\sL^0(S_i,S_e,W)\|_Y  +  C\|(S_i,S_e,W)\|_Y.
\lb{estsL3}
\end{align}
Combining  \er{estsL1}--\er{estsL3} yields  the estimate \er{ellest}.

The estimate \eqref{ellest} implies that the nullspace of $\sL^0$ has finite dimension and the range of $\sL^0$ is closed.  This is a general fact about linear operators; see for instance  \cite[Example 4 in $\S$3.12]{Zei1}.  
Therefore $\sL^0$ is a semi-Fredholm operator. 
Lemma~\ref{null1} ensures that at the bifurcation point the 
 nullspace of $\partial_{(\rho_i,R_e,V)} \sF(\lambda^{*},0,0,0)$ has dimension one and 
the codimension of its range is is also one, 
so that its index is zero.    Since $\sO$ is connected 
and the index is a topological invariant \cite[Theorem 4.51, p166]{AA1}, 
$\sL^0$ also has index zero.  
This means that the codimension of $\sL^0$ is also finite.
This completes the proof of Lemma \ref{index1}. 
\end{proof}  

Our next task is to verify (H4), which is a statement about the full nonlinear system.
\begin{lem}\lb{cpt1}
For each $j\in \mathbb N$, the set 
$K_j  =  \{(\lambda,\rho_i,R_e,V)\in {\sO_j};\ \sF(\lambda,\rho_i,R_e,V)=0\}$
is  compact in ${\mathbb R}\times X$.
\end{lem}
\begin{proof}
Let  $\{(\lambda_n,\rho_{in},R_{en},V_n)\}$ be any sequence in  $K_j$. 
It suffices to show that  it has a convergent subsequence
whose limit also belongs to $K_j$.  
By the assumed bound 
$|\lambda_n|+\|(\rho_{in},R_{en},V_n)\|_{X}\leq j$, 
there exists a subsequence, still denoted by $\{(\lambda_n,\rho_{in},R_{en},V_n)\}$,
and $(\lambda,\rho_i,R_e,V)$ such that
\begin{equation}\lb{converge1}
 \lambda_{n}  \to  \lambda  \text{ in }  \mathbb R,\quad
 \rho_{in}  \to  \rho_i  \text{ in }  C^0_{r}(\overline{\Omega}),\quad
 R_{en}  \to  R_{e}  \text{ in }  C^1_{r}(\overline{\Omega}),\quad
 V_{n}  \to  V  \text{ in }  C^2_{r}(\overline{\Omega}).
\end{equation}
Furthermore, 
\[
\inf_{r} | \partial_{r} (V+\lambda H)| =
\inf_{x} | \nabla (V+\lambda H)| \geq \tfrac1j.
\]
Since $\sO_j$ is closed in $X$,
it remains to show that 
\begin{gather*}
\sF_j(\lambda,\rho_i,R_e,V)=0 \qu \text{for $j=1,2,3,$}
\\
\rho_{in}\to \rho_i \ \text{in} \ C^1_{r}(\overline{\Omega}), \qu
R_{en}\to R_e \ \text{in} \ C^2_{r}(\overline{\Omega}), \qu
V_n\to V \ \text{in} \ C^3_{r}(\overline{\Omega}).
\end{gather*}

We may suppose that the cathode $\cC$ and anode $\cA$ are $\{r=r_{1}\}$ and $\{r=r_{2}\}$ for some $r_{2}>r_{1}>0$, respectively, since the other case can be treated similarly. 
We recall \eqref{dop1}.
The first equation $\sF_1(\lambda_n,\rho_{in},R_{en},V_n)=0$ with $\rho_{in}|_{\cC}=0$ is equivalent to
\[
 \rho_{in}(r)=\frac{k_e}{k_i}r^{1-d}\{ \partial_r (V_n+\lambda_n H)(r)\}^{-1}
\int_{r_{1}}^r s^{d-1} h(|\partial_r (V_n+\lambda_n H)(s)|)e^{-\frac{\lambda_n}{2}H(s)}R_{en}(s)\,ds.
\]
Taking the limit and using \er{converge1}, we see that 
\[
 \rho_{i}(r)=\frac{k_e}{k_i}r^{1-d}\{ \partial_r (V+\lambda H)(r)\}^{-1}
\int_{r_{1}}^r s^{d-1} h(|\partial_r (V+\lambda H)(s)|)e^{-\frac{\lambda}{2}H(s)}R_{e}(s)\,ds,
\]
where the right side converges in $C^1_{r}(\overline{\Omega})$.
Hence we see that $\sF_1(\lambda,\rho_i,R_e,V)=0$ and $\rho_{in}\to \rho_i$ in $C^1_{r}(\overline{\Omega})$.

Taking the limit and using \eqref{converge1} in the third equation $\sF_3(\lambda_n,\rho_{in},R_{en},V_n)=0$ immediately leads to 
\[
 r^{1-d} \partial_r (r^{d-1} \partial_r V)=\rho_i-e^{-\frac{\lambda}{2}H}R_{e}.
\]
Hence $\sF_3(\lambda,\rho_i,R_e,V)=0$ and $V_n\to V$ in $C^3_{r}(\overline{\Omega})$.

The second equation $\sF_2(\lambda_n,\rho_{in},R_{en},V_n)=0$ 
can be written as 
\[   
\partial_r \{r^{d-1} (\partial_r R_{en} - R_{en} \partial_r V_n)\} 
=r^{d-1}\left\{\frac{\lambda_{n}}{2} \partial_{r} V_{n} \partial_{r} H
+\frac{\lambda_{n}^2}{4}|\partial_{r} H|^{2}
-h\left(|\partial_{r} (V_{n}+\lambda H)|\right)\right\}R_{en}.
\]
Because the right side converges in $C^1_{r}(\overline{\Omega})$, we see that 
$\partial_r R_{en} - R_{en} \partial_r V_n$ converges in $C^2_{r}(\overline{\Omega})$. 
On the other hand, $R_{en} \partial_r V_n$ converges in $C^1_{r}(\overline{\Omega})$.  
Hence $\partial_r R_{en}$ converges in $C^1_{r}(\overline{\Omega})$, 
which means that $R_{en}$ converges in $C^2_{r}(\overline{\Omega})$. 
It  is now clear that $\sF_2(\lambda,\rho_i,R_e,V)=0$ holds.
\end{proof}

As we have checked all the conditions of Theorem \ref{Global0}, 
the first {\it global bifurcation theorem} is valid, as follows.
\begin{thm}\lb{Global1}
There exists in $\sO$ a continuous curve 
${\sK}=\{(\la(s),\rho_i(s),R_e(s),V(s));s \in \mathbb R\}$
of stationary solutions to system \er{sp0} such that
\begin{enumerate}[{(C}1{)}]
\item $(\la(0),\rho_i(0),R_e(0),V(0))=(\lambda^{*},0,0,0)$, where $\lambda^{*}$ is defined in \er{bp1};
\item $(\rho_i(s),R_e(s),V(s))=s(\varphi_{i}^{*},\varphi_{e}^{*},\varphi_{v}^{*})+o(s)$ in the space $X$
 as $s \to 0$, where $(\varphi_{i}^{*},\varphi_{e}^{*},\varphi_{v}^{*})$ are the basis in Lemma \ref{null1};
\item there exists a neighborhood $\sW$ of $(\lambda^{*},0,0,0)$
and $\ve < 1$ such that
\begin{multline*}
\{(\la,\rho_i,R_e,V)\in \sW ; (\rho_i,R_e,V)\neq (0,0,0), \ \sF(\la,\rho_i,R_e,V)=0 \}
\\
=\{(\la(s),\rho_i(s),R_e(s),V(s)) ; 0<|s|<\ve\};
\end{multline*}
\item $\sK$ has a real-analytic reparametrization locally around each of its points;
\item at least one of the following five alternatives occurs:
\begin{enumerate}[(a)]
\item $\varliminf_{s \to \infty}\la(s)= 0$;
\item $\varliminf_{s \to \infty}(\inf_{x \in \Omega } | \nabla ( V(x,s)+\la(s) H(x) )| )=0$;
\item $\varlimsup_{s \to \infty}\la(s)=\infty$;
\item $\varlimsup_{s \to \infty}(\|\rho_i\|_{C^1}+\|R_e\|_{C^2}+\|V\|_{C^3})(s)=\infty$;
\item there exists $T>0$ such that
$$(\la(s),\rho_i(s),R_e(s),V(s))=(\la(s+T),\rho_i(s+T),R_e(s+T),V(s+T))$$ for all $s \in \mathbb R$.
\end{enumerate}
\end{enumerate}
Moreover, the curve of solutions to problem \er{sp0} having the properties (C1)--(C5)
is unique (up to reparametrization).
\end{thm}

Conditions (C1)--(C3) are an expression of the local bifurcation, while (C4)--(C5) are assertions about the 
global curve $\sK$.  Alternatives (c) and (d) assert that $\sK$ may be unbounded.  
Alternative (e) asserts that $\sK$ might form a closed curve (a `loop').  

\subsection{Positive Densities}
For the physical problem  $\rho_i$ and $R_e$ are densities 
of particles and so they must be non-negative. 
In this section we investigate the part of the curve $\sK$ that corresponds to such densities. 

A simple observation is the following proposition, which states that {\it either} 
$\rho_i$ and $R_e$ remain positive {\it or} the curve of positive solutions 
forms a half-loop going from $\lambda^*$ to another voltage $\lambda^\#$.
Here $\lambda^*$ is defined in \er{bp1}, and $\lambda^\# \neq \lambda^*$ is some voltage such that $\kappa(\lambda^{\#})=0$.  
We call $\lambda^\#$ the {\it anti-sparking voltage}.
Actually,  if $a$ is large and $b$ is small in the definition of $g$, 
there exists such a voltage $\lambda^{\#}$ with $\kappa'(\lambda^{\#})<0$.
The bifurcation diagram of the half-loop case is sketched  in Figure \ref{fig3}.

\begin{pro}\lb{Global2}
For the global bifurcation curve $(\la(s),\rho_i(s),R_e(s),V(s))$ in Proposition \ref{Global1},
one of the following two alternatives occurs:
\begin{enumerate}[(i)] 
\item  $\rho_i(s,x)>0$ and $R_e(s,x)>0$ \  for all $0<s<\infty$ and $x \in \Omega$.
\item there exists a voltage $\lambda^\# \ne \lambda^*$ as well as  a finite parameter value  $s^\#>0$  such that 
\begin{enumerate}[(1)]
\item $\rho_i(s,x)>0$ and $R_e(s,x)>0$ \ for all $s\in(0,s^\#)$ and $x \in \Omega$;
\item $(\la(s^\#),\rho_i(s^\#),R_e(s^\#),V(s^\#))=(\lambda^\#,0,0,0)$.
\end{enumerate}
\end{enumerate}
\end{pro}

\begin{proof}
Let us suppose that the anode $\cA$ and cathode $\cC$ are $\{r=r_{1}\}$ and $\{r=r_{2}\}$, respectively, since the other case can be treated similarly. 
We define  
\begin{gather}\label{sstar}
s^\#:=\inf\{s>0 : R_e(s,r)=0 \ \text{for some $r=|x| \in (r_{1},r_{2})$}\}.
\end{gather}
We shall show that $s^\#$ satisfies (ii).
By the shape of $\phi_e(r)$, we find that $\rho_e(s,r)>0$ in $(r_1,r_2)$, 
making use of (C2) in Theorem \ref{Global1},  so that $s^\#>0$. 
If $s^\#=\infty$, then alternative (i) occurs.
Indeed, $| \nabla (V+\lambda H)|=\partial_{r}(V+\lambda H)$ is positive  owing to 
$(\la(s),\rho_i(s),R_e(s),V(s)) \in \sO$.   
The following formula yields the positivity of $\rho_i$.   
\begin{equation}\label{R_i1}
 \rho_{i}(r)=\frac{k_e}{k_i}r^{1-d}\{ \partial_r (V+\lambda H)(r)\}^{-1}
\int_{r_{1}}^r t^{d-1} h(|\partial_r (V+\lambda H)(t)|)e^{-\frac{\lambda}{2}H(t)}R_{e}(t)\,dt.
\end{equation}

In case $s^\#<\infty$, we must prove  (ii).    First consider $R_e(s^\#,\cdot)$.    
Certainly 
$R_e(s^\#,\cdot)$ takes the value zero, which is its minimum, 
at some point $r_0\in [r_{1},r_{2}]$.
In case $r_0\in (r_{1},r_{2})$, $\partial_r R_e(s^\#,r_0)=0$ also holds.
Solving the ODE $\sF_2(\lambda,\rho_i,R_e,V)=0$ with $R_e(s^\#,r_0)=\partial_r R_e(s^\#,r_0)=0$,
we see by uniqueness that $R_e(s^\#,\cdot)\equiv 0$. 
On the other hand, in case $r_0$ is one of the endpoints $\{r_{1},r_{2}\}$, by \er{sstar} there exists a sequence $\{(s_n,r_n)\}$ such that 
$R_e(s_n,r_n)=0$ with $s_n \searrow s^\#$ and $r_n \to r_{0}$.
Rolle's theorem ensures that there also exists some $\tilde{r}_n$ between $r_0$ and $r_n$ 
such that $\partial_r R_e(s_n,\tilde{r}_n)=0$.
Letting $n\to \infty$, we see that $\tilde{r}_n \to r_{0}$ and thus $\partial_r R_e(s^\#,r_{0})=0$.
Hence we again deduce  by uniqueness that $R_e\equiv 0$.  
Thus $R_e\equiv 0$ in any case.  
By \er{R_i1}, we also have $\rho_i\equiv 0$ and thus $V\equiv0$.
Hence $(\rho_i,R_e,V)(s^\#)=(0,0,0)$ is the trivial solution.
So (1) and (2) in the theorem are valid. 

Continuing to assume that $s^\#<\infty$, we now know that $\rho_i$, $R_e$ and $V$ are 
identically zero at $s=s^\#$.  
We define $\lambda^\#= \la(s^\#)$.  
Suppose that $\lambda^\#= \lambda^{*}$.  
Then the curve $\sK$ would go from the point $P=(\lambda^{*},0,0,0)$ at $s=0$ 
to the same point $P$ at $s={s^\#}$.   
However, by (C3) and (C4) of Proposition \ref{Global1},   
$\sK$ is a simple curve at $P$ and is real-analytic. 
So the only way $\sK$ could go from $P$ to $P$ would be if it were a loop with the part with 
$s$ approaching $s^\#$ from below coinciding 
with the part with $s$ approaching $0$ from below.  
By (C2) of Theorem \ref{Global1}, $\rho_i(s,\cdot)$ and $R_e(s,\cdot)$ would be negative 
for $-1\ll s-s^\#<0$, which of course contradicts their positivity.  
Hence $\lambda^\# \neq \lambda^{*}$.
\end{proof}

\begin{cor}
Let $a$ be sufficiently large and $b$ be sufficiently small in Proposition 5.6 (ii).  Then the following are true.  
\begin{enumerate}[(1)]
\setcounter{enumi}{2}
\item 
$(\rho_i(s),R_e(s))=(s^\#-s)(\varphi_i,\varphi_e)+o(|s-s^\#|)$
as $s\nearrow s^\#$,  where $\varphi_i$ and $\varphi_e$ are positive functions. 
\item 
$\rho_i(s,x)<0$ and $R_e(s,x)<0$ \ for $0<s-s^\#\ll1$ and $x \in \Omega$.
\end{enumerate}

\end{cor}
\begin{proof}
The bifurcation curve $\sK$ meets $(\lambda^\#,0,0,0)$.  
The nullspace $\sN^\# = N[\partial_{(\rho_i,R_e,V)} \sF(\la(s^\#),0,0,0)]$ must be non-trivial 
because $\sK$ crosses the line of trivial solutions.   
Consider a non-zero vector $(\varphi_i, \varphi_e,\varphi_v)$ that belongs to $\sN^\#$.   
If its second component  $\varphi_e$ were zero, then  $\varphi_i$ and consequently   $\varphi_v$ would also be zero.  Thus  $\varphi_e$ is non-zero and it belongs to the 
 nullspace of $A-\kappa(\lambda^\#)I$ where $A$ is defined in \eqref{operatorA}.     
 Thus $\varphi_e$ is a ground state (the lowest eigenfunction) of $A$, which is simple and positive.   
 All the higher eigenfunctions are orthogonal to $\varphi_e$ 
so that they must change sign within $\Omega$ and are therefore excluded.   
 So $\sN^\#$ is one-dimensional and $\varphi_i>0$ is given by \eqref{varphi_i}.  
 From this consideration and Lemma \ref{SparkingV1},  
the curve $\sK$ crosses $(\lambda^\#,0,0,0)$ transversely.  
Hence $R_e(s) = (s-s^\#)\varphi_e +o(|s-s^\#|)$ as $s\to s^\#$ and similarly for $\rho_i$.  
\end{proof}

Let us investigate in greater  detail the case 
that the global positivity alternative (i) in Proposition \ref{Global2} occurs.
The next three lemmas show that if any one of the alternatives (a), (b) or (c) in Theorem \ref{Global1} occurs, then alternative (d) also occurs.
To be specific hereafter, we suppose that the anode $\cA$ and cathode $\cC$  are $\{r=r_{1}\}$ and $\{r=r_{2}\}$, respectively.  Note that the condition $ |\nabla (V+\lambda H)| \geq \tfrac1j$ in the definition of $\cO_{j}$ can be written as $ \partial_{r} (V+\lambda H)  \geq \tfrac1j$.

\begin{lem}\lb{lem(a)}
Assume alternative (i) in Proposition \ref{Global2}.
If $\varliminf_{s \to \infty}\la(s)= 0$, then 
$\sup_{s>0}\|V(s)\|_{C^2}$ is unbounded.
\end{lem}

\begin{proof}
Suppose on the contrary that $\sup_{s>0}\|V(s)\|_{C^2}$ is bounded.
Because $\varliminf_{s \to \infty}\la(s)= 0$ and $\partial_r (V+\la H)(s,r)>0$,
there exists a sequence $\{s_n\}_{n \in \mathbb N}$ 
and a function $V^*(\cdot)$ such that 
\begin{gather}
\left\{
\begin{array}{lllll}
 \lambda(s_n) & \to & 0 & \text{in} & \mathbb R,
 \\
 V(s_n, \cdot) & \to & V^*(\cdot) & \text{in} & C^1_{r}(\overline{\Omega}),
 \end{array}\right.
\notag\\
V^*|_{\cC}=V^*|_{\cA}=0,
\lb{boundary2}\\
\partial_r V^*\geq 0.
\lb{p2}
\end{gather}
The boundary condition \er{boundary2} means that $\int_{r_{1}}^{r_{2}} \partial_r V^*(r)\,dr =0$.
This together with \er{p2} implies $\partial_r V^* \equiv 0$.
Using \er{boundary2} again, we have $V^* \equiv 0$.

It follows that
for suitably large $n$ the expressions $|\la(s_n)|$,
$\|V(s_n)\|_{C^1}$, and $\|h(|\nabla (V(s_n)+\la(s_n)H)|)\|_{C^0}$
are arbitrarily small.  
We multiply $\sF_2(\la(s_n),\rho_i(s_n),R_e(s_n),V(s_n))=0$ by $R_e(s_n)$
and integrate by parts over $\Omega$.
Then using Poincar\'e's inequality and taking $n$ suitably large, we obtain 
\begin{align*}
&\int_{\Omega} | \nabla R_e| ^2(s_n) \,dx
\\
&=-\int_{\Omega}
R_e(s_n) \nabla V(s_n)\cdot \nabla R_e(s_n) \,dx
\\
&\quad -\int_{\Omega} \left\{
\frac{\lambda(s_n)}{2}\nabla V(s_n) \cdot \nabla H
+\frac{\lambda^2(s_n)}{4}|\nabla H|^{2}
-h(|\nabla (V(s_n)+\la(s_n)H)|)\right\}R_e^2(s_n)\,dx
\\
&\leq \frac{1}{2}\int_{\Omega} |\nabla R_e|^2(s_n) \,dx.
\end{align*}
Hence $\partial_r R_e(s_n) \equiv0$.  
Since $R_e$ vanishes at the endpoints, we conclude that $R_e(s_n)\equiv 0$,   
which contradicts  the assumed positivity. 
\end{proof}

\begin{lem}\lb{lem(c)}
Assume alternative (i) in Proposition \ref{Global2}.
If $\varlimsup_{s \to \infty}\la(s)= \infty$, then 
$\sup_{s>0}\|V(s)\|_{C^2}$ is unbounded.
\end{lem}
\begin{proof}
Suppose that $\sup_{s>0}\|V(s)\|_{C^2}$ is bounded. 
We take a subsequence as above. 
For suitably large $n$, the expressions $\frac1{\la(s_n)}$
and $\|h(|\nabla (V(s_n)+\la(s_n)H)|)/\la^2(s_n)\|_{C^0}$
are arbitrarily small.  Write $s=s_n$ for brevity.  
Multiplying $\sF_2(\la(s),\rho_i(s),R_e(s),V(s))=0$ 
by $R_e(s)/\la^2(s)$ and then integrating by parts over $\Omega$, 
 we obtain
\begin{align*}
&\int_{\Omega} \frac{|\nabla R_e|^2(s)}{\la^2(s)}+\frac{1}{4} |\nabla H|^{2} R_e^2(s) \,dx
\\
&=\frac{1}{{2\lambda^2(s)}}\int_{\Omega} \left\{
\Delta V(s) - \lambda(s) \nabla V(s) \cdot \nabla H
+2h(|\nabla (V(s_n)+\la(s_n)H)|)\right\}R_e^2(s)\,dx
\\
&\leq \frac{1}{8}\int_{\Omega} |\nabla H|^{2} R_e^2(s) \,dx,
\end{align*}
since $\Delta V(s)$ is bounded.
Once again this leads to $R_e(s)\equiv 0$,  
which contradicts the assumed positivity.
\end{proof}

\begin{lem}\lb{lem(b)}
Assume alternative (i) in Proposition \ref{Global2}.
If 
\[  
\varliminf_{s \to \infty}(\inf_{x \in \Omega } | \nabla ( V(x,s)+\la(s)H(x))|) =0,  
\]
then $\sup_{s>0}\{\|\rho_i(s)\|_{C^0}+\|R_e(s)\|_{C^2}+\|V(s)\|_{C^2}\}$ is unbounded.
\end{lem}
\begin{proof}

Assume the contrary.  Thus $\sup_{s>0}\{\|\rho_i(s)\|_{C^0}+\|R_e(s)\|_{C^2}+\|V(s)\|_{C^2}\}$ is bounded.  
From Lemma \ref{lem(c)} and the hypothesis  
$\varliminf_{s \to \infty}(\inf_{r \in (r_{1},r_{2}) }   \partial_{r} (V+\lambda H)(r)  )=0$, we see that there exists a sequence $\{s_n\}$
and $(\la^*,\rho_i^*,R_e^*,V^*)$ with $\la^*<\infty $ such that
\begin{gather}
\left\{
\begin{array}{llllll}
 \lambda(s_n) & \to & \la^* & \text{in} & \mathbb R,
 \\
 \rho_{i}(s_n) & \rightharpoonup & \rho_i^* & \text{in} & L^{\infty}_{r}(\overline{\Omega}) & \text{weakly-star},
 \\
 R_{e}(s_n) & \to & R_{e}^* & \text{in} & C^1_{r}(\overline{\Omega}),
 \\
 \partial_r^2 R_{e}(s_n) & \rightharpoonup & \partial_r^2R_{e}^* & \text{in} & L^{\infty}_{r}(\overline{\Omega}) & \text{weakly-star},
 \\
 V(s_n) & \to & V^* & \text{in} & C^1_{r}(\overline{\Omega}),
 \\
 \partial_r^2 V(s_n) & \rightharpoonup & \partial_r^2V^* & \text{in} & L^{\infty}_{r}(\overline{\Omega}) & \text{weakly-star},
 \end{array}\right.
\lb{converge3}\\
R_{e}^*|_{\cC}=R_{e}^*|_{\cA}=V^*|_{\cC}=V^*|_{\cA}=0,
\lb{boundary1} \\
\rho_i^* \geq 0, \quad R_e^* \geq 0,
\lb{p1}\\
\inf_{r \in (r_{1},r_{2}) } \partial_{r} (V^{*}+\lambda^{*} H)(r) =0.
\lb{singular1}
\end{gather}
We shall  show that the limit satisfies 
\begin{equation*}
\sF_j(\lambda^*,\rho_i^*,R_e^*,V^*)=0 \quad \text{for a.e. $x$ and  $j=1,2,3$.}
\end{equation*}
Indeed, the equation $\sF_1(\lambda(s_n),\rho_{i}(s_n),R_{e}(s_n),V(s_n))=0$ 
with $\rho_{i}(s_n,\cdot)|_{\cA}=0$ is equivalent to
\begin{align*}
&r^{d-1} \{ \partial_r (V(s_n,r)+\lambda(s_n) H(r))\} \rho_{i}(s_{n},r)
\\
&=\frac{k_e}{k_i}\int_{r_{1}}^r t^{d-1} h( (\partial_rV(s_n,t)+\lambda(s_n) \partial_r H(t)) )e^{-\frac{\lambda(s_n)}{2}H(t)}R_{e}(s_n,t)\,dt.
\end{align*}
		Multiplying  
by a test function $\varphi \in C^0([r_{1},r_{2}])$ and integrating over $[r_{1},r_{2}]$, we obtain
\begin{align}
&\int_{r_{1}}^{r_{2}} r^{d-1} \{ \partial_r (V(s_n,r)+\lambda(s_n) H(r))\} \rho_{i}(s_{n},r) \varphi \,dr
\notag\\
&=\int_{r_{1}}^{r_{2}} \frac{k_e}{k_i} \left( \int_{r_{1}}^r t^{d-1} h(\partial_rV(s_n,t)+\lambda(s_n) \partial_r H(t))e^{-\frac{\lambda(s_n)}{2}H(t)}R_{e}(s_n,t)\,dt \right) \varphi  \,dr.
\lb{rho0}
\end{align}
		We note that 
\begin{align*}
&{}
\left|\int_{r_{1}}^{r_{2}} r^{d-1}\left[  \partial_r (V(s_n,r)+\lambda(s_n) H(r)) \rho_{i}(s_{n},r)
-\partial_r (V^{*}(r)+\lambda^{*} H(r))\} \rho_{i}^{*}(r)\right]\varphi \,dr\right|
\\
&\leq 
\left|\int_{r_{1}}^{r_{2}}r^{d-1} \left\{\partial_r (V(s_n,r)+\lambda(s_n) H(r))
-\partial_r (V^{*}(r)+\lambda^{*} H(r)) \right\} \rho_{i}(s_n,r) \varphi \,dr\right|
\\
&\quad
+\left|\int_{r_{1}}^{r_{2}} r^{d-1} (\rho_i(s_n,r) - \rho_i^*(r))\varphi {\partial_r (V^{*}(r)+\lambda^{*} H(r)) } \,dr\right|.
\end{align*}
		So passing to the limit $n\to\infty$ in \er{rho0} and making use of  \er{converge3}, we obtain 
\begin{align*}
&\int_{r_{1}}^{r_{2}} r^{d-1} \{ \partial_r (V^{*}(r)+\lambda^{*} H(r))\} \rho_{i}^{*}(r) \varphi  \,dr
\notag\\
&=\int_{r_{1}}^{r_{2}} \frac{k_e}{k_i} \left( \int_{r_{1}}^r t^{d-1} h(\partial_rV^{*}(t)+\lambda^{*} \partial_r H(t))e^{-\frac{\lambda^{*}}{2}H(t)}R_{e}^{*}(t)\,dt \right) \varphi  \,dr,   
 \quad   \forall\varphi \in C^0([r_{1},r_{2}]). 
 \end{align*} 
This immediately yields 
\begin{equation}\lb{rho1}
r^{d-1} \{ \partial_r (V^{*}+\lambda^{*} H)(r)\} \rho_{i}^{*}(r) 
=\frac{k_e}{k_i}
\int_{r_{1}}^r t^{d-1} h( \partial_rV^{*}(t)+\lambda^{*} \partial_r H(t))e^{-\frac{\lambda^{*}}{2}H(t)}R_{e}^{*}(t)\,dt \quad a.e.,
\end{equation}
which is equivalent to $\sF_1(\lambda^*,\rho_i^*,R_e^*,V^*)=0$ a.e.

We can write  $\sF_2(\lambda(s_n),\rho_{i}(s_n),R_{e}(s_n),V(s_n))=0$ and
$R_{e}(s_n,\cdot)|_{\cA}=R_{e}(s_n,\cdot)|_{\cC}=0$ weakly in the form  
\[
\int_\Omega \nabla R_{e}(s_n)\cdot \nabla \varphi \,dx
+\frac{(\lambda(s_n))^2}{4}\int_\Omega |\nabla H|^{2} R_{e}(s_n)\varphi \,dx
=-\int_0^L G_{2n}\varphi \,dx, \qu \text{ $\forall\varphi \in H^1_{0r}(\Omega)$},
\]
where 
\[
G_{2n}:= -\nabla V(s_n) \cdot \nabla R_{e}(s_n)
+\left\{\frac{\lambda(s_n)}{2} \nabla V(s_n) \cdot \nabla H
-\Delta V(s_n)
-h\left(|\nabla (V(s_n)+\lambda(s_n) H)|\right)\right\}R_e(s_n).
\]
Noting that 
\begin{align*}
&{}
\left|\int_\Omega \{\Delta V(s_n) R_{e}(s_n)- (\Delta V^*) R_{e}^*\}\varphi \,dx\right|
\\
&\leq 
\left|\int_\Omega \Delta V(s_n)(R_{e}(s_n)-R_{e}^*) \varphi \,dx\right|
+\left|\int_0^L (\Delta V(s_n) - \Delta V^*) {R_{e}^*\varphi} \,dx\right|, 
\end{align*}
taking the limit $n\to\infty$ in the weak form, and using \er{converge3}, we have
\[
\int_\Omega \nabla R_{e}^{*}\cdot \nabla \varphi \,dx
+\frac{(\lambda^{*})^2}{4}\int_\Omega |\nabla H|^{2} R_{e}^{*}\varphi \,dx
=-\int_0^L G_{2n}^{*}\varphi \,dx \qu \text{for $\forall\varphi \in H^1_{0r}(\Omega)$},
\]
where 
\[
G_{2}^{*}:= -\nabla V^{*} \cdot \nabla R_{e}^{*}
+\left\{\frac{\lambda^{*}}{2} \nabla V^{*} \cdot \nabla H
-\Delta V^{*}
-h\left(|\nabla (V^{*}+\lambda^{*} H)|\right)\right\}R_e^{*} \in L^{2}_{r}(\Omega).
\]
This means that $R_e \in H^1_{r}(\Omega)$ is a weak solution of $\sF_2=0$.
Furthermore, a standard theory of elliptic equations ensures that
$R_e \in H^2_{r}(\Omega)\cap H^1_{0r}(\Omega)$ is also a strong solution to $\sF_2(\lambda^*,\rho_i^*,R_e^*,V^*)$ $=0$.
Similarly it is easily seen that $\sF_3(\lambda^*,\rho_i^*,R_e^*,V^*)=0$.
We now set
\[
 r_*:=\inf\{r \in (r_{1},r_{2});\partial_{r} (V^{*}+\lambda^{*} H)(r) =0\}.
\]
We claim that $r_*=r_1$.  
On the contrary, suppose $r_* >r_{1}$. 
The equation \er{rho1}, which would hold for a sequence $r_\nu\to r_*$, yields the inequality  
\begin{gather*}
0=(r_{*})^{d-1} \partial_r (V^{*}+\lambda^{*} H)(r_{*}) \| \rho_{i} \|_{L^{\infty}}
\geq \frac{k_e}{k_i}
\int_{r_{1}}^{r_{*}} t^{d-1} h( \partial_r(V^{*}+\lambda^{*} H) )e^{-\frac{\lambda^{*}}{2}H}R_{e}^{*}\,dt.
\end{gather*}
Together with the positivity \er{p1} this would imply that 
$(h(\partial_r(V^{*}+\lambda^{*} H))e^{-\frac{\lambda^*}{2}H}R_{e}^*)(r)=0$ 
for $r \in [r_{1},r_{*}]$.
From the definition of $r_*$, 
we see that 
\begin{equation}\lb{regular1}
\partial_r(V^{*}+\lambda^{*} H)(r)>0 \quad \text{for $r\in [r_{1},r_*)$,}
\end{equation} 
so that $h(\partial_r(V^{*}+\lambda^{*} H))>0$ on $[r_{1},r_*)$.
Therefore, $R_e^*\equiv0$  in $[r_{1},r_*)$.
Hence from \er{rho1} and \er{regular1},  
$\rho_i^*=0$ a.e.  in $[r_{1},r_*)$.
Now from the equation $\sF_3(\lambda^*,\rho_i^*,R_e^*,V^*)=0$ and $\Delta H=0$,
we would see that $r^{d-1}\partial_r(V^{*}+\lambda^{*} H)$ is a constant in $[r_{1},r_*]$.  Thus 
$r^{d-1}\partial_r(V^{*}+\lambda^{*} H)=0$ in $[r_{1},r_*]$.
This contradicts the definition of $r_*$.
Therefore $r_* =r_{1}$.  

Let us now suppose that there exists $t_0>r_{1}$ for which 
$\partial_r(V^{*}+\lambda^{*} H)(t_0)>0$.   Then we define 
\[
 t^*:=\sup\{r<t_0 ; \partial_r(V^{*}+\lambda^{*} H)(r)=0 \}.
\]
Note that $t^* \in [r_{1},t_0)$ and $\partial_r(V^{*}+\lambda^{*} H)(t^{*})=0$.
On the other hand, we multiply the equation 
$\sF_1(\lambda^*,\rho_i^*,R_e^*,V^*)=0\ a.e.$  by $r^{d-1}$,
integrate the result over $[t^*,t]$ for any $t \in [t^*,t_0]$,
and use $\sF_3(\lambda^*,\rho_i^*,R_e^*,V^*)=0$ to obtain
\begin{align}
&t^{2d-2}\{\partial_r(V^{*}+\lambda^{*} H)\}\left( \frac{1}{t^{d-1}}\partial_{r} (t^{d-1} \partial_{r}  V^*)+e^{-\frac{\lambda^*}{2}H}R_{e}^*\right)
\notag\\
&= \frac{k_e}{k_i} t^{d-1}
\int_{t^{*}}^t t^{d-1} h( \partial_r(V^{*}+\lambda^{*} H))e^{-\frac{\lambda^{*}}{2}H}R_{e}^{*}\,d\tau
\quad \text{for a.e. $t \in [t^*,t_0]$}. 
\lb{es1}
\end{align}
By \er{p1}, \er{singular1}, and $\Delta H=0$, the left hand side is estimated from below as
\begin{align*}
&t^{2d-2}\{\partial_r(V^{*}+\lambda^{*} H)\}\left( \frac{1}{t^{d-1}}\partial_{r} (r^{d-1} \partial_{r}  V^*)+e^{-\frac{\lambda^*}{2}H}R_{e}^*\right)
\\
&\geq t^{d-1}\partial_r(V^{*}+\lambda^{*} H) \partial_{r} (r^{d-1} \partial_{r}  (V^*+\lambda^{*} H))
=\frac{1}{2}\partial_r \left[\left\{r^{d-1} \partial_{r}  (V^*+\lambda^{*} H)\right\}^2\right] \ a.e. 
\end{align*}
since $\partial_r V^*$ is absolutely continuous with respect to $r$.
The integrand on the right hand side of \eqref{es1} is estimated from above by 
$C e^{-b(\partial_r(V^{*}+\lambda^{*} H))^{-1}}$, due to the behavior of $h$.  
Consequently, substituting these expressions into \er{es1}, 
integrating the result over $[t^*,r]$, 
and using $\partial_r(V^{*}+\lambda^{*} H)(t^*)=0$, we have
\begin{equation}\lb{es2}
\left(\partial_r(V^{*}+\lambda^{*} H)\right)^2(r)
\leq C \int_{t^*}^r\int_{t^*}^t e^{-b(\partial_r(V^{*}+\lambda^{*} H))^{-1}} \,d\tau dt
\quad \text{for $r \in [t^*,t_0]$}. 
\end{equation}
Now we define $r_n$ as 
\[
 r_n:=\inf\left\{r\leq t_0;\ \ \partial_r(V^{*}+\lambda^{*} H)(r)= \frac{1}{n}\right\}.
\]
Notice that $t^* < r_n$ and $\partial_r(V^{*}+\lambda^{*} H)(r)\leq 1/n$ for any $r\in [t^*,r_n]$,
since the continuous function $\partial_r(V^{*}+\lambda^{*} H)$ vanishes at $r=t^*$.
Then we evaluate \er{es2} at $r=r_n$ to obtain
\[
\frac{1}{n^2}
\leq C \int_{t^*}^{r_{n}}\int_{t^*}^t e^{-b(\partial_r(V^{*}+\lambda^{*} H))^{-1}} \,d\tau dt
\leq C e^{-{bn}}.
\]
For suitably large $n$, this clearly does not hold.
So once again we have  a contradiction.

Our conclusion is that $r_*=r_{1}$ and $\partial_r(V^{*}+\lambda^{*} H)\equiv 0$ .
Hence  $\Delta V^*\equiv 0$  so that 
the equation  $\sF_2(\lambda^*,\rho_i^*,R_e^*,V^*)=0$ yields
$\Delta(e^{-\la^* H/2}R_e^*)=0$.
Solving this with \er{boundary1} yields $R_e^* \equiv 0$.
Then from $\sF_3(\lambda^*,\rho_i^*,R_e^*,V^*)=0$  
we obtain $\rho_i^* \equiv 0$.
Solving $\sF_3(\lambda^*,\rho_i^*,R_e^*,V^*)=0$ with \er{boundary1} 
and $\rho_i^* \equiv R_e^* \equiv 0$,
we also have $V^* \equiv 0$.
Consequently $\la^*=0$ holds and  $\varliminf_{s\to0}\la(s)=0$.
This contradicts Lemma \ref{lem(a)}, since
$\sup_{s>0}\|V(s)\|_{C^2}$ is bounded.
\end{proof}

Let us reduce Condition (d) in Theorem \ref{Global1} to a simpler condition.
We prefer to write the result directly in terms of the ion density $\rho_i$ and
the electron density $\rho_e=R_e e^{-\frac{\lambda}{2}H}$.

\begin{lem}\lb{lem(d)}
Assume the global positivity alternative (i) in Proposition \ref{Global2}. \  
If  the maximum of the densities $\sup_{s>0}\{\|\rho_i(s)\|_{C^0}+\|\rho_e(s)\|_{C^0}\}$ 
is bounded, then $\sup_{s>0}\{\|\rho_i(s)\|_{C^1}+\|R_e(s)\|_{C^2}+\|V(s)\|_{C^3}\}$ 
is also bounded.
\end{lem}
\begin{proof}
It is clear from $\sF_3=0$
together with the definition  $\rho_e=R_e e^{-\frac{\lambda}{2}H}$, that \begin{equation}\lb{bound_v1}
\sup_{s>0}\|V(s)\|_{C^2}
\leq C \sup_{s>0}\{\|\rho_i(s)\|_{C^0}+\|\rho_e(s)\|_{C^0}\}
<+\infty. 
\end{equation}
Then  from Lemma \ref{lem(c)} we have 
\begin{equation}\lb{bound_la1}
\sup_{s>0}|\la(s)|<+\infty.
\end{equation}
From \er{bound_la1} and the boundness of $\|\rho_e(s)\|_{C^0}$, 
we also have 
\begin{equation}\lb{bound_R1}
\sup_{s>0}\|R_e(s)\|_{C^0}<+\infty.  
\end{equation}
Applying a standard estimate of one-dimensional elliptic equations to $\sF_2=0$
with \er{bound_v1}--\er{bound_R1}, 
we  infer that $\sup_{s>0}\|R_e(s)\|_{C^2}<+\infty$.
Now  Lemma \ref{lem(b)} implies that  
$\varliminf_{s \to \infty}(\inf_{r \in (r_{1},r_{2}) }  \partial_{r} ( V(r,s)+\la(s)H(r))) > 0$.
Together with  \er{R_i1}, this leads to $\sup_{s>0}\|\rho_i(s)\|_{C^1}<+\infty$.
Finally the finiteness of $\sup_{s>0}\|V(s)\|_{C^3}$
follows from $\sF_3(\lambda(s),\rho_i(s),R_e(s),V(s))=0$.
\end{proof}

We conclude with the following {\it main result}, essentially Theorem \ref{thm4}, 
which asserts that {\it either} $\rho_i$ and $\rho_e$ are positive along all of $\sK$ 
with $0<s<\infty$ and 
$\rho_i + \rho_e$ is unbounded 
{\it or else} there is a half-loop of positive solutions going from $\lambda^{*}$ to $\lambda^{\#}$. 

\begin{thm} \label{mainthm}
One of the following two alternatives occurs:
\begin{enumerate}[(A)]
\item 
For all $s\in(0,\infty)$, both $\rho_i(s,x)$ and $\rho_e(s,x)$  
are positive  for all $x \in \Omega$.   Furthermore, 
\[
\varlimsup_{s\to\infty}\{\|\rho_i(s)\|_{C^0}+\|\rho_e(s)\|_{C^0}\}=\infty.
\]
\item 
There exists $s^\#>0$ such that the following conditions  hold: 
\begin{enumerate}[(1)]
\item Both $\rho_i(s,x)$ and $\rho_e(s,x)$ are positive for any $s\in(0,s^\#)$ and $x \in \Omega$;
\item $(\la(s^\#),\rho_i(s^\#),\rho_e(s^\#),V(s^\#))=(\lambda^\#,0,0,0)$.
\end{enumerate}
\end{enumerate}
\end{thm}

\begin{proof}  (B) is implied by the second alternative (ii) in Proposition \ref{Global2}.  
So let us suppose that (B) does not hold.
Then the first alternative (i) in  Proposition \ref{Global2} must hold.  
Now in Proposition \ref{Global1} there are five alternatives.  Alternative (e) cannot happen because 
$\rho_i$ and $R_e$ are negative on part of the loop. 
Lemmas \ref{lem(a)}--\ref{lem(b)} assert that any one of (a) or (b) or (c) implies that 
$\sup_{s>0}\{\|\rho_i(s)\|_{C^0}+\|R_e(s)\|_{C^2}+\|V(s)\|_{C^2}\}$ is unbounded.  
Then Lemma \ref{lem(d)} leads to the unboundedness of 
$\sup_{s>0}\{\|\rho_i(s)\|_{C^0} + \|\rho_e(s)\|_{C^0}\}$.   
This means that (A) holds. 
\end{proof}

\bigskip

\noindent
{\bf Acknowledgement:} 
The authors would like to thank the anonymous reviewers for their valuable comments and suggestions.

\bigskip

\noindent
{\bf DECLARATIONS:} 

{\bf Funding and/or Conflict of Interest / Competing Interests:

The work of M. Suzuki was supported by JSPS KAKENHI Number 21K03308.

No additional funding was received by either author.  

There are no competing or conflicting interests to declare. }

\end{document}